\documentclass[11pt]{amsart}
\usepackage{hyperref}
\hypersetup{linktocpage = true, colorlinks = true, linkcolor = blue, citecolor= red, urlcolor = green}

\usepackage{tikz}
\usetikzlibrary{positioning, quotes, arrows, decorations.markings, decorations.pathmorphing, decorations.pathreplacing, shapes, patterns, calc}
\usepackage{url}

\usepackage{graphicx, color, bm}
\usepackage{array,tikz,tikz-cd,float}

\usepackage{amsmath,amsthm}
\usepackage{amssymb, amsfonts, verbatim, subfigure}

\usepackage[bbgreekl]{mathbbol}
\usepackage{mathpazo}
\usepackage[mathscr,mathcal]{euscript}

\tikzset{middlearrow/.style={
		decoration={markings,
			mark= at position 0.5 with {\arrow{#1}} ,
		},
		postaction={decorate}
	}
}

\newcommand{\pd}{{\partial}}

\newcommand{\ol}{\overline}
\newcommand{\wt}{\widetilde}
\newcommand{\id}{\mathbb{1}}

\newcommand{\be}{\begin{equation}}
	\newcommand{\ee}{\end{equation}}
\newcommand{\beqn}{\begin{equation}}
	\newcommand{\eeqn}{\end{equation}}
\newcommand{\bp}{\begin{pmatrix}}
	\newcommand{\ep}{\end{pmatrix}}
\newcommand{\bsp}{\left(\begin{smallmatrix}}
	\newcommand{\esp}{\end{smallmatrix}\right)}

\newcommand*\diff{\mathop{}\!\mathrm{d}}

\newcommand{\R}{{\mathbf R}}
\renewcommand{\P}{{\mathbb P}}

\newcommand{\C}{{\mathbf C}}
\newcommand{\Z}{{\mathbf Z}}

\newcommand{\CA}{{\mathcal A}}

\newcommand{\CF}{{\mathcal F}}

\newcommand{\CH}{{\mathcal H}}
\newcommand{\CI}{{\mathcal I}}
\newcommand{\CK}{{\mathcal K}}

\newcommand{\CN}{{\mathcal N}}
\newcommand{\CO}{{\mathcal O}}

\newcommand{\CR}{{\mathcal R}}

\newcommand{\CV}{{\mathcal V}}
\newcommand{\CW}{{\mathcal W}}

\newcommand{\CL}{{\mathcal L}}

\newcommand{\T}{\text{T}}

\newcommand{\fgl}{\mathfrak{gl}}
\newcommand{\fg}{\mathfrak{g}}

\newcommand{\im}{\mathrm{i}}

\newcommand{\norm}[1]{{{:\!{#1}\!:}}}

\numberwithin{equation}{section}
\numberwithin{figure}{section}
\numberwithin{table}{section}

\newcommand{\lie}{\mathfrak}

\newcommand{\bu}{\bullet}
\newcommand{\del}{\partial}
\newcommand{\cR}{\mathcal{R}}
\newcommand{\op}{\operatorname}

\renewcommand\Hat{\widehat}
\renewcommand{\d}{\mathrm{d}}

\newcommand{\cA}{\mathcal{A}}
\newcommand{\define}{\overset{\text{def}}{=}}

\newcommand{\dbar}{\overline{\del}}

\newcommand{\rav}{\mathbb{R}\text{av}}
\renewcommand{\Bar}{\overline}
\newcommand{\zbar}{\Bar{z}}
\newcommand{\wbar}{\Bar{w}}
\newcommand{\ul}{\underline}
\newcommand{\three}{\CK_{z,w,z-w}}

\DeclareMathOperator{\End}{End}
\DeclareMathOperator{\Sym}{Sym}
\DeclareMathOperator{\Hom}{Hom}
\DeclareMathOperator{\Spec}{Spec}

\DeclareMathOperator{\Vect}{\text{Vect}}

\DeclareMathOperator{\Res}{Res}

\usepackage{todonotes}

\newtheoremstyle{thm}
{7pt}
{7pt}
{\itshape}
{}
{\bf}
{.}
{5pt}
{\thmnumber{#2 }\thmname{#1}\thmnote{ (#3)}}

\newtheoremstyle{def}
{7pt}
{10pt}
{\itshape}
{}
{\bf}
{.}
{5pt}
{\thmnumber{#2} \thmname{#1}\thmnote{ (#3)}}

\newtheoremstyle{rem}
{4pt}
{10pt}
{}
{}
{\itshape}
{:}
{3pt}
{}

\newtheoremstyle{texttheorem}
{8pt}
{8pt}
{\itshape}
{}
{\bf}
{. \hspace{5pt}}
{3pt}
{}

\theoremstyle{thm}

\newtheorem*{theorem*}{Theorem}
\newtheorem*{lemma*}{Lemma}
\newtheorem*{corollary*}{Corollary}
\newtheorem*{proposition*}{Proposition}
\newtheorem*{definition*}{Definition}

\newtheorem{theorem}{Theorem}[subsection]
\newtheorem{thm-def}{Theorem/Definition}[theorem]
\newtheorem{prop}[theorem]{Proposition}

\newtheorem*{question*}{Question}
\newtheorem{lemma}[theorem]{Lemma}

\newtheorem{corollary}[theorem]{Corollary}

\numberwithin{equation}{subsection}

\theoremstyle{definition}
\newtheorem{dfn}[theorem]{Definition}

\theoremstyle{rem}

\newcommand{\defterm}[1]{\textbf{\emph{#1}}}

\usepackage{stmaryrd}
\date{}

\title{Raviolo vertex algebras}
\author{Niklas Garner}
\author{Brian R. Williams}

\begin{document}
	
\begin{abstract}
	We develop an algebraic structure modeling local operators in a three-dimensional quantum field theory which is partially holomorphic and partially topological. 
	The geometric space organizing our algebraic structure is called the raviolo (or bubble) and replaces the punctured disk underlying vertex algebras; we refer to this structure as a raviolo vertex algebra.
	The raviolo has appeared in many contexts related to three-dimensional supersymmetric gauge theory, especially in work on the affine Grassmannian.
	We prove a number of structure theorems for raviolo vertex algebras and provide simple examples that share many similarities with their vertex algebra counterparts.
\end{abstract}

\maketitle
\tableofcontents

Our goal is to develop the algebraic structure underpinning the local observables of a particular class of three-dimensional quantum field theories that has received considerable attention in recent years, particularly in the context of supersymmetric gauge theory.
It is well-known that the local observables of a two-dimensional chiral conformal field theory organize into the structure of a vertex algebra.
Since its inception, the theory of vertex algebras has led to remarkable interactions between physics and mathematics of which quantum field theory and representation theory have significantly profited and we are optimistic that an analogous role will be played by this higher dimensional version of a vertex algebra.

In two dimensions, a chiral conformal field theory is one which depends on the complex structure of the underlying surface.
The sort of three-dimensional theories whose local operators we aim to model depend on the data of a particular complex foliation which equips the three-manifold with local coordinates of the form $(z,t)$ where $z$ is a holomorphic coordinate and $t$ is a real smooth coordinate.
Such a foliation is called a \textit{transverse holomorphic foliation} (THF) \cite{DuchampKalka,Rawnsley,Asuke}; we recall its precise definition in the main text below.

Typical examples of three-manifolds equipped a THF include products $\Sigma \times S$ where $\Sigma$ is a Riemann surface and $S$ is a real one-dimensional manifold.
More generally, however, THFs are not necessarily split in this form.
The geometric incarnation of the central algebraic object in this note is an example of a three-manifold which is equipped with such a non-split THF structure.
Simply, it is the THF structure on punctured affine space $\C \times \R - \{0\}$ inherited from the tautological one on $\C \times \R$.
Of course, this punctured affine space is homotopically equivalent to the two-sphere $S^2$ but geometrically the data of the foliation allows us to consider an analog of the algebra of `holomorphic' functions on this space--these are functions which are constant along the leaves of the foliation.

Holomorphic functions on the punctured affine line $\C^\times = \C-\{0\}$, or more precisely functions on the formal disk $D^\times$, conveniently organize the modes of a vertex algebra.
There is a formal version of the THF structure on punctured affine space $\C \times \R - \{0\}$ which we will refer to as the \textit{formal raviolo} and denote it $\rav$.
As a non-separated scheme this is given by gluing two formal disks over a shared punctured disk $\rav = D \cup_{D^\times} D$.
The space $\C^\times$ plays a fundamental role in two-dimensional chiral conformal field theory since it is the configuration space of two points in $\C$, modulo overall translations, therefore two-point correlation functions, which control the operator product expansions (OPEs) in a vertex algebra, belong to holomorphic functions on $\C^\times$.
Analogously, in a three-dimensional theory which is holomorphic-topological, functions which are constant along the leaves of the THF on $\C \times \R - \{0\}$ control the raviolo analog of the OPE.

An important subtlety here is that the formal raviolo is not an affine scheme, therefore instead of considering just the algebra of functions on $\rav$ we consider the derived sections of its structure sheaf.
In Section \ref{sec:raviolo} we construct an explicit `de Rham' model for this dg algebra using the THF structure on punctured affine space.
This model is similar to the one used in \cite{FHK} for the derived global sections of the higher dimensional formal punctured disk motivated by the Jouanolou torsor.
A potential drawback of our approach is that in the bulk of the paper we work at the level of cohomology of the raviolo rather than at the cochain level; indeed a \textit{raviolo field} is an operator-valued function parameterized by the cohomology of the formal raviolo.
We expect that a careful analysis of the cochain level version of our algebraic structures could lead to a refinement of our constructions.

From this point, our definition of a \textit{raviolo vertex algebra} is parallel to that of an ordinary vertex algebra with the role of the formal disk being replaced by the formal raviolo $\rav$.
In Section \ref{sec:definition} we give the definition of a raviolo vertex algebra.
Most importantly, we develop the notion of mutual locality between two raviolo fields in terms of distributional calculus on two formal ravioli.

In Section \ref{sec:elementary} we discuss basic properties of raviolo vertex algebras.
The first main result (Theorem \ref{thm:assoc}) of this section is a formulation of associativity for raviolo vertex algebras and the corresponding OPE.
Physically, this OPE encapsulates information about the actual operator product of the holomorphic-topological QFT, which is necessarily regular, as well as the possibly singular operator products of holomorphic-topological descendants.
The second main result (Theorem \ref{thm:RVAvsshiftedPVA}) in this section is an equivalence between raviolo vertex algebras and one-shifted Poisson vertex algebras \cite{OhYagi, CostelloDimofteGaiotto-boundary}, or Gerstenhaber vertex algebras \cite{Bouaziz}.
The work \cite{OhYagi} propose one-shifted Possion vertex algebras as a model for the algebra of local operators in twisted three-dimensional supersymmetric quantum field theories, see also \cite{Zeng}. 
The chiral algebraic analog of this structure appeared in \cite{Tamarkin} under the name pro-$c$-Gerstenhaber
algebra. In that context, it is argued in loc. cit. that this algebraic structure is present in the derived center of a chiral algebra (whose cohomology controls deformations of the chiral algebra); we expect an analogous statement to hold for vertex algebras.
In the context of three-dimensional quantum field theory, it is expected that the derived center of a boundary vertex algebra should be related to local operators in the three-dimensional bulk~\cite{CostelloDimofteGaiotto-boundary}, see also \cite{Zeng}.
It is satisfying that our bottom-up approach to holomorphic-topological theories fits this expectation and agrees with established prescriptions for describing observables in such theories.

The local operators of an $n$-dimensional topological field theory always form an algebra over the operad of little disks in $\R^n$.
In \cite{CG1} it is shown how special algebras over a holomorphic analog of the operad of little disks in $\C$ can be used to recover vertex algebras.
Such algebras can be constructed from translation invariant factorization algebras on $\C$ for which the infinitesimal translation $\del / \del \zbar$ act (homotopically) trivially.
We anticipate a similar result for ravioli vertex algebras: there is a colored operad of disks in $\R \times \C$ which can be enriched to the THF setting.
Furthermore, we expect that translation invariant factorization algebras on $\R \times \C$ for which both $\del / \del t$ and $\del/\del \zbar$ act (homotopically) trivially give rise to raviolo vertex algebras.
The interpretation of raviolo vertex algebras in terms of factorization algebras points towards a global theory, such as a raviolo version of conformal blocks which should exist on general THF three-manifolds.

With the goal of constructing examples, in Section \ref{sec:examples} we state and prove a reconstruction theorem for raviolo vertex algebras (Proposition \ref{prop:reconstruction}).
This is directly analogous to the one for vertex algebras (see, for example, \cite{FBZ}).
Using this we provide a handful of universal examples of raviolo vertex algebras stressing parallels with the usual free field, affine Kac--Moody, Heisenberg, and Virasoro vertex algebras.
Section \ref{sec:examples} ends with a discussion of a class of deformations relevant to three-dimensional gauge theory, for example leading to the raviolo analog of BRST reduction of a vertex algebra.

Finally, in Section \ref{sec:modules} we initiate the study of modules for raviolo vertex algebras.
For raviolo vertex algebras induced from Lie algebras, such as the raviolo analogs of free field, Heisenberg, affine Kac--Moody, and the Virasoro algebras, we relate raviolo vertex algebra modules with smooth modules for the associated Lie algebra.
We end by considering Fock modules for the raviolo Heisenberg algebra and use these to construct a lattice-type raviolo vertex algebra which models non-perturbative local operators in twisted three-dimensional $\mathcal{N}=2$ abelian Yang--Mills theory.
In \cite{Zeng} motivated by a bulk/boundary (3d/2d) correspondence, many other non-perturbative examples of shifted Poisson vertex algebras were considered, and it would be very interesting to understand these within the context of raviolo vertex algebras.

\subsection*{Acknowledgements}

We thank Owen Gwilliam, Zhengping Gui, Surya Raghavendran, and Keyou Zeng for conversation and inspiration of all kinds related to the formation and development of the ideas in this paper.
We are especially grateful to Surya Raghavendran for his collaboration on work pertaining to symmetry enhancement for twists of three-dimensional supersymmetric gauge theories which partly inspired this project.

\section{The raviolo}
\label{sec:raviolo}

In this section we construct a model of the formal raviolo (sometimes referred to as the bubble, see \cite[Section 7]{Kamnitzer}).
Algebro-geometrically, the formal raviolo is given as the scheme theoretic intersection of two formal disks relative to a formal punctured disk
\beqn
\rav = D \cup_{D^\times} D .
\eeqn
Topologically, the formal bubble is equivalent to the two-sphere, but scheme theoretically it differs from $\C \P^1$.
Indeed, the raviolo is non-separated, but it can be thought of as an infinitesimal version of $\C \P^1$ where the transition function identifies $z_1 = z_2$ rather than $z_1 = z_2^{-1}$.

The raviolo $\rav$ shares a close relationship to the affine Grassmannian.
If $G$ is a Lie group then the moduli of $G$-bundles on $\rav$ is equivalent to the space of pairs of $G$-bundles $P,P'$ on the formal disk which are isomorphic when restricted to the formal punctured disk.
The affine Grassmannian $Gr_G$ is naturally a $\text{Map}(D, G)$-torsor over~$\text{Bun}_G(\rav)$.

The raviolo is not affine in the usual sense, indeed $H^1(D \cup_{D^\times} D, \CO) \ne 0$, but it is affine in a derived sense.
In analogy with the formal punctured disk $D^\times = \Spec \C(\!(z)\!)$ the formal raviolo is the spectrum of a certain commutative dg algebra $\CA = (\CA^\bu, \d')$.
We construct a model for this dg algebra using the geometric perspective of THF structures.

\subsection{Transverse holomorphic foliations}

A transversely holomorphic foliation (THF) on a smooth manifold $M$ is a smooth foliation $\CF$ of even codimension whose leaf space is equipped with the structure of a complex manifold \cite{DuchampKalka,Asuke,Rawnsley}.
In other words, there exists a foliation atlas whose transition functions are biholomorphisms.
In this paper we will only be concerned with three-manifolds $M$ so that the leaf space of $\CF$ is simply a Riemann surface.
The product $\Sigma \times S$, where $\Sigma$ is a complex manifold and $S$ is a smooth manifold is equipped with a natural foliation of this type. 
In this case, $\CF$ is the restriction of the tangent bundle of $S$ along the obvious projection.
Locally, any THF three-manifold is split of the form $\C \times \R$, whose coordinates we will denote by $(z , t)$ where $z$ is a holomorphic coordinate along $\C$ and $t$ is a real smooth coordinate so that locally the tangent space to $\CF$ is spanned by $\del/\del t$.

Let $Q$ be the (real) quotient bundle $\T_M / \T_\CF$.
Since $\CF$ is a THF, the bundle $Q$ admits an almost complex structure so that the complexification admits a canonical decomposition
\beqn
Q \otimes \C = Q^{1,0} \oplus Q^{0,1}
\eeqn
which satisfies $\Bar{Q^{1,0}} = Q^{0,1}$.
Locally, $Q^{1,0}$ is spanned by $\del/\del z$ and $Q^{0,1}$ by $\del/\del \zbar$.
The canonical projection onto $Q^{1,0}$ determines a subbundle of the complexified tangent bundle
\beqn
V_{\CF} \define \ker (\T_M^\C \to Q^{1,0}) .
\eeqn
Conversely, if $V \subset \T_M \otimes \C$ is an involutive complex subbundle such that $\T_M \otimes \C = V \oplus \Bar{V}$ then $M$ admits a THF with $Q^{1,0} = V / (V \cap \Bar{V})$ \cite{DuchampKalka}.
Locally, $V$ is spanned by $\del/\del \zbar, \del / \del t$.

In the context of foliations, the natural replacement for the notion of a holomorphic function defined on an open set $U$ is a function which is constant along the leaves of the foliation.
In the notation above, this is a function $f \in C^\infty(U)$ such that $L_X f = 0$ for all $X \in \Gamma(U,V_\CF)$.
We denote the algebra of such functions $\CO_{\CF}(U)$; note that the assignment $U \mapsto \CO_{\CF}(U)$ defines a sheaf on $M$.


Given any foliation, there is a cohomology theory associated to it. 
In terms of de Rham cohomology, this is built from the so-called $\CF$-basic de Rham forms. 
An explicit model for THF structures is as follows.
For each $p,q$ denote by $\CA^{(p);q}$ smooth sections of the bundle $\wedge^p (Q^{1,0})^\vee \otimes \wedge^q V_\CF^\vee$. 
The derivative along the leaves of the foliation defined by $V$ defines a map 
\[
\d' \colon \CA^{(p);q} \to \CA^{(p);q+1}  .
\]
By integrability one has $\d' \circ \d' = 0$ and so $\d'$ equips $\CA^{(p);\bu} = \oplus_q \CA^{(p);q}[-q]$ with the structure of a cochain complex for each $p$. 
Notice that $\CA^{(p);\bu}$ is concentrated in degrees 0, 1, and 2 for all $p$.
The obvious exterior product $\CA^{(p);q} \times \CA^{(r);s} \to \CA^{(p+r);q+s}$ further endows 
\[
\left(\CA^{(\bu);\bu}, \d' \right) = \left(\oplus_p \CA^{(p);\bu}[-p] , \d' \right) 
\]
with the structure of a bigraded commutative dg algebra.
In particular $(\CA^{(0);\bu},\d')$ is a commutative dg algebra with wedge product being the multiplication.

In degree zero $\CA^{(0);0}$ consists just of smooth (complex valued) functions. 
Locally, a section of $\CA^{(0);1}$ is of the form
\[
f_{\zbar} (z,t) \d \zbar + f_t (z,t) \d t
\]
where $f_{\zbar},f_t$ are smooth functions.
Similarly, a section of $\CA^{(0);2}$ is of the form $f_{\zbar t}(z,t) \d \zbar \d t$.
The differential $\d'$ acting on a function $f$ is 
\[
\d' f = \frac{\del f}{\del \zbar} \d \zbar + \frac{\del f}{\del t} \d t .
\] 
Thus locally $\d'$ is simply a sum of the Dolbeault operator $\dbar$ on $\C$ and the de Rham operator $\d$ on $\R$.
Similar local formulae hold for $\CA^{(p);\bu}$.

The \defterm{Dolbeault--de Rham cohomology} of $M$ with coefficients in $\wedge^p Q^\vee$ is
\[
H^{(p);\bu}_{\d'} (M) \define H^{\bu} \left(\CA^{(p);\bu}(M), \d' \right) .
\]
We refer to $(\CA^{(p);\bu}(M), \d')$ as the Dolbeault--de Rham complex of $\wedge^p Q^\vee$.
We will mostly be concerned with the case $p=0$.
In this case, $H^{(0);\bu}_{\d'} (M)$ can naturally be identified with the algebra $\CO_\CF(M)$ of functions which are constant along the leaves of the foliation.
In fact, it is shown in \cite{Rawnsley} that $\CA^{(0);\bu}$ is a de Rham-type model for the sheaf cohomology of $\CO_{\CF}$ in the sense that there is an isomorphism
\beqn
H^\bu(M,\CO_{\CF}) \simeq H^{(0);\bu}_{\d'} (M)
\eeqn
where on the left hand side we mean sheaf cohomology.
A similar result holds for $\CA^{(p);\bu}$ where the sheaf on the left hand side is replaced by the sheaf of $V_\CF$-constant sections of $\wedge^p (Q^{1,0})^\vee$.

As a split example consider the case $M = \Sigma \times S$ where $\Sigma$ is a Riemann surface and $S$ is a one-manifold.
Then there is an isomorphism of cochain complexes
\[
\left(\CA^{(p);\bu}(M), \d'\right) \simeq \left(\Omega^{p,\bu}(\Sigma) \Hat{\otimes} \Omega^\bu(S) \, , \, \dbar \otimes \id + \id \otimes \d\right) 
\]
where $\Omega^{p,\bu}$ is the Dolbeault complex and $\Omega^\bu$ is the de Rham complex.
It follows that in cohomology one has an isomorphism
\beqn
H^{(p);\bu}_{\d'}(M) \simeq H^{p,\bu}(\Sigma) \otimes H^\bu(S)
\eeqn
for $p=0,1$.

The restriction of a THF to any open submanifold is again a THF.
In particular, the manifold 
\beqn
\C \times \R - \{0\}
\eeqn
is equipped with a natural THF coming from the obvious split one on $\C \times \R$.
We will need the following result characterizing the $\d'$-cohomology of punctured space.

\begin{theorem}\label{thm:cohomology}
For $p=0,1$ the cohomology $H^{(p);\bu}_{\d'}(\C \times \R - \{0\})$ is concentrated in degrees zero and one.
In degree zero, there is an isomorphism
\beqn
H^{(p);0}_{\d'} (\C \times \R - \{0\}) \simeq \Omega^{p,hol}(\C)
\eeqn
induced by pulling back holomorphic $p$-forms along the composition
\beqn
\C \times \R - \{0\} \hookrightarrow \C \times \R \to \C . 
\eeqn 
In degree one, there is an isomorphism
\beqn
H^{(p);1}_{\d'} (\C \times \R - \{0\}) \simeq \left(\Omega^{1-p,hol}(\C)\right)^\vee,
\eeqn
where the dual is the topological dual.
\end{theorem}

Before proving this result, we will need to develop an analog of the residue in complex geometry.
Denote the coordinates on $\C \times \R$ adapted to the standard THF by 
\beqn
\ul{z} = (z,t) .
\eeqn
For $\ul{z} = (z,t), \ul{w} = (w,s)$ let 
\beqn
r(\ul{z},\ul{w})^2 = |z-w|^2 + (t-s)^2 
\eeqn
be the squared distance between $\ul{z}$ and $\ul{w}$.
Viewing $\ul{w}$ as fixed, define the $(0;1)$-form $\omega(\ul{z};\ul{w})$ on $\C \times \R - \{\ul{w}\}$ by the formula
\beqn\label{eqn:omegas2}
\omega(\ul{z};\ul{w}) \define \frac{(t-s) \d \zbar - 2(\zbar - \wbar) \d t}{r(\ul{z},\ul{w})^3} .
\eeqn
It is immediate to check that 
\beqn
	\d' \omega(\ul{z};\ul{w}) = 0 .
\eeqn
It turns out that this $(0;1)$ form is \textit{not} $\d'$-exact, so it represents a cohomology class.
Indeed, consider the one-form $\d z$; this is a $\d'$-closed section of $(Q^{1,0})^\vee$ so that $\d z \omega(\ul{z};\ul{w}) \in \CA^{(1);1} (\C \times \R - \{0\})$.
This class represents the volume form of a two-sphere centered at $\ul{w}$.

\begin{lemma}\label{lem:residue}
One has
\beqn
	\oint_{r(\ul{z},\ul{w}) = r}\d z\, \omega(\ul{z};\ul{w}) = 8 \pi \im,
\eeqn
independent of the radius $r>0$.
\end{lemma}
\begin{proof}
Suppose that $\ul{\xi} = (\xi, u)$ is a new variable on $\C \times \R$ such that $z = w + r \xi, t = s + r u$.
Then $\omega(\ul{z};\ul{w}) = \omega(r \ul{\xi};0)$ so that by Stokes' theorem we have
\beqn
	\oint_{r(\ul{z},\ul{w}) = r}\d z\, \omega(\ul{z};\ul{w}) = \oint_{r(\ul{\xi},0)=1} \d \xi \left(u \d \Bar{\xi} - 2 \Bar{\xi}\d u \right) = 3 \int_{B(0,1)} \d u \d \xi \d \Bar{\xi} = 8\pi \im .
\eeqn
In the last line we have used the identity $\d \xi \d \Bar{\xi} \d u = 2 \im \; vol_{\R^3}$, where $vol_{\R^3}$ is the standard volume element on $\R^3$ and that the volume of the unit ball is $vol(B(0,1)) = 4 \pi / 3$.
\end{proof}

Essentially by Stokes' formula again, we have the following analog of the Cauchy residue formula.
If $f$ is a smooth function defined on a ball $\{\ul{z} \; | \; r(\ul{z},\ul{w}) \leq r\}$ then
\beqn
	f(\ul{w}) = \oint_{r(\ul{z},\ul{w}) = r} \frac{\d z}{8\pi \im} f(\ul{z}) \omega(\ul{z};\ul{w}) - \int_{r(\ul{z},\ul{w}) \leq r} \frac{\d z}{8\pi \im} \d' f (\ul{z}) \omega(\ul{z};\ul{w}).
\eeqn
In particular, when $f$ is $\d'$-closed we have
\beqn\label{eqn:residueformula}
	f(\ul{w}) = \oint_{r(\ul{z},\ul{w}) = r} \frac{\d z}{8\pi \im} \omega(\ul{z};\ul{w}) f(\ul{z}) .
\eeqn

\begin{lemma}
\label{lem:exact1}
One has the following identities
\begin{align*}
\frac{\del}{\del \wbar} \omega(\ul{z};\ul{w}) & = \d' V_{\wbar} (\ul{z},\ul{w}) \\
\frac{\del}{\del s} \omega(\ul{z};\ul{w}) & = \d' V_{s} (\ul{z},\ul{w})
\end{align*}
where
\begin{align*}
V_{\wbar} = \frac{s-t}{r(\ul{z},\ul{w})^3} \qquad V_{s} = \frac{2(\ol{z}-\ol{w})}{r(\ul{z},\ul{w})^3}.
\end{align*}
\end{lemma}

In complex geometry, Hartogs' lemma states that when $d > 1$ a holomorphic function with an isolated singularity can always be extended to a holomorphic function; in other words, isolated singularities are removable.
From this lemma we can prove the following analog of Hartogs's lemma in the context of THF's.

\begin{prop}\label{prop:hartog}
Let $K$ be a compact subset of an open set $U \subset \C \times \R$ such that $U - K$ is connected.
Then every $\d'$-closed smooth function $f$ on $U - K$ is the restriction of a $\d'$-closed smooth function on $U$.
\end{prop}
\begin{proof}
Define
\beqn
F(\ul{w}) = \int_{\del U}\frac{\d z}{8\pi \im} \omega(\ul{z};\ul{w}) f(\ul{z}) .
\eeqn
Then, $F$ is defined for all $\ul{w} \in U$.
By the residue formula \eqref{eqn:residueformula} it is straightforward to see that $F$ agrees with $f$ on $U \setminus K$.
To see that $F$ is $\d'$-closed we observe that by Lemma \ref{lem:exact1} we have
\beqn
\frac{\del}{\del \wbar} F = \int_{\del U}\frac{\d z}{8\pi \im} \d' V_{\wbar}(\ul{z},\ul{w}) f(\ul{z}) = \int_{\del U} \d \left(-\frac{\d z}{8\pi \im}V_{\wbar}(\ul{z},\ul{w})f(\ul{z})\right) = 0 .
\eeqn
Similarly, $\frac{\del}{\del s} F = 0$. 
\end{proof}

We can now turn to the proof of Theorem \ref{thm:cohomology}.

\begin{proof}[Proof of Theorem \ref{thm:cohomology}]
First we will argue that 
\beqn
H^{(p);0}_{\d'} (\C \times \R - \{0\}) \simeq H^{(p);0}_{\d'} (\C \times \R) .
\eeqn
In other words, any $\CF$-flat section of $\wedge^p Q^{1,0}$ on $\C \times \R - \{0\}$ extends to a $\CF$-flat section over $\C \times \R$.
Indeed, by Proposition \ref{prop:hartog} any $\d'$-closed function on $\C \times \R - \{0\}$ extends uniquely to a $\d'$-closed function on $\C \times \R$.
A similar result holds for $\d'$-closed sections of $(Q^{1,0})^\vee$.

Now, since the $\d'$-complex on $\C \times \R$ simply splits as a tensor product of the Dolbeault complex of $(p,\bu)$ forms on $\C$ with the de Rham complex on $\R$ we have an isomorphism
\beqn
H^{(p);\bu}_{\d'}(\C \times \R) \simeq H^{p,\bu}_{\dbar} (\C) \otimes H^\bu(\R) .
\eeqn
This proves the first statement in the proposition.

The second statement follows from the fact that the integration pairing
\beqn
\oint_{S^2} \colon \CA^{(0);\bu} \times \CA^{(1);\bu} \to \C[-1]
\eeqn
induces a perfect pairing in $\d'$-cohomology.
\end{proof}

We note that by Lemma \ref{lem:residue} the element dual to $1 \in \CO^{hol}(\C)$ in $H^{(1);1}_{\d'}(\C \times \R - \{0\})$ is the cohomology class of the residue element $\omega(\ul{z};0) \d z$. 

The 1-form $\omega(\ul{z};\ul{w})$ on $\C\times \R -\{\ul{w}\}$ is the direct analog of the function $(z-w)^{-1}$ on $\C-\{w\}$. In particular, they are each Green's functions for the corresponding differential: the latter is the Green's function supported at $w$ for the Dolbeault differential $\ol{\pd}$ and the former the Green's function supported at $\ul{w}$ for the THF differential $\d'$. Correspondingly, they play the role of propagators in field theories having the corresponding kinetic operators.

\subsection{An algebraic model}
\label{sec:model}
The goal in this section is to construct an explicit algebraic model for the Dolbeault-de Rham complex of the THF manifold $\C \times \R - \{0\}$.

Let $\CR$ be the commutative algebra generated by variables 
\beqn
	z,\lambda,x
\eeqn
subject to the relation
\beqn
	z \lambda + x^2 = 1 .
\eeqn
Let $\CA$ be the graded commutative algebra freely generated over $\CR$ by a degree $+1$ element $\omega$.
Thus $\CA = \CR \oplus \CR \omega [-1]$.
Define a differential
\beqn
	\d' \colon \CA \to \CA [1]
\eeqn
on $\CA$ by the formulas
\beqn
	\d' (z) = 0, \quad \d' (\lambda) = x \omega, \quad \d' (x) = -\tfrac{1}{2} z \omega ,
\eeqn
and extend it to all of $\CA$ by the rule that it is a graded derivation.

\begin{prop}
\label{prop:algcohomology}
The cohomology of $(\cA,\d')$ is concentrated in degree zero and one.
The zeroth cohomology $H^0 (\cA) = \C[z]$  is identified with polynomials in the holomorphic variable $z$.
The first cohomology $H^1(\cA) = \C[\lambda] \omega$ is identified with polynomials in the variable $\lambda$ times the element $\omega$.
\end{prop}
\begin{proof}
To compute the cohomology we decompose $\CR$ as a $\lie{sl}(2)$ representation.
The vector fields $(e,f,h)$ where 
\beqn
e = z \del_x - 2 x \del_\lambda, \quad f = -\lambda \del_x + 2x \del_z 
\eeqn
and $h = [e,f]$ endow $\CR$ with the structure of an $\lie{sl}(2)$ representation.
Consider the Laplacian $\Delta = \del_x^2 + 4 \del_z \del_\lambda$ acting on the polynomial algebra $\C[x,z,\lambda]$.
Let $\CH_n \subset \C[x,z,\lambda]$ be the space of polynomials in $x,z,\lambda$ of homogenous degree $n$ which are $\Delta$-closed.
Then under the quotient map $q \colon \C[x,z,\lambda] \to \CR$ we have a decomposition $\cR = \oplus_n \CR_n$; where $\cR_n = q(\CH_n)$.
Computing characters we see that $\CR_n = \Sym^{2n}(\C^2)$ as $\lie{sl}(2)$ representations.
Note that for each $n$, the highest weight vector of $\CR_n$ is proportional to $z^n$ while the lowest weight vector is proportional to $\lambda^n$.

With this identification, we see that $\d'=-\tfrac{1}{2}\omega \wedge e$ where $e$ is as above.
We can thus identify the zeroth cohomology of $\CA$ as
\beqn
H^0 (\CA) = \C[z] = \bigoplus_{n \geq 0} \C \{\text{highest\;weight\;vector\;in} \; \CR_n\} .
\eeqn
Similarly
\beqn
H^1(\CA) = \C[\lambda] \omega = \bigoplus_{n \geq 0} \C \{\text{lowest\;weight\;vector\;in} \; \CR_n\} \cdot \omega .
\eeqn
\end{proof}

The commutative dg algebra $\cA$ is an algebraic model for the THF cohomology of the THF manifold $\C\times \R - \{0\}$.

\begin{prop}
\label{prop:model1}
There is an injective morphism of commutative dg algebras 
\[
j \colon \CA \hookrightarrow \cA^{(0);\bu}(\C\times \R - \{0\}) .
\]
which is a dense embedding in cohomology.
On degree zero generators it is defined by $j(z) = z$, $j(\lambda) = \frac{\zbar}{r^2}$, $j(x) = \frac{t}{r}$ where $r^2 = z \zbar + t^2$.
On the degree one generator we define $j (\omega) = \omega(\ul{z};0)$, see equation \eqref{eqn:omegas2}.
\end{prop}

\begin{proof}
It is an immediate computation to see that $j$ is a cochain map.
The remaining assertions follow from Theorem \ref{thm:cohomology} and Proposition \ref{prop:algcohomology}.
\end{proof}

Both $\CA$ and $\cA^{(0);\bu}(\C\times \R - \{0\})$ are cochain complexes which have actions by the Lie algebra of algebraic vector fields $\Vect^{alg}(\C) = \C[z] \del_z$ defined via the Lie derivative.
The map $j$ is equivariant for this action.
More generally there is an algebraic model for the Dolbeault-de Rham complex with values in an arbitrary invariant vector bundle.
If $s \in \R$ we have the line bundle $K^{\otimes s}$ whose smooth sections are formal expressions of the form $f(\ul{z}) \d z^{s}$.
There is a similar algebraic model for $\cA^{(0);\bu}(\C\times \R - \{0\},K^{\otimes s})$ which as a graded vector space is
\beqn
\CA^{(s)} \define \CR \cdot \d z^s \oplus \CR \cdot \d z^{s} \omega [-1] .
\eeqn
The differential is given by the same formula
\beqn
\d' (f(z,t) \d z^s) = (\d' f(z,t)) \d z^{s} .
\eeqn
The map $j$ extends in a natural way to a cochain map $j^{(s)} \colon \CA^{(s)} \to \cA^{(0);\bu}(\C\times \R - \{0\},K^{\otimes s})$ which is also $\Vect^{alg}(\C)$ equivariant and dense in cohomology.

The proposition implies a relationship between the cohomology classes $\lambda^m \omega \in H^1 (\CA)$ in the algebraic model and the $m$th holomorphic derivative of the class $\omega(\ul{z};0)$ via
\beqn
	j\left(\lambda^m\omega\right) = \frac{(-2)^m}{(2m+1)!!} \pd_z^m \omega(\ul{z};0) .
\eeqn
The algebra structure on the cohomology of $\cA$ is inherited from the one at the cochain level: the degree 0 part $\C[z]$ has its natural ring structure, the product of two degree 1 elements vanishes, and the product of a degree 0 and degree 1 element is as follows. Based on the form of the differential, we already know that $z \omega = -2\d' x \equiv 0$; more generally we have
\beqn
\begin{aligned}
	z \lambda^{m+1} \omega & = (1-x^2)\lambda^{m}\omega\\
	& = \lambda^m \omega - \big(\tfrac{1}{2(m+1)} z \lambda^{m+1}\omega + \diff'(\tfrac{1}{m+1}x \lambda^{m+1})\big)
\end{aligned}
\eeqn
In particular, we have
\beqn
	z \lambda^{m+1} \omega \equiv \frac{2(m+1)}{2m+1} \lambda^m \omega 
\eeqn
In view of this relation, and the above identification with derivatives of $\omega(\ul{z};0)$, it convenient to define
\beqn
	\Omega^m \define \frac{(2m+1)!!}{2^m m!} \lambda^m \omega \rightsquigarrow j\left(\Omega^m\right) = \frac{(-1)^m}{m!} \pd_z^m \omega(\ul{z};0)
\eeqn
so the action of $z$ at the level of cohomology takes the form
\be
z \Omega^m = \begin{cases}
	0 & m = 0\\ 
	\Omega^{m-1} & m > 0
\end{cases}
\ee
and, more generally,
\be
z^n \Omega^m = \begin{cases}
	0 & n > m\\
	\Omega^{m-n} & n \leq m
\end{cases}
\ee

Schematically, $\Omega^m$ plays the role of the homogenous Laurent polynomial ``$\frac{1}{z^{m+1}}$''.
Notice that unlike the case of ordinary Laurent polynomials, however, there is no relation of the form $z^n \frac{1}{z^m} = z^{n-m}$ for $n \geq m$.\footnote{This points to the existence of a transferred $A_\infty$ structure on the cohomology of $\CA$, which we do not consider here.}


From here on we will mostly be interested in the cohomology of $\cA$ as a graded commutative algebra and we denote it by 
\beqn
\CK_{poly} \define H^\bu (\cA,\d') .
\eeqn
Thus, in degree zero $\CK_{poly}^0 = \C[z]$ and in degree one 
\beqn
\CK_{poly}^1 = \C[\lambda] \omega = \C\{\Omega^0,\Omega^1,\ldots\} .
\eeqn
$\CK_{poly}$ vanishes in degrees higher than $1$. There is a relation in this graded algebra $z^n \Omega^m = \Omega^{m-n}$ for $n \leq m$ and $z^n \Omega^m = 0$ otherwise.
One other way we can present $\CK_{poly}$ is as the quotient of the polynomial algebra generated by the infinitely many variables $z,\Omega^0,\Omega^1, \ldots$ of degrees $0,1,1,\ldots$ subject to the relations $z^n \Omega^m = \Omega^{m-n}$ for $n \leq m$, $z^n \Omega^m = 0$ for $n > m$, and $\Omega^n \Omega^m = 0$ for all $n,m \geq 0$.
We refer to $\CK_{poly}$ as the space of polynomials on the algebraic raviolo.
This is the raviolo analog of the space of polynomials $\C[z,z^{-1}]$ on $\C^\times$.

There are variants of $\CK_{poly}$ that will be of interest to us:
\begin{itemize}
	\item Let $\CK$ be the commutative graded algebra which in degree zero is $\C[\![z]\!]$, in degree one is $\C[\lambda] \omega = \C\{\Omega^0, \Omega^1,\ldots\}$ and whose ring structure satisfies the same relations as $\CK_{poly}$.
	This is the analog of the ring of Laurent series $\C(\!(z)\!)$ in a single variable, and hence we refer to this as the space of raviolo Laurent series.
	There is an embedding of algebras $\CK_{poly} \hookrightarrow \CK$.
	\item Let $\CK_{dist}$ denote the graded vector space which in degree zero is $\C[\![z]\!]$, and in degree one is $\C[\![\lambda]\!] \omega$.
	Notice that one cannot multiply arbitrary elements in $\CK_{dist}$ so that there is no natural commutative algebra structure.
	Nevertheless, $\CK_{dist}$ is a graded module over $\CK_{poly}$.
	We call $\CK_{dist}$ the space of formal raviolo distributions, it is the analog of $\C[\![z,z^{-1}]\!]$.
\end{itemize}
For $\alpha$ in $\CK_{poly}$ or $\CK$, we say that $\alpha$ has a \emph{pole of order $m$} if the coefficient in front of $\Omega^m$ is nonzero, but the coefficient of $\Omega^{n}$ in $\alpha$ vanishes for all $n > m$. If the coefficient in front of $\Omega^n$ vanishes for all $n\geq0$, we say $\alpha$ is \emph{regular}.

We note that the graded algebras $\CK,\CK_{poly}$ are each equipped with a natural degree zero derivation $\pd_z$ defined by
\be
	\pd_z z^n = n z^{n-1} \qquad \pd_z \Omega^m = -(m+1) \Omega^{m+1}
\ee
Given an element $f(z) \in \CK_{dist}$ (or $\CK, \CK_{poly}$), it is useful to separate it into its homogeneous pieces. Explicitly, if
\be
	f(z) = \sum_{n\geq0} z^{n} f_n + \Omega^n f_{-n-1}
\ee
then we define
\be
	f(z)_+ \define \sum_{n\geq0} z^{n} f_n \qquad f(z)_- \define \sum_{n\geq0} \Omega^n f_{-n-1}
\ee

We have utilized integration of Dolbeault-de Rham forms along the $2$-sphere to characterize THF cohomology in the previous sections.
We now develop an algebraic version of the residue.
In terms of differential forms in the analytic model for THF cohomology, integration over the $2$-sphere centered at the origin defines a linear map
\beqn
\oint_{S^2} \colon \CA^{(1);\bu} \to \C[-1]
\eeqn
of cohomological degree $-1$.
Similarly, if we fix the holomorphic one-form $\d z$ we obtain a linear map 
\beqn
\oint_{S^2} \d z \wedge (-) \colon \CA^{(0);\bu} \to \C[-1]
\eeqn
also of cohomological degree one.
This map clearly descends to cohomology.

As we have already pointed out, the two-form $\d z \wedge \omega(\ul{z};0)$ is the fundamental class of this $2$-sphere (up to a factor of $8 \pi \im$), see Lemma \ref{lem:residue}.
Formally, we can mimic this integration to define a residue map on our algebra model $\CK$.
We define the linear map
\beqn
	\Res  \colon \CK^{(1)} \to \C[-1]
\eeqn
of cohomological degree one by the formula
\be
	\Res \d z \, z^n \define 0\,, \qquad \Res \d z \, \Omega^m z^n \define \delta_{n,m} .
\ee
Note that this is equivalent to $\Res \d z \, z^m = 0$ and $\Res \d z \, \Omega^m = \delta_{m,0}$. In the analytic description, we see that $\Res$ can be identified with performing a surface integral over an $S^2$ (of any radius) centered at $0$, divided by $8 \pi i$. We add a subscript to this residue map if we need to specify the corresponding variable, e.g. $\Res_z$ or $\Res_w$.

More generally, if we do not explicitly include the factor $\d z$, the 2-sphere residue defines a non-degenerate pairing
\be
	\langle -, -\rangle \colon \CK^{(s)} \otimes \CK^{(1-s)} \to \C[-1] .
\ee
In particular, the linear dual of $\CK$ can be identified with $\CK^{(1)}[1]$.

\subsection{The raviolo delta function}

Let $\CK_{dist}^{z,w}$ denote the graded vector space which in degree zero is $\C[\![z,w]\!]$, in degree one is $\C[\![z,\lambda_w]\!] \omega_w \oplus \C[\![\lambda_z,w]\!] \omega_z$, in degree two is $\C[\![\lambda_z,\lambda_w]\!] \omega_z \omega_w$.
In other words, $\CK_{dist}^{z,w} = \CK_{dist}^{\otimes 2}$.
We will refer to this as the space of \textit{formal ravioli distributions} in two variables $z,w$.
In this space, define the degree one element
\beqn
	\Delta(z-w) \define \sum_{n \geq 0} w^n \Omega_z^{n} - \sum_{n < 0} z^{-n-1} \Omega_w^{-n-1}
\eeqn 
where $\Omega_z^n$ and $\Omega_w^n$ are defined in terms of $\lambda_z, \omega_z$ and $\lambda_w, \omega_w$ as above.
Notice that for any raviolo distribution $A(z)$ in one variable, we have the equality
\beqn\label{eqn:essential}
	\Delta(z-w) A(z) = \Delta(z-w) A(w) .
\eeqn

We also define 
\begin{align*}
	\Delta_-^{(j)}(z-w) & \define \sum_{n \geq 0} \begin{pmatrix} n + j \\ j \end{pmatrix}  w^{n} \Omega_z^{n+j} \in \CK^1_{z,dist} \otimes \CK^0_{w,dist} \\
	\Delta^{(j)}_+(z-w) & \define \sum_{n < 0} (-1)^{j+1} \begin{pmatrix} j-n-1 \\ j \end{pmatrix} z^{-n-1} \Omega_w^{j-n-1} \in \CK^0_{z,dist} \otimes \CK^1_{w,dist} ,
\end{align*}
so that 
\beqn
	\Delta(z-w) = \Delta_-(z-w) + \Delta_+(z-w),
\eeqn
where we set $\Delta_{\pm} = \Delta_\pm^{(0)}$.
We observe that $\Delta^{(j)}_{\pm} (z-w) = \frac{1}{j!} \del_w^j \Delta_{\pm}(z-w)$.

If $R$ is any $\C$-algebra then $R[\![z]\!]$ is the space of $R$-valued single variable formal Taylor series.
We let $R \langle\!\langle z \rangle\!\rangle = R \otimes \CK$
be the space of $R$-valued formal raviolo Laurent series; this is the analog of $R$-valued single variable formal Laurent series $R(\!(z)\!)$. Concretely, $R \langle \!\langle z \rangle \! \rangle$ is the space of series
\beqn\label{eqn:RK}
	\sum_{n < 0} z^{-n-1} a_n + \sum_{n \geq 0} \Omega_z^{n} a_n 
\eeqn
where $a_n \in R$ for all $n \in \Z$ with the property that there exists an $N$ such that $a_n = 0$ for $n \geq N$. For example, $\CK = \C\langle\!\langle z \rangle\!\rangle$.
Note that $R\langle\!\langle z \rangle\!\rangle$ is naturally a graded algebra.

We can apply this construction to $R = \C\langle \!\langle z\rangle\!\rangle$ to obtain the graded algebra $\C\langle \!\langle z\rangle\!\rangle\langle \!\langle w\rangle\!\rangle$.
Elements of this space are expressions of the form \eqref{eqn:RK} where each $a_n \in \C\langle \!\langle z\rangle\!\rangle$ and $a_n = 0$ for $n \gg 0$.
For example, observe that for all $j \geq 0$ one has 
\beqn
	\Delta^{(j)}_- (z-w) \in \C\langle \!\langle z\rangle\!\rangle\langle \!\langle w\rangle\!\rangle .
\eeqn
Similarly, for all $j \geq 0$
\beqn
	\Delta^{(j)}_+ (z-w) \in \C\langle \!\langle w\rangle\!\rangle\langle \!\langle z\rangle\!\rangle .
\eeqn

Both $\C\langle \!\langle z\rangle\!\rangle\langle \!\langle w\rangle\!\rangle$ and $\C\langle \!\langle z\rangle\!\rangle\langle \!\langle w\rangle\!\rangle$ naturally embed inside of the space of bivariate formal ravioli distributions $\CK^{z,w}_{dist}$ defined above.
Their intersection is the quotient of the algebra $\C[\![z,w]\!][\Omega^{n}_z, \Omega^m_z]_{n,m \geq 0}$, where $z,w$ are degree zero and $\Omega^n_z,\Omega^m_w$ are degree one, subject to the usual relations 
\beqn
	x^n \Omega^m_x = \Omega_x^{m-n}
\eeqn
for $n \leq m$ and $x=z,w$, $x^n \Omega^m_x = 0$ for $n > m$ and $x=z,w$ as well as
\beqn
	\Omega^n_x \Omega^m_x = 0
\eeqn
for $n,m \geq 0$ and $x=z,w$.
We denote this algebra by
\beqn
	\CK^{z,w}_{poly} \define \C\langle \!\langle z\rangle\!\rangle\langle \!\langle w\rangle\!\rangle \cap \C\langle \!\langle w\rangle\!\rangle\langle \!\langle z\rangle\!\rangle \subset \CK_{dist}^{z,w} .
\eeqn



\begin{prop}
The following identities hold involving the formal distribution~$\Delta(z-w)$.
\begin{itemize}
	\item[(a)] For any Laurent series $f \in \CK$ one has
	\beqn
	\Res_z \d z \Delta(z-w) f(z) = f(w) .
	\eeqn
	\item[(b)] $\Delta(z-w) = -\Delta(w-z)$.
	\item[(c)] $\del_z \Delta(z-w) = - \del_w \Delta(z-w)$.
\end{itemize}
\end{prop}
\begin{proof}
For (a) it suffices to check this relation for $f(z) = z^n$ or $f(z) = \Omega^{m}_z$ for $n \geq 0$, $m \geq 0$ which is straightforward.

For (b) notice
\begin{align*}
	\Delta(z-w) & = \sum_{n \geq 0} w^n \Omega_z^{n} - \sum_{n < 0} z^{-n-1} \Omega_w^{-n-1} \\
	& = \sum_{n < 0} w^{-n-1} \Omega_z^{-n-1} - \sum_{n \geq 0} z^n \Omega_w^n \\
	& = -\Delta(w-z) .
\end{align*}

For (c) we see directly that
\[
-\del_z \Delta(z-w) = \sum_{n \geq 0} (n+1) w^n \Omega_z^{n+1} - \sum_{n < 0} (n+1) z^{-n-2} \Omega_w^{-n-1} 
\]
is the same as  
\[
\del_w \Delta(z-w) = \sum_{n \geq 0} n w^{n-1} \Omega_z^n - \sum_{n < 0} n z^{-n-1} \Omega_w^{-n} 
\]
after reindexing the summations via $n \to n+1$.
\end{proof}

In comparing equivalent definitions of mutual locality for raviolo fields, we will use the following technical lemma which provides a uniqueness result for bivariate formal ravioli distributions.

\begin{lemma}
\label{lem:deltafunction}
Let $N \geq 0$ be a non-negative integer and let $f(z,w) \in \CK_{dist}^{z,w}$ be a bivariate formal ravioli distribution satisfying the following properties:
\begin{enumerate}
	\item $(z-w)^{N+1} f(z,w) = 0$, and
	\item or all $m \geq 0$ 
	\[
		\left(\Omega^m_z - \sum_{j=0}^N (-1)^j \binom{m+j}{j} (z-w)^j \Omega^{m+j}_w \right) f(z,w) = 0 .
	\]
\end{enumerate}
Then $f(z,w)$ can be written uniquely as a sum 
\beqn\label{eqn:fexp}
	f(z,w) = \sum_{i=0}^{N} \del_w^i \Delta(z-w) g^{(i)}(w)
\eeqn
where $g^{(i)}(w) \in \CK_{dist}^w$.
\end{lemma}
\begin{proof}
Suppose that $g_0(w) \in \CK_{dist}^w$ is any formal raviolo distribution in the variable $w$.
Applying \eqref{eqn:essential} to the cases $A(z) = z^n, A(z) = \Omega^m_z$ we obtain
\begin{align*}
	(z^n - w^n)\Delta(z-w) g_0(w) & = 0, \quad n \geq 0\\
	(\Omega^m_z - \Omega^m_w) \Delta(z-w) g_0(w)& = 0, \quad m \geq 0 .
\end{align*}
Note that the first equation is trivial for $n = 0$, and the relation for $n > 1$ follows from the $n=1$ case.

Next, consider the derivative $\del_w \Delta(z-w)$.
Differentiating the relation \eqref{eqn:essential} we obtain
\begin{align*}
	\del_w \Delta(z-w) A(z) & = \del_w \Delta(z-w) A(w) + \Delta(z-w) \del_w A(w) \\
	& =  \del_w \Delta(z-w) \left(A(w) + (z-w) \del_w A(w)\right) ,
\end{align*}
for any single variable formal raviolo distribution $A(z)$.
The second equality follows from the relation $\Delta(z-w) = (z-w) \del_w \Delta(z-w)$, which follows from taking the derivative of \eqref{eqn:essential} when $A = z$.
Let $g_1(w)$ be any formal raviolo distribution in the variable $w$.
Applying this derivative formula to the test functions $A = z^n, A=\Omega^m_z$ we obtain
\begin{align*}
	\left(z^n - (w^n + n w^{n-1} (z-w) \right) \del_w \Delta(z-w)g_1(w) & = 0, \quad n \geq 0\\
	\left(\Omega^m_z - (\Omega^m_w - (m+1) (z-w) \Omega_w^{m+1})\right) \del_w \Delta(z-w) g_1(w) & = 0, \quad m \geq 0 .
\end{align*}
In the first relation, the cases $n=0,1$ are trivial, while the relation for $n > 2$ follows from the case $n=2$ which reads
\beqn
	(z-w)^2 \pd_w \Delta(z-w)g_1(w) = 0 .
\eeqn

More generally, taking the $N$th derivative of the relation \eqref{eqn:essential} we obtain
\beqn
	\del_w^N \Delta(z-w) A(z) = \del_w^N \Delta(z-w) \sum_{j=0}^N \frac{(z-w)^j}{j!} \del_w^j A(w) 
\eeqn
for any distribution $A(z)$.
Applying this to the test distribution $A(z) = z^{N+1}$ we obtain
\beqn
	(z-w)^{N+1} \del_w^N \Delta(z-w) g_N(w) = 0
\eeqn
which holds for any formal raviolo distribution $g_N(w)$ in the variable $w$.
Similarly, applying this to the test distribution $A(z) = \Omega_z^m$, we obtain
\beqn
	\left(\Omega^m_z - \sum_{j=0}^N (-1)^j \binom{m+j}{j} (z-w)^j \Omega^{m+j}_w \right) \del_w^N \Delta(z-w) g_N(w) = 0 , \quad m \geq 0 .
\eeqn
Thus, we see that any bivariate formal ravioli distribution of the form \eqref{eqn:fexp} satisfies conditions (1) and (2).

Conversely, suppose that $f$ satisfies the equations in (1) and (2).
We prove the result for $f$ of homogenous cohomological degree.
If $f$ is of degree zero then condition (1) implies that $f \equiv 0$. 

Suppose that $f$ is degree one.
Then it admits an expansion of the form
\beqn
	f(z,w) = \sum_{m < 0, n \geq 0} f_{m,n} z^{-m-1} \Omega_w^{n} + \sum_{m \geq 0, n < 0} f_{m,n} \Omega_z^m w^{-n-1} ,
\eeqn
where $f_{m,n}$ are coefficients.
Define
\beqn
	(\delta f)_{m,n} \define f_{m+1,n} - f_{m,n+1} .
\eeqn
Condition (1) implies the following recursive relation among the coefficients
\beqn\label{eqn:Delta}
	\left(\delta^{N+1} f\right)_{m,n} = 0 
\eeqn
which must hold in the two cases 
\[
	\begin{array}{lllll}
	\alpha \colon & m < 0 , & n \geq 0 \\
	\beta \colon & m \geq 0, & n  < 0 .
	\end{array}
\]
We remark that in the relation \eqref{eqn:Delta} the coefficients $f_{m,n}$ are only defined when $(m,n)$ lie in $\alpha$ or $\beta$. 
Our convention is that when indices $m,n$ appear that do not fall into the sets $\alpha,\beta$ above then we declare $f_{m,n} = 0$.
For example, in the case $N=0$ we have $f_{m+1,n} - f_{m,n+1} = 0$ when $m < -1, n \geq 0$ (respectively, $m \geq 0, n < -1$) and $f_{-1, n+1} = 0$ (respectively, $f_{m+1,-1}=0$) for $n \geq 0$ (respectively, $m \geq 0$).

Relation \eqref{eqn:Delta} implies that $f(z,w)$ is of the form
\beqn
	f(z,w) = \sum_{p \in \Z} f^{(p)}(z,w)
\eeqn
where $f^{(p)}_{m,n} = 0$ unless $m+n=p$ where $(m,n) \in \alpha$ or $\beta$.
Suppose that~$p \in \Z$ and $(m,n) \in \alpha$.
Then, Eq. \eqref{eqn:Delta} has $N+1$ independent solutions
\beqn
	w^{p+i+1} \del_w^i \Delta_+(z-w) , \quad i = 0,\ldots,N .
\eeqn
In fact, since \eqref{eqn:Delta} is a difference equation with $N+1$ terms, these are all of the solutions.
Similarly, if $(m,n) \in \beta$, then Eq. \eqref{eqn:Delta} has the $N+1$ independent solutions
\beqn
	w^{p+i+1} \del_w^i \Delta_-(z-w), \quad i =0,\ldots, N .
\eeqn
Thus, we see that condition (1) implies that $f(z,w)$ is of the form
\beqn\label{eqn:fexp1}
	f(z,w) = \sum_{i=0}^N \del_w^i \Delta_+(z-w) g_{+}^{(i)}(w) + \sum_{i=0}^N \del_w^i \Delta_-(z-w)  g_{-}^{(i)} (w) 
\eeqn
where $g_\pm^{(i)} (w) = \sum_{m \geq 0} g_{\pm,m}^{(i)} w^{m} \in \CK_{dist}^w$ are of cohomological degree zero.

We now consider condition (2).
Observe the following identities
\begin{align*}
	\Omega^m_z \Delta_+(z-w) & = \Omega^m_w \Delta_- (z-w) = \Omega_{z,w}^m \\
	\Omega^m_z \Delta_-(z-w) & = \Omega^m_w \Delta_+(z-w) = 0 
\end{align*}
for each $m \geq 0$, where
\beqn
	\Omega^m_{z,w} \define \sum_{\ell = 0}^m \Omega^\ell_z \Omega^{m-\ell}_w \in \CK^{z,w}_{dist} .
\eeqn
Given the expansion as in \eqref{eqn:fexp1} we see that
\beqn
	\Omega_z^m f (z,w) = \sum_{i=0}^N \del_w^i \Omega_{z,w}^m  g_+^{(i)} (w)
\eeqn
and
\beqn
	\left(\sum_{j=0}^N (-1)^j \binom{m+j}{j} (z-w)^j \Omega^{m+j}_w \right) f(z,w) = \sum_{i=0}^N \del_w^i \Omega_{z,w}^m g_-^{(i)} (w).
\eeqn
Applying condition (2) successively for each $m \geq 0$ we thus find that all coefficients of $g_{\pm}^{(i)}$ must satisfy $g^{(i)}_{+,m} - g^{(i)}_{-,m} = 0$ for every $i,m \geq 0$; thus $f(z,w) = \sum_{i=0}^N \del_w^i \Delta(z-w) g^{(i)} (w)$ where $g^{(i)} = g^{(i)}_+ = g^{(i)}_-$ as desired. 

Finally suppose that $f$ is degree two.
In this case, we note that condition (2) is vacuously true.
The formal distribution admits an expansion of the form
\beqn
	f(z,w) = \sum_{n,m \geq 0} f_{m,n} \Omega_z^m \Omega_w^{n} 
\eeqn
for coefficients $f_{m,n}$.
Condition (1) implies the relation
\beqn\label{eqn:Delta2}
	(\delta^{N+1} f)_{m,n} = 0
\eeqn
for $m,n \geq 0$.
This means that $f(z,w)$ is of the form
\beqn
	f(z,w) = \sum_{p \in \Z} f^{(p)}(z,w)
\eeqn
where $f^{(p)}_{m,n} = 0$ unless $n-m=p$.
For each $p$ this equation has $N+1$ independent solutions
\beqn
	\del_w^i \Delta (z-w) \Omega_w^{p+i} , \quad i = 0,\ldots,N .
\eeqn
This completes the proof.
\end{proof}

\subsection{Distributional expansions}
\label{sec:expansions}
The vertex algebra axioms use the function space 
\beqn
	\C[\![z,w]\!][z^{-1}, w^{-1}, (z-w)^{-1}]
\eeqn
to formulate notions like mutual locality and associativity. Geometrically, this function space is closely related to the space of holomorphic functions on the configuration space of three points in $\C$, modulo overall translations.
In our context, the natural replacement for this is thus the space of `holomorphic' functions on the configuration space of three points in the THF space $\C \times \R$, modulo overall translations. (By holomorphic we mean those functions preserved by the foliation defining the THF structure on $\C \times \R$).
As in the case of two points, we must consider not only functions on the configuration space of three points, but the full THF cohomology of this space.

There are three differential forms of particular interest: we have the two 1-forms $\omega(\ul{z};0)$ and $\omega(\ul{w};0)$ inherited from the two coordinate copies of $\C\times\R-\{0\}$, as well as the 1-form
\beqn
	\omega(\ul{z}, \ul{w}) \define \frac{(t-s) (\d \zbar - \d \wbar) - 2(\zbar - \wbar) (\d t - \d s)}{r(\ul{z},\ul{w})^3} .
\eeqn
from the diagonal. Note that $\omega(\ul{z}, \ul{w}) = \omega(\ul{w}, \ul{z})$ is not quite the same as $\omega(\ul{z};\ul{w})$, cf. Eq. \eqref{eqn:omegas2}; rather, it is $\omega(\ul{z}-\ul{w};0)$. As was the case on $\C \times \R - \{0\}$, the derivatives of these 1-forms with respect to the anti-holomorphic coordinates $\zbar, \wbar$ and the smooth coordinates $t$, $s$ are trivial in $\diff'$ cohomology, as are $z \omega(\ul{z};0)$, $w \omega(\ul{w};0)$, and $(z-w) \omega(\ul{z},\ul{w})$. Also note that
\beqn
	\pd_w \omega(\ul{z};0) = 0 = \pd_z \omega(\ul{w};0) \qquad \pd_z \omega(\ul{z}, \ul{w}) = - \pd_w \omega(\ul{z}, \ul{w})
\eeqn
Our algebraic model for the space of functions should thus have three towers of generators $\Omega^n_z$, $\Omega^n_w$, and $\Omega^n_{z-w}$ in degree 1. Explicitly, we define $\three$ as the quotient of the graded polynomial algebra 
\beqn
	\C[\![z,w]\!][\Omega_z^n,\Omega^m_m,\Omega^k_{z-w}]_{n,m,k \geq 0}
\eeqn
where $z,w$ are of degree zero, and $\Omega^n_x$ are of degree $+1$ for $x=z,w,z-w$, $n \geq 0$ by the following relations:
\begin{itemize}
	\item $x^n \Omega^m_x = \Omega_x^{m-n}$ for $n \leq m$ and $x=z,w,z-w$, $x^n \Omega^m_x = 0$ for $n > m$ and $x=z,w,z-w$.
	\item $\Omega^n_x \Omega^m_x = 0$ for $n,m \geq 0$ and $x=z,w,z-w$.
	\item Finally, the relation between degree one generators
	\beqn
	\Omega^0_{z-w} \Omega^0_z + \Omega^0_w \Omega^0_{z-w} + \Omega^0_z \Omega^0_w = 0
	\eeqn
	together with all relations obtained by taking holomorphic derivatives $\del_z,\del_w$.\footnote{This is analogous to the degree two relation present in the de Rham cohomology $H^\bu(Conf_3(\R^2))$ of three points in $\R^2$, see \cite{Arnold,Totaro}.
	We expect that our relation is present in the THF cohomology of three points in $\C \times \R$.}
\end{itemize}

There are three particularly important regions of this configuration space: when $\ul{w}$ is close to $0$, $r(\ul{z},0) \gg r(\ul{w},0)$; when $\ul{z}$ is close to $0$, $r(\ul{z},0) \ll  r(\ul{w},0)$; and when $\ul{z}$ is close to $\ul{w}$, $r(\ul{z}, \ul{w}) \ll r(\ul{w},0)$. When $\ul{w}$ is close to $0$ we can expand $\omega(\ul{z},\ul{w})$ as a series in $\ul{w}$; as the $\zbar$ and $t$ derivatives of $\omega(\ul{z};0)$ are trivial in cohomology, we are lead to the following expansion in our algebraic model:
\beqn
	\Omega^0_{z-w} \to \sum_{n \geq 0} \frac{(-w)^n}{n!} \pd^n_z\Omega^0_z = \sum_{n \geq 0} w^n \Omega^n_z = \Delta_-(z-w)
\eeqn
and, more generally,
\beqn
	\Omega^m_{z-w} \to \sum_{n \geq 0} \binom{m+n}{n} w^n \Omega^{m+n}_z = \tfrac{1}{m!} \pd^m_w \Delta_-(z-w) .
\eeqn
Similarly, when $\ul{z}$ is close to $0$ we expand as a series in $z$
\beqn
	\Omega^m_{z-w} \to (-1)^m\sum_{n \geq 0} \binom{m+n}{n} z^n \Omega^{m+n}_w = -\tfrac{1}{m!}\pd^m_w\Delta_+(z-w) .
\eeqn
Finally, when $\ul{z}$ is close to $\ul{w}$ we expand as a series in $(z-w)$
\beqn
	\Omega^n_{z} \to \sum_{n\geq0} (-1)^n \binom{n+m}{n} (z-w)^n \Omega^{m+n}_w .
\eeqn
\begin{lemma}
\label{lem:expansions}
The assignment
\beqn
	\Omega^n_{z-w} \mapsto \tfrac{1}{n!} \del^n_w \Delta_-(z-w)
\eeqn
defines a graded algebra map $\three \to \C\langle\!\langle z\rangle\!\rangle\langle\!\langle w\rangle\!\rangle$.
Similarly, the assignment
\beqn
	\Omega^n_{z-w} \mapsto -\tfrac{1}{n!} \del^n_w \Delta_+ (z-w) 
\eeqn
defines a graded algebra map $\three \to \C\langle\!\langle w\rangle\!\rangle\langle\!\langle z\rangle\!\rangle$ and the assignment
\beqn
	\Omega^n_{z} \mapsto \sum_{n\geq0} (-1)^n \binom{n+m}{n} (z-w)^n \Omega^{m+n}_w
\eeqn
defines a graded algebra map $\three \to \C\langle\!\langle w\rangle\!\rangle\langle\!\langle z-w\rangle\!\rangle$.
\end{lemma}
\begin{proof}
For the first claim it suffices to show that
\beqn
	\Delta_-(z-w) \Omega^0_z + \Omega^0_w \Delta_-(z-w) + \Omega^0_z \Omega^0_w = 0
\eeqn
in the space $\C\langle\!\langle z\rangle\!\rangle\langle\!\langle w\rangle\!\rangle$, but this follows by noting
\beqn
	\Delta_-(z-w) \Omega^0_z = 0 \qquad \Omega^0_w \Delta_-(z-w) = \Omega^0_w \Omega^0_z
\eeqn
The second and third claims similarly follow by noting
\beqn
	-\Delta_+(z-w)\Omega^0_z = \Omega^0_w \Omega^0_z \qquad -\Omega^0_w \Delta_+(z-w) = 0
\eeqn
and
\beqn
	\Omega^0_{z-w}\Delta_0(z-w) = \Omega^0_{z-w} \Omega^0_w \qquad \Delta_0(z-w)\Omega^0_w = 0
\eeqn

\end{proof}

Viewing $\C\langle\!\langle z\rangle\!\rangle\langle\!\langle w\rangle\!\rangle$ and $\C\langle\!\langle w\rangle\!\rangle\langle\!\langle z\rangle\!\rangle$ as subspaces of $\CK^{z,w}_{dist}$, we see that the images of $\Omega^0_{z-w}$ under the above maps are not the same: $\Delta_-(z-w) \neq - \Delta_+(z-w)$. Their difference is precisely the raviolo delta function $\Delta(z-w)$. This is entirely analogous to the ordinary formal distributions on the disk, cf. Section 1.1 of \cite{FBZ}.

\section{The definition of a raviolo vertex algebra}
\label{sec:definition}

Like a vertex algebra, we start with a vector space $\CV$, which we imagine being the space of local operators in a mixed holomorphic-topological quantum field theory.
We assume that the vector space is $\Z$-graded by cohomological degree---this is a key difference with the ordinary vertex algebra case.
While most vertex algebras studied in practice make sense without reference to cohomological degree, non-trivial examples of raviolo vertex algebras exist only when the cohomological degree is non-trivial.\footnote{This should be compared to algebras over the little $n$-disks operad valued in ordinary vector spaces.
For any $n>1$ this simply recovers the notion of a commutative algebra.}
We denote by $\CV^r$ the space of vectors which are of homogenous cohomological degree $r$.

%
%
%
%

\subsection{Raviolo fields}
\label{sec:ravfields}

We start with the raviolo analog of a \emph{field} on $\CV$, cf. Section 1.2.1 of \cite{FBZ}. This object will allow us to organize the action of a given local operator on $\CV$ in a way that makes manifest the geometry of the raviolo.

\begin{dfn}
\label{dfn:field}
A \defterm{raviolo field} on $\CV$ is an element
\beqn
	A(z) = \sum_{m < 0} z^{-m-1} A_m + \sum_{m \geq 0} \Omega^{m}_z A_m \in \End(\CV)\otimes \CK_{dist} .
\eeqn
such that for any $v \in \CV$ there exists $N$ sufficiently large with
\beqn
	A_m v = 0, \quad \text{for all} \; m \geq N 
\eeqn
We call the endomorphisms $A_m$ \defterm{modes} of $A(z)$ and the above expression the \defterm{mode expansion} of $A(z)$. We say that $A(z)$ is \defterm{homogeneous} (of degree $|A|$) if $A_m$ is an endomorphism of cohomological degree $|A|$ for $m < 0$ and of cohomological degree $|A|-1$ for $m \geq 0$.
\end{dfn}

A few remarks are in order. First, we often simply call this a \emph{field} when it is the understood that we are not speaking of a field in the sense of vertex algebras. We denote by $\CF_{rav}(\CV) \subset \End(\CV) \otimes \CK_{dist}$ the vector space of (raviolo) fields on $\CV$. We also note that the main property of a field is equivalent to saying $A(z) v$ is a formal raviolo Laurent series valued in $\CV$, that is $A(z) v \in \CV \langle\!\langle z \rangle \! \rangle = \CV \otimes \CK$ for every $v\in \CV$. Equivalently, there exists $N \geq 0$ sufficiently large such that $z^N A(z) v \in \CV[\![z]\!]$. Finally, the shift in degree for the $A_m$ with $m \geq 0$ is due to the fact that $\Omega_z^m$ carries cohomological degree $+1$ in $\CK_{dist}$. Unless otherwise specified, we will only consider homogeneous fields.



The action of $\pd_z$ on $\CK_{dist}$ induces an action on $\End(\CV) \otimes \CK_{dist}$. It is easy to verify that $\pd_z$ preserves the space of fields $\CF_{rav}(\CV)$.
\begin{lemma}
	If $A(z)$ is a field on $\CV$ then
	\begin{equation*}
		\pd_z A(z) = \sum_{m < 0} z^{-m-1} \big(-mA_{m-1}\big) + \sum_{m \geq 0} \Omega^{m}_z \big(-m A_{m-1}\big)
	\end{equation*}
	is also a field on $\CV$.
\end{lemma}

Given two fields $A(z)$ and $B(w)$ on $\CV$, their composition $A(z)B(w)$ is itself not a field. Moreover, it generally does not have a well-defined $z \to w$ limit. If we are to realize the action of local operators on $\CV$ by way of raviolo fields, we need to replace composition of fields with a suitable regularized product that allows us to take this coincidence limit.

\begin{dfn}
	\label{dfn:NOP}
	Let $A(z)$ and $B(w)$ be fields on $\CV$. We define their \defterm{normal-ordered product} $\norm{A(z)B(w)}$ to be the following element of $\End(\CV)\otimes\CK_{dist}^{z,w}$:
	\beqn
	\begin{aligned}
		\norm{A(z) B(w)} & = A(z)_+ B(w) + (-1)^{|A||B|} B(w) A(z)_-\\
		& =\sum_{n < 0} w^{-n-1}\bigg(\sum_{m < 0} z^{-m-1} A_{m} B_{n} + (-1)^{(|A|+1)|B|} \sum_{m \geq 0}\Omega^m_z B_{n} A_{m}\bigg)\\
		& \hspace{-1cm} + \sum_{n \geq 0} \Omega^{n}_w \bigg((-1)^{|A|}\sum_{m < 0} z^{-m-1} A_{m} B_{n} + (-1)^{(|A|+1)(|B|+1)} \sum_{m \geq 0}\Omega^m_z B_{n} A_{m}\bigg)\\
	\end{aligned}
	\eeqn
\end{dfn}

The following lemma is immediate from the definition.

\begin{lemma}
	\label{lem:NOP}
	Let $A(z)$ and $B(w)$ be fields on $\CV$. The specialization of their normal-ordered product $\norm{A(z) B(w)}$ at $w = z$ is a well-defined field on $\CV$. Moreover,
	\beqn
		\norm{A(z) B(z)} = \Res_w \bigg(\Delta_-(w-z) A(w) B(z) - (-1)^{|A||B|} \Delta_+(w-z) B(z) A(w)\bigg)
	\eeqn
\end{lemma}

For brevity, we often denote the specialization $\norm{A(z) B(z)}$ by $\norm{AB}(z)$. We note that the normal-ordered product $\norm{AB}(z)$ is homogeneous of degree $|A|+|B|$.

\begin{proof}
	The desired specialization is given by
	\beqn
	\begin{aligned}
		\norm{AB}(z) & = \sum_{m < 0} z^{-m-1}\bigg(\sum_{n = m}^{-1} A_{n} B_{m-n-1}\bigg)\\
		& \hspace{-1cm} + \sum_{m \geq 0} \Omega^{m}_z \bigg((-1)^{|A|}\sum_{n < 0} A_{n} B_{m-n-1} + (-1)^{|A||B|} B_{n} A_{m-n-1}\bigg)\\
	\end{aligned}
	\eeqn
	To verify that this defines a field on $\CV$, we choose $v \in \CV$; as $A(z), B(w)$ are fields on $\CV$, there exists an integer $N \geq 0$ such that $A_{n} v = 0$ and $B_{n} v = 0$ for all $n \geq N$. As $m-n-1 \geq m$ for $n < 0$, it follows that $v$ is annihilated by the coefficient of $\Omega^m_z$ in $\norm{AB}(z)$ for all $m \geq N$. The second assertion follows from computing residues.
\end{proof}

It is straight forward to verify the following lemma characterizing the interplay between $\pd_z$ and the normal-ordered product.
\begin{lemma}
	\label{lem:derivNOP}
	Let $A(z), B(w)$ be fields on $\CV$, then $\pd_z$ is a derivation of the normal-ordered product
	\begin{equation*}
		\pd_z \norm{A(z)B(z)} = \norm{\pd_z A(z) B(z)} + \norm{A(z) \pd_z B(z)}
	\end{equation*}
\end{lemma}

\subsection{Locality}
\label{sec:locality}

The most important preliminary definition is that of \textit{mutual locality} of two fields $A(z), B(w)$. For this we closely follow the analogous notion in the setting of vertex algebras, cf. Section 1.2 \cite{FBZ}, for example.
In Section \ref{sec:model} we introduced the bivariate ravioli algebras $\C\langle\!\langle z\rangle\!\rangle\langle\!\langle w\rangle\!\rangle$ and $\C\langle\!\langle w\rangle\!\rangle\langle\!\langle z\rangle\!\rangle$, which should be viewed as the ravioli analogs of the bivariate Laurent series algebras $\C(\!(z)\!)(\!(w)\!)$ and $\C(\!(w)\!)(\!(z)\!)$ respectively. Similarly, we introduced the algebra $\three$ in Section \ref{sec:expansions} serving as the ravioli analog of $\C[\![z,w]\!][z^{-1}, w^{-1}, (z-w)^{-1}]$; in Lemma \ref{lem:expansions}, we showed that there were maps of graded algebras
\beqn
	\C\langle\!\langle z\rangle\!\rangle\langle\!\langle w\rangle\!\rangle \leftarrow \three \rightarrow \C\langle\!\langle w\rangle\!\rangle\langle\!\langle z\rangle\!\rangle
\eeqn
encoding the expansion of an element of $\three$ in when either $z$ is close to $0$ (left) or $w$ is close to $0$ (right), cf. Eq. 1.1.9 of \cite{FBZ}.

Choose $v \in \CV$ and $\varphi \in \CV^*$ (the linear dual of $\CV$). Given two raviolo fields $A(z)$ and $B(w)$ it follows that the matrix element $\langle \varphi, A(z) B(w)v\rangle$ belongs to $\C\langle\!\langle z\rangle\!\rangle\langle\!\langle w\rangle\!\rangle$. Similarly, $(-1)^{|A||B|}\langle \varphi, B(w) A(z)v\rangle$ belongs to $\C\langle\!\langle z\rangle\!\rangle\langle\!\langle w\rangle\!\rangle$. Requiring these two matrix elements are equal in $\CK^{z,w}_{dist}$ is too strong of a constraint to impose; as in vertex algebras, we instead require that they are the images of the same element of $\three$.

\begin{dfn}\label{dfn:mutual}
	We say two fields $A(z), B(w)$ are \defterm{mutually local} if for every $v \in \CV$ and $\varphi \in \CV^*$ (the linear dual of $\CV$) the matrix elements
	\begin{equation*}
		\langle \varphi, A(z) B(w) v\rangle \qquad \text{and} \qquad (-1)^{|A||B|}\langle \varphi, B(w) A(z) v\rangle
	\end{equation*}
	are expansions of one and the same element
	\begin{equation*}
		f_{v, \varphi} \in \CK_{z,w,z-w}
	\end{equation*}
	in $\C\langle\!\langle z \rangle\!\rangle\langle\!\langle w \rangle\!\rangle$ and $\C\langle\!\langle w \rangle\!\rangle\langle\!\langle z \rangle\!\rangle$, respectively, and the order of the pole of $f_{v,\varphi}$ in $z-w$ is uniformly bounded for all $v, \varphi$.
\end{dfn}	
	
As in the theory of vertex algebras, there are several equivalent formulations of mutual locality in terms of commutators%
\footnote{We use the convention that $[x,y] = xy - (-1)^{|x||y|} yx$ denotes the graded commutator.} %
and normal-ordered products, cf. Theorem 2.3 of \cite{Kac} or Propositions 1.2.5 and 3.3.1 of \cite{FBZ}.

\begin{prop}
	\label{prop:locality}
	Let $A(z)$ and $B(w)$ be fields on $\CV$, then the following are equivalent.
	
	\begin{itemize}
		\item[1)] $A(z)$ and $B(w)$ are mutually local.
		\item[2)]There exists $N \geq 0$ such that
		\begin{equation*}
			(z-w)^{N+1} [A(z), B(w)] = 0
		\end{equation*}
		and
		\begin{equation*}
			\displaystyle \bigg(\Omega^m_z - \sum^N_{n=0} (w-z)^n \binom{m+n}{n} \Omega^{m+n}_w\bigg) [A(z), B(w)] = 0
		\end{equation*}
		as elements of $\End(\CV) \otimes \CK_{dist}^{z,w}$.
		\item[3)] There is an identity in $\End(\CV) \otimes \CK^{z,w}_{dist}$
		\begin{equation*}
			[A(z), B(w)] = \sum^N_{n=0}\tfrac{1}{n!} \pd^n_w \Delta(z-w) C^n(w)
		\end{equation*}
		for fields $C^n(w)$.
		\item[4)] The modes $A_{m}$ and $B_{l}$ of $A(z)$ and $B(w)$ have the following commutators
		\begin{equation*}
			[A_{m}, B_{l}] = \begin{cases}
				\displaystyle (-1)^{|A|+1}\sum^N_{n= 0} \binom{m}{n} C^n_{m+l-n} & l,m \geq 0 \\
				\displaystyle \sum^N_{n = \max(0,m+l+1)} \binom{m}{n} C^n_{m+l-n} & m \geq 0, l < 0 \\ 
				\displaystyle (-1)^{|A|+1} \sum^N_{n = \max(0,m+l+1)}\binom{m}{n} C^n_{m+l-n} & m<0, l \geq 0  \\
				\displaystyle 0 & l,m < 0\\
			\end{cases}
		\end{equation*}
		where $C^n_{m}$ are the modes of fields  $C^n(w)$.
		\item[5)] There are identities
		\begin{equation*}
			A(z) B(w) = \norm{A(z) B(w)} + \sum_{n=0}^N \tfrac{1}{n!} \pd^n_w\Delta_-(z-w) C^n(w)
		\end{equation*}
		and
		\begin{equation*}
			(-1)^{|A||B|} B(w) A(z) = \norm{A(z) B(w)} - \sum_{n=0}^N \tfrac{1}{n!} \pd^n_w\Delta_+(z-w) C^n(w)
		\end{equation*}
		in $\End(\CV) \otimes \CK^{z,w}_{dist}$ for fields $C^n(w)$.
	\end{itemize}
\end{prop}

\begin{proof}
	We start by showing the equivalence of $3)$, $4)$, and $5)$. The equivalence of $3)$ and $4)$ follows from taking ravioli residues of the stated commutator. To see $3)$ implies $5)$, we note that the normal-ordered product satisfies
	\beqn
		A(z) B(w) = \norm{A(z) B(w)} + [A(z)_-, B(w)]
	\eeqn
	The presented formula for the commutator $[A(z), B(w)]$ implies
	\beqn
		[A(z)_-, B(w)] = \sum_{n\geq0} \tfrac{1}{n!} \pd^n_w \Delta_-(z-w) C^n(w)
	\eeqn
	yielding the first identity. The second identity is similarly obtained by considering the commutator $[A(z)_+, B(w)]$ and using
	\beqn
		(-1)^{|A||B|}B(w) A(z) = \norm{A(z) B(w)} - [A(z)_+, B(w)]
	\eeqn
	The opposite implication $5) \Rightarrow 3)$ follows by explicitly computing the commutator from the given expressions:
	\beqn
	\begin{aligned}[]
		[A(z), B(w)] & = \sum_{n\geq0}\tfrac{1}{n!} \big(\pd^n_w\Delta_+(z-w) + \pd^n_w\Delta_-(z-w)\big) C^n(w)\\
		& = \sum_{n\geq0}\tfrac{1}{n!} \pd^n_w\Delta(z-w) C^n(w)
	\end{aligned}
	\eeqn
	Thus $3)$, $4)$, and $5)$ are equivalent.
	
	To finish the proof, we show $1)$ implies $2)$, $2)$ implies $3)$, and $5) \equiv 3)$ implies $1)$. Suppose $A(z)$, $B(w)$ are mutually local fields and choose $v \in \CV$ and $\varphi \in \CV^*$. It follows that the matrix elements $\langle \varphi, A(z) B(w) v\rangle$ and $(-1)^{|A||B|} \langle \varphi, B(w) A(z) v \rangle$ are expansions of the same element $f_{v,\varphi}(z,w)$ of $\three$. We write $f_{v,\varphi}$ uniquely as
	\beqn
		f_{v,\varphi} = g_{v, \varphi}(z,w) + \sum_{n = 0}^N \Omega^{n}_{z-w} f^n_{v,\varphi}(w)
	\eeqn
	where $f^n_{v,\varphi}(w) \in \CK^w$  and $g_{v, \varphi}(z,w) \in \CK_{poly}^{z,w}$. The expansion of $f_{v,\varphi}$ in $\C\langle\!\langle z \rangle\!\rangle\langle\!\langle w \rangle\!\rangle$ takes the form
	\beqn
		g_{v, \varphi}(z,w) + \sum_{n = 0}^N \tfrac{1}{n!} \pd^n_w \Delta_{-}(z,w) f^n_{v,\varphi}(w)
	\eeqn
	whereas its expansion in $\C\langle\!\langle w \rangle\!\rangle\langle\!\langle z \rangle\!\rangle$ takes the form
	\beqn
		g_{v, \varphi}(z,w) - \sum_{n = 0}^N \tfrac{1}{n!}\pd^n_w \Delta_{+}(z,w) f^n_{v,\varphi}(w)
	\eeqn
	We conclude that the matrix elements of the commutator are given by
	\beqn
		\langle \varphi, [A(z), B(w)] v \rangle = \sum_{n = 0}^N \tfrac{1}{n!}\pd^n_w \Delta(z,w) f^n_{v,\varphi}(w)
	\eeqn
	In particular, Lemma \ref{lem:deltafunction} implies
	\beqn
	\label{eq:local1}
		(z-w)^{N+1}\langle \varphi, [A(z), B(w)] v \rangle = 0
	\eeqn
	and
	\beqn
	\label{eq:local2}
		\bigg(\Omega^m_z - \sum_{n=0}^N(w-z)^n \binom{m+n}{n} \Omega^{m+n}_w\bigg) \langle\varphi, [A(z), B(w)] v \rangle = 0
	\eeqn
	for every $m$. The requirement that the order of the pole in $(z-w)$ is uniformly bounded in $v, \varphi$ implies that if $N$ is sufficiently large then these equations hold independent of $v, \varphi$ and hence we conclude $A(z)$ and $B(w)$ satisfy condition $2)$.
	
	To show $2)$ implies $3)$, we note that Proposition \ref{prop:locality} and Lemma \ref{lem:deltafunction} imply that the commutator $[A(z),B(w)]$ has the expected form with $C^n(w)$ a general element of $\End(\CV)\otimes \CK_{dist}$, but not necessarily a field. We verify that the $C^n(w)$ are fields as follows. We start by considering the negative modes of the commutator:
	\beqn
		[A(z)_-, B(w)_-] = \sum_{n\geq0} \tfrac{1}{n!}\pd^{n}_{w} \Delta_-(z,w) C^n(w)_-
	\eeqn
	where $C^n(w) = 0$ for $n > N$. Now choose $v \in \CV$. Using the fact that $A$ and $B$ are fields, there exist $M,L \geq 0$ sufficiently large so that
	\beqn
		0 = z^M w^L [A(z)_-, B(w)_-]v = z^M w^L \bigg(\sum_{n\geq0} \tfrac{1}{n!}\pd^{n}_{w} \Delta_-(z,w) C^n(w)_-\bigg)v
	\eeqn
	Expanding $z = w + (z-w)$ and using $(z-w) \pd^{n+1}_w \Delta_-(z,w) = \pd^n_w \Delta_-(z,w)$ and $(z-w) \Delta_-(z,w) = 0$, this equation becomes
	\beqn
		\bigg(\sum_{m=0}^M\sum_{n \geq 0} \frac{1}{(m+n)!} \binom{M}{m} w^{L+M-m}\pd^{n}_{w} \Delta_-(z,w) C^{m+n}(w)_-\bigg)v = 0
	\eeqn
	By linear independence of the $\pd^n_w \Delta_-(z,w)$, we conclude
	\beqn
		\bigg(\sum_{m=0}^M \frac{1}{(m+n)!} \binom{M}{m} w^{L+M-m} C^{m+n}(w)_-\bigg)v = 0
	\eeqn
	for every $n \geq 0$. Consider the case $n = N$. Only the $m = 0$ term contributes because all higher $C^n$ vanish, therefore
	\beqn
		w^{L+M} C^N(w)_- v = 0
	\eeqn
	so that $C^N(w)$ is a field. For $n = N-1$, we find
	\beqn
		\bigg(\frac{1}{(N-1)!} w^{L+M} C^{N-1}(w)_- + \frac{M}{N!} w^{L+M-1} C^{N}(w)_-\bigg)v = 0
	\eeqn
	and by multiplying this equation by $w$ and using the above we see that $C^{N-1}(w)$ is a field:
	\beqn
		w^{L+M+1} C^{N-1}(w)_-v = 0
	\eeqn
	Continuing in this fashion, we conclude
	\beqn
		w^{L+M+N-n} C^{n}(w)_-v = 0
	\eeqn
	and hence all $C^n(w)$ are fields.
	
	Finally, to show that condition $5)$ implies $1)$ we choose $v \in \CV$ and $\varphi \in \CV^*$, leading to the following identities:
	\beqn
	\begin{aligned}
		\langle \varphi, A(z) B(w) v \rangle & = \langle \varphi, \norm{A(z) B(w)}v\rangle\\
		& + (-1)^{|\varphi|} \sum_{n=0}^N \tfrac{1}{n!} \pd^n_w \Delta_-(z-w) \langle \varphi, C^n(w) v\rangle
	\end{aligned}
	\eeqn
	and
	\beqn
	\begin{aligned}
		(-1)^{|A||B|}\langle \varphi, B(w) A(z) v \rangle & = \langle \varphi, \norm{A(z) B(w)}v\rangle\\
		&+ (-1)^{|\varphi|} \sum_{n=0}^N \tfrac{1}{n!} \pd^n_w \Delta_+(z-w) \langle \varphi, C^n(w) v\rangle
	\end{aligned}
	\eeqn
	Based on the form of the normal-ordered product $\norm{A(z)B(w)}$, and using the fact that $A(z)$ and $B(w)$ are fields, it follows that $\langle \varphi, \norm{A(z) B(w)}v\rangle$ belongs to $\CK^{z,w}_{poly}$. To finish the proof, we note that $\Delta^{(n)}_+(z-w)$ and $\Delta^{(n)}_-(z-w)$ are expansions of $n! \Omega^n_{z-w}$ in their respective domains, so that $\langle \varphi, A(z) B(w) v \rangle$ and $(-1)^{|A||B|}\langle \varphi, B(w) A(z) v \rangle$ are both expansions in their respective domains of
	\beqn
		\langle \varphi, \norm{A(z) B(w)}v\rangle + (-1)^{|\varphi|} \sum_{n=0}^N \Omega^n_{z-w} \langle \varphi, C^n(w) v\rangle
	\eeqn
	whence $A(z)$ and $B(w)$ are mutually local.
\end{proof}

Each of these reformulations of locality has an analog in the theory of vertex algebras and the differences are worth pointing out. Conditions $3)$ and $5)$ are direct translations of the analogous statements for vertex algebras, modulo replacing formal delta function $\delta(z-w)$ and its positive/negative parts $\delta_\pm(z-w)$ with their raviolo analogs $\Delta(z-w)$ and $\Delta_\pm(z-w)$, so we do not discuss them. The first relation in condition $2)$ is the same as in vertex algebras, the second is new: its vertex-algebraic analog would read
\beqn
	\bigg(\frac{1}{z^{m+1}} - \sum_{n = 0}^N \binom{m+n}{n} \frac{(w-z)^n}{w^{m+n+1}}\bigg)[A(z), B(w)] = 0
\eeqn
but this Laurent polynomial coefficient is proportional to $(z-w)^N$, so this relation is a consequence of the other. This formulation of locality for vertex-algebraic fields comes directly from properties of the formal delta function $\delta(z-w)$, cf. Lemma 1.1.4 \cite{FBZ} or Corollary 2.2 of \cite{Kac}. That the analogous formulation of locality for raviolo fields requires an extra constraint follows from the analogous properties of the formal raviolo delta function $\Delta(z-w)$, cf. Lemma \ref{lem:deltafunction}.

We finally turn to condition $4)$, where the differences are most dramatic. For reference, we note that the analogous commutators of vertex-algebraic fields $A(z)$, $B(w)$ take the form
\beqn
	[A_m, B_l] = \sum_{n \geq 0} \binom{m}{n} C^n_{m+l-n}
\eeqn
The additional signs $(-1)^{|A|+1}$ when $l \geq 0$ are due to the fact that the ravioli residue pairings $\Res_z \d z$ and $\Res_w \d w$ have cohomological degree $1$; these signs do not appear in ordinary vertex algebras as the usual residue pairing has degree $0$. The truncation of the sum when $m \geq 0$, $l < 0$ or $m < 0$, $l \geq 0$ and the fact that the commutator must vanish when $m,l < 0$ do not appear in vertex algebras and both stem from the relations $z \Omega^0_z = 0$ and $w \Omega_w = 0$. As we will see, the vanishing of these last set of commutators has pretty remarkable consequences.


We now collect some useful properties of mutually local fields for later use. We start by showing that taking derivatives doesn't impact mutual locality.
\begin{prop}
	\label{prop:derivlocality}
	Suppose $A(z)$ and $B(w)$ are mutually local fields on $\CV$, then $\pd^m_z A(z)$ and $\pd^l_w B(w)$ are mutually local for any $m,l \geq 0$.
\end{prop}

\begin{proof}
	We use the formulation of mutual locality given by assertion 3) of Proposition \ref{prop:locality}. In particular, mutual locality implies the existence of fields $C^n(w)$ such that
	\beqn
	[A(z), B(w)] = \sum_n \tfrac{1}{n!} \pd^n_w \Delta(z,w) C^n(w)
	\eeqn
	Using the fact that $\pd_z \Delta(z,w) = - \pd_w \Delta(z,w)$, we see that $\pd_z A(z)$ and $B(w)$ are mutually local because
	\beqn
	[\pd_z A(z), B(w)] = -\sum_n \tfrac{1}{n!} \pd^{n+1}_w \Delta(z,w) C^n(w)
	\eeqn
	Similarly, because that $\pd_w C^n(w)$ is also a field, we see that $A(z)$ and $\pd_w B(w)$ are mutually local. The assertion follows by induction on $m$ and $l$.
\end{proof}

We also have the following raviolo analog of Dong's Lemma, cf. Proposition 2.3.4 of \cite{FBZ} or Proposition 3.2.7 of \cite{Li}:
\begin{lemma}
	\label{lem:Dong}
	Let $A(z)$, $B(x)$, and $C(w)$ be fields on $\CV$ and suppose $A(z)$ is mutually local with $B(x)$ and $C(w)$, then $A(z)$ and $\norm{BC}(w)$ are mutually local. 
\end{lemma}

\begin{proof}
	We compute the commutator between $A(z)$ and $\norm{BC}(w)$. By Lemma \ref{lem:NOP}, this commutator takes the form
	\beqn
	\begin{aligned}[]
		[A(z), \norm{BC}(w)] & = \text{Res}_x \Delta_-(x-w) [A(z), B(x)C(w)]\\
		& \qquad - (-1)^{|B||C|} \text{Res}_x \Delta_+(x-w) [A(z), C(w)B(x)]\\
		& = \text{Res}_x \Delta_-(x-w) [A(z), B(x)] C(w)\\
		& \qquad  - (-1)^{(|A|+|B|)|C|}\text{Res}_x \Delta_+(x-w) C(w) [A(z), B(x)]\\
		& + (-1)^{|A||B|} \text{Res}_x \Delta_-(x-w)B(x) [A(z), C(w)]\\
		& \qquad - (-1)^{|B|(|A|+|C|)} \text{Res}_x \Delta_+(x-w) [A(z), C(w)] B(x)\bigg)\\
	\end{aligned}
	\eeqn
	Mutual locality implies that the commutator $[A(z), B(x)]$ takes the form
	\beqn
	[A(z), B(x)] = \sum_{m \geq 0} \tfrac{1}{m!}\pd^m_x \Delta(z-x) D^m(x)
	\eeqn
	and, similarly,
	\beqn
	[A(z), C(w)] = \sum_{m \geq 0} \tfrac{1}{m!} \pd^m_w \Delta(z-w) E^m(w) 
	\eeqn
	for fields $D^m(x)$ and $E^m(w)$, where $D^m(x) = 0$ and $E^m(w) = 0$ for $m \gg 0$. Using the fact that $(x-w)\Delta_\pm(x-w) = 0$ together with the Taylor formula
	\beqn
	\pd^m_x \Delta(z-x) = \sum_{n \geq 0} \tfrac{1}{n!}(x-w)^n \pd^{m+n}_w\Delta(z-w)
	\eeqn
	we conclude that the commutator $[A(z), \norm{BC}(w)]$ is equal to
	\beqn
	\sum_{m \geq 0}\tfrac{1}{m!}\pd^m_w \Delta(z-w) \bigg(\norm{D^m C}(w) + (-1)^{(|A|+1)|B|}\norm{B E^m}(w)\bigg)
	\eeqn
	whence $A(z)$ and $\norm{BC}(w)$ are mutually local.
\end{proof}

Not only does this proof show $A(z)$ and $\norm{BC}(w)$ are mutually local, it actually gives the following useful expression for the commutator of $A(z)$ and $\norm{BC}(w)$:

\begin{corollary}
	\label{cor:NOPcomm}
	Let $A(z)$, $B(w)$, $C(u)$ be fields on $\CV$ and suppose $A(z)$ is mutually local with respect to $B(w)$ and $C(u)$ with commutators
	\begin{equation*}
		[A(z), B(w)] = \sum_{m\geq0}\tfrac{1}{m!} \pd^m_w \Delta(z-w) D^m(w)
	\end{equation*}
	and
	\begin{equation*}
		[A(z), C(u)] = \sum_{m\geq0}\tfrac{1}{m!} \pd^m_u \Delta(z-u) E^m(u)
	\end{equation*}
	then we have the following identity:
	\begin{equation*}
		[A(z), \norm{BC}(w)] = \sum_{m \geq 0}\tfrac{1}{m!} \pd^m_w \Delta(z-w) \bigg(\norm{D^m C}(w) + (-1)^{(|A|+1)|B|}\norm{B E^m}(w)\bigg)
	\end{equation*}
\end{corollary}



The final property of mutually local fields we present uses the reformulation of mutual locality in Proposition \ref{prop:locality} in terms of the commutators of modes.

\begin{prop}
	\label{prop:assoc}
	Let $A(z)$, $B(w)$ be mutually local fields on $\CV$, then
	\begin{equation*}
		\norm{A B}(z) = (-1)^{|A||B|}\norm{BA}(z)\,.
	\end{equation*} Moreover, given pairwise mutually local fields $A(z)$, $B(w)$, $C(u)$ on $\CV$ we have
	\begin{equation*}
		\norm{A\norm{BC}}(z) = \norm{\norm{AB} C}(z)\,.
	\end{equation*}
\end{prop}

This proposition implies the normal-ordered product defines a commutative and associative product on a collection of mutually local raviolo fields. This is in stark contrast to the theory of vertex algebras, where the normal-ordered product of mutually local fields is generally neither commutative nor associative.
	
As we remarked above, the field $\norm{AB}(z)$ is meant to represent the usual operator product in the underlying quantum field theory. That the normal-ordered product defines a commutative, associative product reflects the fact that this operator product is commutative and associative. The model for local operators in a minimally twisted three-dimensional supersymmetric quantum field theory described by \cite{OhYagi, CostelloDimofteGaiotto-boundary} is as a commutative vertex algebra with some extra structure (e.g. a $1$-shifted $\lambda$-bracket). Although the normal-ordered product in a vertex algebra is generally not associative, this is true for commutative vertex algebras, cf. Section 1.4 of \cite{FBZ}.

\begin{proof}
	Comparing the above formula for the normal-ordered product $\norm{AB}(z)$ to that of $(-1)^{|A||B|}\norm{BA}(z)$, we see the two expressions agree thanks to the commutativity of $A_{m}$, $B_{n}$ for $n,m < 0$:
	\beqn
		\norm{AB}(z) - (-1)^{|A||B|}\norm{BA}(z) = \sum_{m<0} \sum_{n=m}^{-1} z^{-m-1} [A_{m}, B_{m-n-1}] = 0
	\eeqn
	
	In order to verify associativity, we compare the two normal-ordered products. The mode expansion of the left-hand side is given by
	\be
	\begin{aligned}
		\norm{A \norm{B C}}(z) & = \sum_{m\geq0} z^m \bigg(\sum_{l=0}^{m} \sum_{n =0}^{m-l} A_{-l-1} B_{-n-1} C_{l+n-m-1}\bigg)\\
		& + \sum_{m \geq 0} \Omega^m_z\bigg(\sum_{l\geq0} \sum_{n \geq 0} (-1)^{|A||B|}A_{-l-1} B_{-n-1} C_{n+m+l}\\
		& \qquad + (-1)^{|A|+(|B|+1)|C|} A_{-l-1} C_{-n-1} B_{n+m+l}\\
		& \qquad + (-1)^{(|B|+|C|)(|A|+1)}B_{-l-1}C_{-n-1} A_{n+m+l}\bigg)
	\end{aligned}
	\ee
	whereas the right-hand side is given by
	\be
	\begin{aligned}
		\norm{\norm{AB} C}(z) & = \sum_{m\geq0} z^m \bigg(\sum_{n=0}^{m} \sum_{l=0}^{n} A_{-l-1} B_{l-n-1} C_{n-m-1}\bigg)\\
		& + \sum_{m \geq 0} \Omega^m_z\bigg(\sum_{l\geq0} \sum_{n \geq 0} (-1)^{|A||B|}A_{-l-1} B_{-n-1} C_{n+m+l}\\
		& \qquad + (-1)^{|A|+(|A|+|B|+1)|C|} C_{-l-1} A_{-n-1} B_{n+m+l}\\
		& \qquad + (-1)^{(|B|+|C|)(|A|+1) + |B||C|}C_{-l-1}B_{-n-1} A_{n+m+l}\bigg)
	\end{aligned}
	\ee
	Again, these two expressions are equal to one another due to the commutativity of the $A_{n}$, $B_{m}$, $C_{l}$ for $n, m, l <0$.
\end{proof}

\subsection{The main definition}
We now move to the definition of a \emph{raviolo vertex algebra}. 
From the point of view of quantum field theory, this structure is parallel to that of a vertex algebra.
A vertex algebra is an algebraic structure modeling the local operators of a chiral/holomorphic field theory in two real dimensions.

Likewise, our motivation for the definition of a raviolo vertex algebra is the algebraic structure underpinning the local operators of a three-dimensional quantum field theory which is partially holomorphic and partially topological.
Such theories arise naturally from twists of three-dimensional supersymmetric theories \cite{Closset:2012ru, Closset:2013vra, OhYagi, CostelloDimofteGaiotto-boundary}.

The essential data of a vertex algebra is the so-called \textit{state-operator} (or \textit{state-field}) correspondence which is a linear assignment from $\CV$ to the space of fields on $\CV$.
With our modified notion of field, i.e. a raviolo field, we can apply the same definition in our setting.
We require that the state-operator correspondence
\beqn
	Y(-,z) \colon \CV \to \CF_{rav}(\CV) \subset \End(\CV)\otimes \CK_{dist}
\eeqn
be compatible with the gradings on $\CV$ and $\CK_{dist}$, i.e. if $a \in \CV^r$ then $Y(a,z)$ is a homogeneous raviolo field of cohomological degree $r$, so that $Y(-,z)$ is a homogeneous linear map of degree $0$. We write the mode expansion of this field as
\beqn
	Y(a,z) = \sum_{m<0} z^{-m-1} a_{(m)} + \sum_{m \geq0} \Omega^m_z a_{(m)}
\eeqn
By a common abuse of notation we will sometimes denote the field $Y(a,z)$ by $a(z)$.

\begin{dfn}
	A \defterm{raviolo vertex algebra} is the data $(\CV, |0\rangle, \del, Y)$ where 
	\begin{itemize}
	\item $\CV = \bigoplus \CV^r$ is a $\Z$ graded vector space.
	\item $|0\rangle \in \CV^0$ is a distinguished element (the vacuum vector).
	\item $\del \colon \CV \to \CV$ is an endomorphism of degree $0$ (the translation operator).
	\item $Y = Y(-,z) \colon \CV \to \End(\CV) \otimes \CK_{dist}$ is a linear map of degree $0$.
	\end{itemize}
	This data is required to satisfy the following axioms
	\begin{enumerate}
	\item For every $a \in \CV$ the element $Y(a,z)$ is a raviolo field as in Definition \ref{dfn:field} (state-field correspondence).
	\item One has
	\beqn
	[\pd, Y(a,z)] = \pd_z Y(a,z)
	\eeqn
	for every $a \in \CV$ (translation axiom).
	\item The vacuum vector satisfies $\del |0 \rangle = 0$, $Y(|0\rangle, z) = \id_\CV$, and for any $a \in \CV$ one has $Y(a,z) |0\rangle \in \CV[\![z]\!],$ and $Y(a,z=0)|0\rangle = a$ (vacuum axiom).
	\item For every $a,b \in \CV$ the fields $Y(a,z),Y(b,w)$ are mutually local as in Definition \ref{dfn:mutual} (locality axiom).
	\end{enumerate}

	We say that a $\pd$-invariant subspace $\CW \subset \CV$ containing the vacuum vector $|0\rangle \in \CW$ is a \defterm{raviolo vertex subalgebra} of $\CV$ if $Y(a,z) b \in \CW\langle\!\langle z \rangle\!\rangle$ for any $a,b \in \CW$. A $\pd$-invariant subspace $\CI$ is called a \defterm{raviolo vertex algebra ideal} of $\CV$ if $Y(a,z) b \in \CI\langle\!\langle z \rangle\!\rangle$ for any $a \in \CI$ and $b \in \CV$.

	A \defterm{morphism} of raviolo vertex algebras
	\[
		\rho \colon (\CV_1, |0\rangle_1, \pd_1, Y_1) \to (\CV_2, |0\rangle_2, \pd_2, Y_2)
	\]
	is the data of a linear map $\rho: \CV_1 \to \CV_2$ of degree $0$ that preserves the vacuum $\rho(|0\rangle_1) = |0\rangle_2$, and intertwines the translation operators $\pd_2 \rho(a) = \rho(\pd_1 a)$ and state-operator correspondences $\rho(Y_1(a,z) b) = Y_2(\rho(a),z) \rho(b)$.
\end{dfn}

The kernel of a morphism $\rho: \CV_1 \to \CV_2$ is an ideal of $\CV_1$ and the image is a subalgebra of $\CV_2$. As with vertex algebras, there is a skew-symmetry property (Proposition \ref{prop:skew}) ensuring that if $\CI$ is an ideal of $\CV$ then $Y(b,z)a \in \CI\langle\!\langle z \rangle\!\rangle$ for any $a \in \CI$ and $b \in \CV$. In particular, the quotient $\CV/\CI$ has the structure of a raviolo vertex algebra. We also note that if $\CV_1$ and $\CV_2$ are raviolo vertex algebras, then the tensor product $\CV_1 \otimes \CV_2$ naturally inherits the structure of a raviolo vertex algebra.

\begin{dfn}
	A \defterm{derivation} $D$ (of degree $r$) of a raviolo vertex algebra $\CV$ is a (homogeneous) linear map $D: \CV \to \CV$ (of degree $r$) such that
	\begin{equation*}
		[D, Y(a, z)] = Y(Da, z)
	\end{equation*}
	as fields on $\CV$. A degree $1$, square-zero derivation $D$ is called a \defterm{differential} on $\CV$ and we call the pair $(\CV, D)$ a \defterm{differential-graded raviolo vertex algebra} or \defterm{dg raviolo vertex algebra}.
\end{dfn}

For example, the translation operator $\pd$ is always a derivation. We note that if $D$ is a derivation of $\CV$ then the kernel of $D$ is a subalgebra of $\CV$. If $D$ is a differential, then the image of $D$ is an ideal of the kernel of $D$. In particular, the cohomology of a differential $D$ has the structure of a raviolo vertex algebra.

\subsection{Additional gradings}
The underlying vector space $\CV$ often has additional gradings that are compatible with the axioms of a raviolo vertex algebra; correspondingly, the space of fields $\CF_{rav}(\CV)$ inherits natural gradings from $\End(\CV)$.

Let $L > 0$ be an integer.
A \textit{spin grading} on a raviolo vertex algebra is an additional (non-cohomological) grading 
\beqn
	\CV = \oplus_{s \in \frac{1}{L}\Z} \CV^{(s)} 
\eeqn
such that the vacuum vector $|0\rangle$ and $Y$ are weight $0$, the translation operator $\del \colon \CV^{(s)} \to \CV^{(s+1)}$ increases the weight by $+1$, and for any $a \in \CV^{(s)}$ the mode $a_{(n)}$ increases weight by $s-n-1$. We call the weight with respect to such a grading \emph{spin}.
This grading does not appear in any Koszul rule of signs.
Typically $L = 1,2,$ or $4$ in the examples that we consider.

It will be useful to allow for an additional $\Z/2$ grading, which amounts to working with graded super vector spaces.
That is, for each cohomological degree $r \in \Z$ we have a super vector space 
\be
	\CV^r = \CV_+^r \oplus \Pi \CV_-^r
\ee
where $\CV_+^r$ is the \emph{even} part and $\CV^r_-$ is the \emph{odd} part.
We assume that the vacuum $|0\rangle$, the state-operator correspondence $Y$, and the translation operator $\del$ are even with respect to the super grading.
Additionally, the Koszul rule of signs must be modified; in any place where the degree $|a|$ of an element enters in an algebraic expression it must be understood as the totalized grading. 
That is, if $a \in \CV^r_{+}$ then $|a| = r \mod 2$, but if $a \in \CV_-^r$ then $|a| = (r+1) \mod 2$. We say that a state/field is \emph{bosonic} (resp. \emph{fermionic}) if it has even (resp. odd) totalized grading and call such states/fields \emph{bosons} (resp. \emph{fermions}).

If $\CV$ has a spin grading then we can consider the following graded character as a formal $q$-series 
\beqn
	\text{ch}_{\CV}(q) \define \sum_{s \in \frac{1}{L} \Z} \text{grdim}(\CV^{(s)}) q^s
\eeqn
where $\text{grdim}(\CV^{(s)})$ is the dimension of the vector space spanned by bosonic (with respect to the totalized grading) elements of spin $s$ minus the dimension of the vector space spanned by the fermionic elements of spin $s$.

\subsection{Raviolo vertex algebras over graded commutative unital $\C$-algebras}
\label{sec:commCalgs}
For this last subsection, fix a graded commutative unital $\C$-algebra $S$. We introduce the notion of a raviolo vertex algebra over $S$; there should be a similar construction when $S$ is replaced by a general graded commutative ring, but we will not need that level of generality. See \cite{Mason} for the analogous construction in the theory of vertex algebras. In brief, a raviolo vertex algebra over $S$ is merely a graded $S$-module $\CV = \bigoplus \CV^r$ equipped with a compatible raviolo vertex algebra structure.

The notion of raviolo field has a natural incarnation over $S$ by simply replacing linear maps by $S$-module morphisms.

\begin{dfn}
\label{dfn:Sfield}
	A \defterm{raviolo field over $S$} on $\CV$ is an element
	\begin{equation*}
		A(z) = \sum_{m<0} z^{-m-1} A_m + \sum_{m\geq0} \Omega^m_z A_m \in \End_S(\CV) \otimes \CK_{dist}
	\end{equation*}
	such that for every $v \in \CV$ there exists $N$ sufficiently large with
	\begin{equation*}
		A_m v = 0\, \qquad \text{for all } m \geq N
	\end{equation*}
	A raviolo field $A(z)$ over $S$ is \defterm{homogeneous} of degree $|A|$ if $A_m$ is an $S$-module morphism of degree $|A|$ for $m < 0$ and degree $|A|-1$ for $m \geq 0$.
\end{dfn}

The space of raviolo fields over $S$ on $\CV$ will be denoted $\CF_{rav/S}(\CV)$. There is a notion of mutual locality of two raviolo fields over $S$; we use the algebraic reformulation of mutual locality provided by Proposition \ref{prop:locality}.

\begin{dfn}
\label{dfn:Smutual}
	Two raviolo fields $A(z)$, $B(w)$ over $S$ are \defterm{mutually local} if there exists $N \geq 0$ such that
	\begin{equation*}
		(z-w)^{N+1} [A(z), B(w)] = 0
	\end{equation*}
	and
	\begin{equation*}
		\displaystyle \bigg(\Omega^m_z - \sum^N_{n=0} (w-z)^n \binom{m+n}{n} \Omega^{m+n}_w\bigg) [A(z), B(w)] = 0
	\end{equation*}
	as elements of $\End_S(\CV) \otimes \CK_{dist}^{z,w}$.
\end{dfn}

With the notion of mutually local raviolo fields over $S$ on an $S$-module $\CV$, we can now state the definition of a raviolo vertex algebra over $S$.

\begin{dfn}
	A \defterm{raviolo vertex algebra over $S$} is the data $(\CV, |0\rangle, \del, Y)$ where 
	\begin{itemize}
		\item $\CV = \bigoplus \CV^r$ is a graded $S$-module.
		\item $|0\rangle \in \CV^0$ is a distinguished element (the vacuum vector).
		\item $\del \colon \CV \to \CV$ is an $S$-module morphism of degree $0$ (the translation operator).
		\item $Y = Y(-,z) \colon \CV \to \End_S(\CV) \otimes \CK_{dist}$ is a $S$-module morphism of degree $0$.
	\end{itemize}
	This data is required to satisfy the following axioms
	\begin{enumerate}
		\item For every $a \in \CV$ the element $Y(a,z)$ is a raviolo field over $S$ as in Definition \ref{dfn:Sfield} (state-field correspondence).
		\item One has
		\beqn
		[\pd, Y(a,z)] = \pd_z Y(a,z)
		\eeqn
		for every $a \in \CV$ (translation axiom).
		\item The vacuum vector satisfies $\del |0 \rangle = 0$, $Y(|0\rangle, z) = \id_\CV$, and for any $a \in \CV$ one has $Y(a,z) |0\rangle \in \CV[\![z]\!],$ and $Y(a,z=0)|0\rangle = a$ (vacuum axiom).
		\item For every $a,b \in \CV$ the fields $Y(a,z),Y(b,w)$ are mutually local as in Definition \ref{dfn:Smutual} (locality axiom).
	\end{enumerate}
	
	We say that a $\pd$-invariant $S$-submodule $\CW \subset \CV$ containing the vacuum vector $|0\rangle \in \CW$ is a \defterm{raviolo vertex subalgebra over $S$} if for any $Y(a,z) b \in \CW\langle\!\langle z \rangle\!\rangle$ for any $a,b \in \CW$. A $\pd$-invariant $S$-submodule $\CI$ is called a \defterm{raviolo vertex algebra ideal over $S$} if $Y(a,z) b \in \CI\langle\!\langle z \rangle\!\rangle$ for any $a \in \CI$ and $b \in \CV$.
	
	A \defterm{morphism} of raviolo vertex algebras over $S$
	\[
		\rho \colon (\CV_1, |0\rangle_1, \pd_1, Y_1) \to (\CV_2, |0\rangle_2, \pd_2, Y_2)
	\]
	is the data of an $S$-module morphism $\rho: \CV_1 \to \CV_2$ of degree $0$ that preserves the vacuum $\rho(|0\rangle_1) = |0\rangle_2$, and intertwines the translation operators $\pd_2 \rho(a) = \rho(\pd_1 a)$ and state-operator correspondences $\rho(Y_1(a,z) b) = Y_2(\rho(a),z) \rho(b)$.
\end{dfn}

There are natural variants of the above that allow for spin and super gradings. Although it is not strictly necessary, our examples are such that $S$ has trivial spin and super gradings, i.e. all elements of $S$ are spin $0$ and even, but can have support in non-trivial cohomological degree.

It is easy to characterize when a raviolo vertex algebra is a raviolo vertex algebra over $S$, but we must put off the proof until the end on Section \ref{sec:elementary} when we have more tools.
\begin{prop}
\label{prop:Sraviolo}
	Let $\CV$ be a raviolo vertex algebra, then $\CV$ is a raviolo vertex algebra over $S$ if and only if $\CV$ is a graded $S$-module such that the $S$-action commutes with $\pd$ and $Y(a,z)$ for all $a \in \CV$.
\end{prop}

In fact, the proof of this result allows us to characterize the action of $S$ on general states in $\CV$, and their corresponding fields, in terms of its action on the vacuum $|0\rangle$ via the state-operator correspondence.

\begin{corollary}
	Let $\CV$ be a raviolo vertex algebra over $S$. The linear map $\kappa^\CV$ representing the action of $\kappa \in S$ on $\CV$ is equal to the constant field $Y(\kappa^\CV |0\rangle, z)$. The action of $\kappa \in S$ on the field $Y(a,z)$ is via normal-ordered product with $Y(\kappa^\CV |0\rangle, z)$.
\end{corollary}

From this perspective, a subalgebra of a raviolo vertex algebra over $S$ is a raviolo vertex subalgebra that is also an $S$-submodule, i.e. it is preserved by the action of $S$. Suppose $I$ is an ideal of $S$. If $\CV$ is a raviolo vertex algebra over $S$, then there is a natural ideal $\CI_I$ of $\CV$ spanned (as an $S$-module) by finite sums of elements of the form $\kappa a$ with $\kappa \in I$ and $a \in \CV$. The quotient $\CV/\CI_I$ is then naturally a raviolo vertex algebra over the quotient $S/I$.

\section{Elementary properties of raviolo vertex algebras}
\label{sec:elementary}
We now prove some elementary properties of raviolo vertex algebras. We stress that the statement of many of our results is a direct translation from the theory of vertex algebras, but often the proofs of these results provided in, e.g., \cite{Kac, FBZ} must be modified due to the fact that the ring of formal raviolo Laurent series $\CK = \C\langle\!\langle z \rangle\!\rangle$ has zero-divisors, whereas $\C(\!(z)\!)$ is a field.

\subsection{Rigidity of the state-operator correspondence}
We start with a raviolo analog of Goddard's uniqueness theorem, cf. Theorem 3.1.1 of \cite{FBZ} or \cite{Goddard}.
\begin{prop}
\label{prop:uniqueness}
	Let $\CV$ be a raviolo vertex algebra and $A(z)$ a field on $\CV$. If $A(z)$ is mutually local with $Y(b,z)$ for all $b \in \CV$ and there exists $a \in \CV$ with
	\begin{equation*}
		A(z)|0\rangle = Y(a,z)|0\rangle
	\end{equation*}
	then $A(z) = Y(a,z)$.
\end{prop}

\begin{proof}
	As in the proof of Theorem 3.1.1 in \cite{FBZ}, mutual locality implies
	\beqn
		(z-w)^{N}A(z)Y(b,w)|0\rangle = (z-w)^{N}Y(a,z) Y(b,w)|0\rangle
	\eeqn 
	for $N$ sufficiently large. The vacuum axiom implies these expressions are well-defined at $w = 0$; using $Y(b,0)|0\rangle = b$, we conclude
	\beqn
		z^{N} A(z) b = z^{N} Y(a,z) b
	\eeqn for all $b \in \CV$. The exponent $N$ varies as we vary $b$, but we can nonetheless conclude that $A(z)_+ b = Y(a,z)_+ b$ for all $b$, and hence $A(z)_+ = Y(a,z)_+$.
	
	In order to show that the remaining modes agree, we start with the following observation:
	\beqn
		[A(z) - Y(a,z), Y(b,w)]|0\rangle = (A(z)_--Y(a,z)_-)Y(b,w)|0\rangle
	\eeqn
	which follows from $A(z)_+ = Y(a,z)_+$ and $A(z)|0\rangle = Y(a,z)|0\rangle$; the vacuum axiom implies that this expression has no terms proportional to $\Omega^m_w$. On the other hand, mutual locality implies that
	\beqn
		[A(z) - Y(a,z), Y(b,w)] = \sum_{m \geq 0} \tfrac{1}{m!} \pd^m_w \Delta(z-w) C^m(w)
	\eeqn
	for some fields $C^m(w)$, from which it follows that
	\beqn
		(A(z)_--Y(a,z)_-)Y(b,w)|0\rangle = \sum_{m \geq 0} \tfrac{1}{m!} \pd^m_w \Delta(z-w) C^m(w)|0\rangle
	\eeqn
	As the left-hand side does not contain terms proportional to $\Omega^n_w$, we conclude
	\beqn
		\sum_{m \geq 0} \tfrac{1}{m!} \bigg(\pd^m_w \Delta_+(z-w) C^m(w)_+ - \pd^m_w \Delta_-(z-w) C^m(w)_-\bigg)|0\rangle = 0
	\eeqn
	The distributions $\pd^m_w\Delta_+(z-w) w^n$ and $\pd^m_w\Delta_-(z-w) \Omega^n_w$ are all linearly independent, so we conclude $C^m(w)|0\rangle = 0$. In particular,
	\beqn
		(A(z)_- - Y(a,z)_-)Y(b,w)|0\rangle = 0 \rightsquigarrow (A(z)_--Y(a,z)_-)b = 0
	\eeqn
	whence $A(z)_- = Y(a,z)_-$.
\end{proof}

\subsection{Relations involving the translation operator}
Using this uniqueness result, we can prove the following properties describing the action of $\pd$; see Lemma 3.1.3, Corollary 3.1.6, and Lemma 3.2.3 of \cite{FBZ} for the analogous results in the theory of vertex algebras.

\begin{lemma}
\label{lem:translation}
	Let $\CV$ be a raviolo vertex algebra, then for all $a \in \CV$ we have
	\begin{itemize}
		\item[1)] $Y(a,z)|0\rangle = e^{z \pd} a$ as elements of $\CV[\![z]\!]$
		\item[2)] $Y(\pd a,z) = \pd_z Y(a,z)$ as fields on $\CV$
		\item[3)] $e^{-w \pd} Y(a,z) e^{w \pd} = Y(a,z-w)$ as elements of $\End(\CV)\otimes \CK^{z,w}_{dist}$
	\end{itemize}
\end{lemma}

\begin{proof}
	The proofs of Lemma 3.1.3, Corollary 3.1.6, and Lemma 3.2.3 in \cite{FBZ} transfer with little modification to raviolo vertex algebras to prove assertions $1)$, $2)$, and $3)$, respectively. The first assertion follows from the vacuum and translation axioms; the second is a consequence of Proposition \ref{prop:uniqueness} applied to $A(z) = \pd_z Y(a,z)$; the third is another consequence of the translation axiom and the Taylor formula.
\end{proof}

The raviolo analog of Goddard's uniqueness theorem together with property 1) in Lemma \ref{lem:translation} implies that the field $Y(a,z)$ is the unique raviolo field on $\CV$ satisfying the differential equation
\beqn
	\pd_z A(z) |0\rangle = \pd A(z)|0\rangle
\eeqn
subject to the initial condition $A(z)|0\rangle|_{z=0} = a$, cf. Remark 3.1.5 of \cite{FBZ}.

The second property we verify is the raviolo analog of the skew-symmetry property of a vertex algebra, cf. Proposition 3.2.5 of \cite{FBZ}.

\begin{prop}
\label{prop:skew}
	Let $\CV$ be a raviolo vertex algebra, then for any $a, b \in \CV$
	\begin{equation*}
		Y(a,z) b = (-1)^{|a||b|} e^{z \pd} Y(b,-z) a
	\end{equation*}
	as elements of $\CV\langle\!\langle z \rangle\!\rangle$.
\end{prop}
\begin{proof}
	Mutual locality of $Y(a,z)$ and $Y(b,z)$ ensures that there exist fields $C^n(w)$ such that
	\beqn
		Y(a,z) Y(b,w)|0\rangle = (-1)^{|a||b|} Y(b,w) Y(a,z)|0\rangle + \sum_{n \geq 0} \tfrac{1}{n!} \pd^n_w \Delta(z,w) C^n(w)|0\rangle
	\eeqn
	We then apply the first assertion in Lemma \ref{lem:translation} to both sides to conclude
	\beqn
		Y(a,z) e^{w \pd} b = (-1)^{|a||b|} Y(b,w) e^{z \pd} a + \sum_{n \geq 0} \tfrac{1}{n!}\pd^n_w \Delta(z,w) C^n(w)|0\rangle
	\eeqn
	This equality implies the following four equalities:
	\beqn
	\begin{aligned}
		Y(a,z)_+ e^{w \pd} b & = (-1)^{|a||b|} Y(b,w)_+ e^{z \pd} a\\
		Y(a,z)_- e^{w \pd} b & = \sum_{n\geq0} \tfrac{1}{n!} \pd^n_w \Delta_-(z,w) C^n(w)_+|0\rangle\\
		0 & = (-1)^{|a||b|}Y(b,w)_- e^{z \pd}a - \sum_{n\geq0} \tfrac{1}{n!} \pd^n_w \Delta_+(z,w) C^n(w)_+|0\rangle\\
		0 & = - \sum_{n\geq0}\tfrac{1}{n!} \pd^n_w \Delta_-(z,w) C^n(w)_-|0\rangle\\
	\end{aligned}
	\eeqn
	
	Applying the third assertion of Lemma \ref{lem:translation} to the first equation and setting $w=0$ leads to
	\beqn
		Y(a,z)_+ b = (-1)^{|a||b|} e^{z \pd} Y(b,-z)_+ a
	\eeqn
	Setting $w=0$ in the second equation, and using $\pd^n_w \Delta_-(z,w)|_{w=0} = n! \Omega^n_z$, gives us
	\beqn
		Y(a,z)_- b = \sum_{n\geq0} \Omega^n_z C^n_{(-1)}|0\rangle \rightsquigarrow a_{(n)}b = C^n_{-1}|0\rangle
	\eeqn
	We cannot set $w=0$ in the third equality due to the factors of $\Omega^m_w$, but we can once again apply the third assertion of Lemma \ref{lem:translation} to turn it into
	\beqn
		(-1)^{|a||b|}e^{z \pd}Y(b,w-z)_- a = \sum_{n\geq0} \tfrac{1}{n!} \pd^n_w \Delta_+(z,w) C^n(w)_+|0\rangle
	\eeqn
	where we identify $m! \Omega^m_{w-z}$ with the expansion $(-1)^m \pd^m_w \Delta_+(z,w)$. The left-hand side of this equality is thus equal to
	\beqn
		(-1)^{|a||b|} \sum_{m,l\geq 0} \frac{z^l(-1)^m}{l!m!} \pd^m_w \Delta_+(z,w) \pd^l b_{(m)}a
	\eeqn
	expanding $z = (z-w) + w$ and taking the term proportional to $w^0$ gives the equality
	\beqn
		(-1)^{|a||b|} \sum_{n,l\geq 0} \frac{(-1)^{n+l}}{l!n!} \pd^m_w \Delta_+(z,w) \pd^l b_{(n+l)}a
	\eeqn
	As the distributions $w^i \pd^j_w \Delta_+(z,w)$ are linearly independent, we can equate the coefficients of $\pd^n_w \Delta_+(z,w)$ on both sides of the above equation to find
	\beqn
		(-1)^{|a||b|+n} \sum_{l\geq 0} \frac{(-1)^l}{l!}\pd^l b_{(n+l)}a = C^n_{-1}|0\rangle = a_{(n)} b
	\eeqn
	for all $n \geq0$. Put together, these equations say 
	\beqn
		(-1)^{|a||b|}e^{z \pd} Y(b,-z)_- a = Y(a,z)_- b
	\eeqn
	as desired.
\end{proof}

\subsection{Associativity and the operator product expansion}
The definitions we have given and results we have proven provide obvious parallels with the theory of vertex algebras. As we have seen, many of the manipulations one performs in a vertex algebra transfer mutatis mutundis to raviolo vertex algebras. The last and most important property we consider is the notion of associativity.

\begin{theorem}
\label{thm:assoc}
	Let $\CV$ be a raviolo vertex algebra and choose $a,b,c \in \CV$, then the three expressions
	\begin{equation*}
		\begin{aligned}
			Y(a,z) Y(b,w) c & \in \CV\langle\!\langle z \rangle\!\rangle\langle\!\langle w \rangle\!\rangle\,,\\
			(-1)^{|a||b|} Y(b,w) Y(a,z) c & \in \CV\langle\!\langle w \rangle\!\rangle\langle\!\langle z \rangle\!\rangle\,, \textit{ and}\\
			Y(Y(a,z-w)b,w)  c & \in \CV\langle\!\langle w \rangle\!\rangle\langle\!\langle z-w \rangle\!\rangle\\
		\end{aligned}
	\end{equation*}
	are expansions, in their respective domains, of the same element of $\CV \otimes \three$.
\end{theorem}

\begin{proof}
	The proof of Theorem 3.2.1 of \cite{FBZ} transfers without issue. For completeness, we sketch the argument. First, mutual locality of $Y(a,z)$ and $Y(b,w)$ implies the assertion that the first two expressions are expansions of the same element. We now show that the first and third expressions are expansions of the same element.
	
	Proposition \ref{prop:skew} and property 3) in Lemma \ref{lem:translation} together imply
	\beqn
		Y(a,z) Y(b,z) c = (-1)^{|b||c|}e^{w \pd} Y(a,z-w) Y(c,-w) b
	\eeqn
	cf. Eq. (3.2.4) of \cite{FBZ}. Similarly, Proposition \ref{prop:skew} implies
	\beqn
		Y(Y(a,z-w)b,w)c = (-1)^{(|a|+|b|)|c|}e^{w \pd} Y(c,-w) Y(a,z-w)b
	\eeqn
	cf. Eq. (3.2.5) of \cite{FBZ}. Mutual locality of $Y(a,z-w)$ and $Y(c,-w)$ then implies that $Y(a,z) Y(b,z) c$ and $Y(Y(a,z-w)b,w)c$ are expansions of the same element.
\end{proof}

As with vertex algebras, we often abbreviate the above result by saying $Y(a,z)Y(b,w)c$ and $Y(Y(a,z-w)b,w)c$ are equal to one another, with the understanding that they are expansions of the same element of $\three$ in different regions. Phrased differently, we find that for any $c$
\beqn
	Y(Y(a,z-w)b)c = \bigg(\sum_{m<0} (z-w)^{-m-1} Y(a_{(m)}b,w) + \sum_{m\geq0} \Omega^m_{z-w} Y(a_{(m)}b,w)\bigg)c
\eeqn
is also equal, by Proposition \ref{prop:locality} and Taylor expanding $Y(a,z)$, to
\beqn
	Y(a,z)Y(b,w)c = \bigg(\sum_{m \geq 0}\tfrac{1}{m!}(z-w)^{m} \norm{\pd^m_wY(a,w) Y(b,w)} + \sum_{m\geq0} \Omega^m_{z-w}A^m(w)\bigg)c
\eeqn
for some fields $A^m(w)$. Comparing the coefficients of $(z-w)^m$ and $\Omega^m_{z-w}$ in these two expressions leads us to the following equality (with the usual caveat that the two sides are expansions in different regions)
\be
\label{eq:OPE}
\begin{aligned}
	Y(a,z)Y(b,w) & = \sum_{n < 0} (z-w)^{-n-1} Y(a_{(n)} b,w) + \sum_{n \geq 0} \Omega_{z-w}^n Y(a_{(n)} b, w)\\
	& = \norm{Y(a,z) Y(b,w)} + \sum_{n \geq 0} \Omega_{z-w}^n Y(a_{(n)} b, w) \\
\end{aligned}
\ee
We call this is the \emph{operator product expansion} (OPE) of the fields $Y(a,z)$ and $Y(b,w)$ and call coefficients of $\Omega_{z-w}^n$ the \emph{singular terms} of the OPE and the remaining terms \emph{regular}. Proposition \ref{prop:skew} implies that the OPE satisfies the following skew symmetry relation:
\be\label{eqn:skewsymmetry}
\begin{aligned}
	Y(b,z) Y(a,w) & = \norm{Y(b,z) Y(a,w)}\\
	& + \sum_{n \geq 0} \Omega_{z-w}^n \bigg((-1)^{|a||b|}\sum_{l\geq0} \frac{(-1)^{n+l}}{l!} Y(\pd^l a_{(n+l)}b, w) \bigg)
\end{aligned}
\ee
Following standard vertex algebra conventions, for brevity we will sometimes omit the regular terms in the OPE and write
\be
	Y(a,z) Y(b,w) \sim \sum_{n \geq 0} \Omega_{z-w}^n Y(a_{(n)}b, w)
\ee
Of course, the notation $\sim$ is borrowed from the theory of vertex algebras and refers to equality up to regular terms. 
As all the $Y(a,z)$ are fields, there is at most a finite number of singular terms in the OPE of any two fields.

An immediate consequence of this analysis, and hence the associativity property established in Theorem \ref{thm:assoc}, is the following result showing that the state-operator correspondence for a general state is uniquely characterized by those of ``simpler'' states.
\begin{corollary}
\label{cor:stateopNOP}
	For any $a^1, \dots, a^l \in \CV$ and $j_1, \dots, j_l<0$, we have
	\begin{equation*}
	\begin{aligned}
		& Y(a^1_{(j_1)} \dots a^l_{(j_l)}|0\rangle,z) =\\
		& \qquad \frac{1}{(-j_1-1)! \dots (-j_l-1)!} \norm{\pd_z^{-j_1-1}Y(a^1,z) \dots \pd^{-j_l-1}_zY(a^l,z)}
	\end{aligned}
	\end{equation*}
\end{corollary}

\begin{proof}
	We prove this by induction on $l$. The case of $l = 1$ follows from Lemma \ref{lem:translation}. The general case follows from the relation
	\beqn
		Y(a_{(-n-1)}b,w) = \frac{1}{n!}\norm{\pd^n_w Y(a,w) Y(b,w)}
	\eeqn
	derived above for general $a,b$.
\end{proof}

Another immediate corollary of our OPE formula and the expression in Proposition \ref{prop:locality} for the commutators of the modes of two mutually local fields.
\begin{corollary}
\label{cor:derivation}
	The linear map $a_{(0)}$ is a derivation of $\CV$ for any $a \in \CV$.
\end{corollary}
\begin{proof}
	Identify $C^0(w)$ in Proposition \ref{prop:locality} with $Y(a_{(0)}b,z)$, we find
	\beqn
		[a_{(0)}, Y(b,z)] = Y(a_{(0)}b,z)
	\eeqn
\end{proof}

To finish this subsection, we return to the proof of Proposition \ref{prop:Sraviolo} providing an alternative description of raviolo vertex algebras over a graded commutative unital $\C$-algebra $S$.

\begin{proof}[Proof of Proposition \ref{prop:Sraviolo}]
	Suppose $\CV$ is a raviolo vertex algebra over $S$. As any graded $S$-module is itself a graded vector space (over $\C$), it is clear that the graded vector space underlying $\CV$ is simultaneously an $S$-module and a raviolo vertex algebra. By an abuse of notation, we denote the data (state-operator correspondence, vacuum vector, translation operator) by the same symbols. The fact that the $S$-action on $\CV$ commutes with $\pd$ and $Y(a,z)$ for any $a \in \CV$ is due to the requirement that $\pd$ and the modes $a_{(m)}$ are morphisms of $S$-modules. To see that the action of $\kappa \in S$ is realized by a field, let $\kappa^\CV$ be the linear map realizing action of $\kappa$ on $\CV$, then
	\beqn
	\kappa^\CV = \kappa^\CV \id_\CV = \kappa^\CV Y(|0\rangle, z) = Y(\kappa^\CV |0\rangle, z)
	\eeqn
	where we used the vacuum axiom and the fact that $Y(-,z): \CV \to \End_S(\CV)\otimes \CK_{dist}$ is a morphism of $S$-modules. The fact that $Y(\kappa^\CV |0\rangle, z)$ is constant is due to the translation axiom and the fact that $\pd$ is a morphism of $S$-modules.
	
	Now suppose $\CV$ is both a raviolo vertex algebra and an $S$-module such that the $S$-action commutes with $\pd$ and $Y(a,z)$ for all $a \in \CV$. The former implies that $\pd$ is a morphism of $S$-modules. The latter implies that all of the modes $a_{(m)}$ are morphisms of $S$-modules, hence $Y(a,z)$ is a raviolo field over $S$ for all $a \in \CV$. Now choose $\kappa \in S$, we again denote by $\kappa^\CV$ the linear map representing the action of $\kappa$ on $\CV$. Consider the vector $\kappa^\CV |0\rangle$, then the corresponding field is necessarily constant because
	\beqn
	\pd_z Y(\kappa^\CV|0\rangle,z) = Y(\pd \kappa^\CV|0\rangle, z) = Y(\kappa^\CV\pd|0\rangle, z) = 0
	\eeqn
	where we used property 2) in Lemma \ref{lem:translation} in the first equality, that $\pd$ commutes with the $S$-action in the second, and the vacuum axiom in the third. In particular, we find that the vacuum axiom implies
	\beqn
	Y(\kappa^\CV|0\rangle,z) |0\rangle = \kappa^\CV|0\rangle
	\eeqn
	The uniqueness result in Proposition \ref{prop:uniqueness} then implies
	\beqn
	Y(\kappa^\CV |0\rangle, z) = \kappa^\CV
	\eeqn
	as the $S$-action commutes with all $Y(a,z)$, hence the constant field $\kappa^\CV$ is mutually local with them.
	
	Using this description of the $S$-action in terms of fields, we can see that $Y(-,z)$ is a morphism of $S$-modules using $Y(a_{(-1)}b,z) = \norm{Y(a,z)Y(b,z)}$, a consequence of Theorem \ref{thm:assoc}:
	\beqn
	\kappa^\CV Y(a,z) = \norm{Y(\kappa^\CV|0\rangle,z) Y(a,z)} = Y(\kappa^\CV a,z)
	\eeqn
	Putting this together, we see that the data $(Y, |0\rangle, \pd, Y)$ matches that of a raviolo vertex algebra over $S$. We already established the fact that $Y(a,z)$ is a raviolo field over $S$ for all $a$; the remaining axioms of a raviolo vertex algebra over $S$ follow from the fact that they are satisfied for $\CV$.
\end{proof}

\subsection{Interpreting the OPE}
Raviolo vertex algebras are meant to capture the algebraic structure underlying local operators in a three-dimensional holomorphic-topological quantum field theory. In the theory of vertex algebras, the notion of an OPE encodes the physical expectation that in a 2d holomorphic quantum field theory the insertion of two operators placed at points $z, w$ close to one another can be expressed as a sum of operators inserted at $w$ depending on the separation $(z-w)$. Due to its equally central role in raviolo vertex algebras, we think it is worthwhile to provide a physical interpretation for the above OPE, and why it deserves the name OPE in the first place.

As mentioned in Section \ref{sec:ravfields}, the expansion in $(z-w)$ of the normal-ordered product $\norm{Y(a,z) Y(b,w)}$ encodes the physical operator product expansion of two local operators placed at points $\ul{z}$, $\ul{w}$ in $\C \times \R$. The specialized normal-ordered product $\norm{Y(a,w) Y(b,w)} = Y(a_{(-1)}b,w)$ is the field corresponding to the physical operator product of $a$ and $b$. The remaining fields $Y(a_{(-m-1)}b,w)$ encode the physical operator product of derivatives of $a$ with $b$.

The fields $Y(a_{(m)}b,w)$ for $m \geq 0$ have a different interpretation: they encode the singularities in the physical operator product expansion of the \emph{first descendant} of $a$ with $b$. Indeed, the singular terms encode the fields corresponding to a tower of brackets that we define as
\beqn\label{eqn:bracket}
	\{\{a,b\}\}^{(n)} \define a_{(n)} b , \quad n \geq 0
\eeqn
Momentarily we will show that this tower of brackets and the operator product satisfy relations akin to that of a Poisson algebra. Indeed, these brackets are the raviolo vertex algebra avatar of the holomorphic-topological descent brackets of \cite{OhYagi, CostelloDimofteGaiotto-boundary}.

Notice that for each $n \geq 0$ the bracket $\{\{-,-\}\}^{(n)}$ defined in \eqref{eqn:bracket} is of cohomological degree $-1$ while the products $(-)_{(-n-1)} (-)$ are of cohomological degree zero.
Let 
\beqn
	a \cdot b \define a_{(-1)} b .
\eeqn
These operations satisfy the following relations.

\begin{prop}
	For each $n \geq 0$ the bracket $\{\{-,-\}\}^{(n)}$ and product $\cdot$ satisfy the following relations where $a,b,c \in \CV$:
	\begin{enumerate}
		\item (Super commutativity) $a \cdot b = (-1)^{|a| |b|} b \cdot a$.
		\item (Associativity) $a \cdot (b\cdot c) = (a \cdot b) \cdot c$.
		\item (Graded skew symmetry) $\{\{a,b\}\}^{(n)} = (-1)^{|a||b|} \sum_{m\geq0} \frac{(-1)^{n+m}}{m!} \pd^m\{\{b,a\}\}^{(n+m)}$.
		\item (Derivation) $\{\{a,b\cdot c\}\}^{(n)} = \{\{a,b\}\}^{(n)} \cdot c + (-1)^{(|a|-1)|b|} b \cdot \{\{a,c\}\}^{(n)}$.
	\end{enumerate}
	
	Additionally, for $n,m \geq 0$ the Jacobi-type identity holds
	\beqn
	\begin{aligned}
		\{\{a,\{\{b,c\}\}^{(m)}\}\}^{(n)} & = (-1)^{|a|+1}\sum_{l\geq 0} \binom{m}{l} \{\{\{\{a,b\}\}^{(l)},c\}\}^{(m+n-l)}\\
		& \qquad + (-1)^{(|a|+1)(|b|+1)} \{\{b,\{\{a,c\}\}^{(n)}\}\}^{(m)}
	\end{aligned}
	\eeqn
\end{prop}

\begin{proof}
	Super commutativity and associativity follow from Proposition \ref{prop:assoc}; graded skew-symmetry follows from Proposition \ref{prop:skew}; the derivation property follows from Corollary \ref{cor:NOPcomm}; the Jacobi-type identity follows from the commutator of the modes $a_{(n)}$ and $b_{(m)}$ for $n,m\geq0$ established in Proposition \ref{prop:locality} together with the above identification.
\end{proof}

\subsection{Comparison to 1-shifted Poisson vertex algebras}
\label{sec:PVAcomp}

As we have mentioned several times, there is another model for the algebraic structure furnished by local operators in a three-dimensional mixed holomorphic-topological quantum field theory as established by \cite{OhYagi}: that of a 1-shifted Poisson vertex algebra. In this section we compare the notion of a 1-shifted Poisson vertex algebra to that of a raviolo vertex algebra, finding they are totally equivalent.
We first recall the definition of a vertex Lie algebra, cf. Chapter 16 of \cite{FBZ}.

A vertex Lie algebra is the data $(\CL, \pd, Y_-)$ where
\begin{itemize}
		\item $\CL = \oplus_r \CL^r$ is a $\Z$-graded vector space.
		\item $\pd: \CL \to \CL$ is an endomorphism of degree $0$ (the translation operator).
		\item $Y_-: \CL \to \End(\CL) \otimes \CK_{dist}^1$ is a linear map of degree $0$.
\end{itemize}
where $Y_-$ maps a vector $a \in \CL$ to a series
\begin{equation*}
		Y_-(a,z) = \sum_{m \geq 0} \Omega^n_z a_{(m)}
\end{equation*}
such that for any $v \in \CL$ we have $a_{(n)} v = 0$ for $n \gg 0$. This data satisfies the following axioms
\begin{enumerate}
		\item For every $a \in \CL$ we have $Y_-(\pd a,z) = \pd_z Y_-(a,z)$ (translation axiom). 
		\item For every $a,b \in \CL$ we have $Y_-(a,z) b = (-1)^{|a||b|} e^{z \pd} Y_-(b,-z) a$ (skew-symmetry axiom).
		\item For every $a,b \in \CL$ we have \begin{equation*}
			[a_{(m)}, Y_-(b,w)] = \sum_{n\geq0} \binom{m}{n} w^{m-n} Y_-(a_{(n)}b,w) .
		\end{equation*} (commutator axiom).
\end{enumerate}

Even though $\CK_{dist}^1$, the space of degree one formal raviolo distributions, appears explicitly this definition of a vertex Lie algebra completely agrees with the standard one (except for the fact that we are considering a cohomologically graded version of a vertex Lie algebra).
More commonly, the expansion of $Y_-$ is in positive powers of the variable $z^{-1}$, but for raviolo vertex algebras it is more natural to use the variables $\Omega^n_z$, $n \geq 0$.
Also recall that it is common to organize the commutators of the modes in a vertex Lie algebra into a \emph{$\lambda$-bracket}, leading to the notion of a Lie conformal algebra, cf. Section 2.7 \cite{Kac}. 

As with vertex Lie algebras, we can extract a vertex Lie algebra from a raviolo vertex algebra.
\begin{lemma}
\label{lem:RVAtoshiftedVLA}
	Let $\CV$ be a raviolo vertex algebra. The choice $Y_-(a,z) = Y(a,z)_-$ gives the data $(\CV[1], \pd, Y_-)$ the structure of a vertex Lie algebra.
\end{lemma}

Notice that the key difference in obtaining a vertex Lie algebra from a vertex or raviolo vertex algebra is the cohomological degree shift by one.
This is because $\CK^1$ is the degree one part of formal functions $\CK$ on the raviolo.

\begin{proof}
	The translation axiom follows from Lemma \ref{lem:translation}; the skew-symmetry axiom follows from the skew-symmetry property established in Proposition \ref{prop:skew}; the commutator axiom follows from mutual locality, cf. Proposition \ref{prop:locality}.
\end{proof}

There is also the other direction: from a vertex Lie algebra one can construct a raviolo vertex algebra as follows.
Let $\CL$ be a vertex Lie algebra.
Define
\beqn
\text{Lie}_{rav}(\CL) \define (\CL \otimes \CK) \slash \text{Im} \left(\del \otimes \id + \id \otimes \del_z\right) .
\eeqn
For $a \in \CL$ denote by $a_{[n]}$ the image of $a \otimes z^n$ in $\text{Lie}_{rav}(\CL)$ for $n \geq 0$.
Similarly, denote by $a_{[n]}$ the image of $a \otimes \Omega^{-n-1}$ in $\text{Lie}_{rav}(\CL)$ for $n <0$.
Note that if $a$ is of cohomological degree $r$ in $\CL$ then $a_{[n]}$ is of cohomological degree $r$ (respectively $r+1$) for $n \geq 0$ (respectively $n < 0$).
Similarly, we let 
\beqn
\text{Lie}_{rav}(\CL)_+ \define \CL[\![z]\!] \slash \text{Im} \left(\del \otimes \id + \id \otimes \del_z\right) .
\eeqn

The proof of the following lemma is completely analogous to the ordinary vertex algebra setting, see \cite[Lemma 16.1.7]{FBZ}.

\begin{lemma}
The graded vector space $\text{Lie}_{rav}(\CL)$ is a graded Lie algebra with Lie bracket defined by
\beqn
\left[a_{[n]}, b_{[m]} \right] = \sum_{k \geq 0} \binom{n}{k} (a_{(n)} \cdot b)_{[n+m-k]} ,
\eeqn
and $\text{Lie}_{rav}(\CL)_+ \subset \text{Lie}_{rav}(\CL)$ is a Lie subalgebra.
\end{lemma}

Next, define
\beqn
\text{Vac}(\CL) \define U\left(\text{Lie}_{rav}(\CL)\right) \otimes_{U\left(\text{Lie}_{rav}(\CL)\right)_+} \C . 
\eeqn
Here $\C$ denotes the one-dimensional trivial representation of $\text{Lie}_{rav}(\CL)_+$.
From the reconstruction theorem proved below, Proposition \ref{prop:reconstruction}, one can prove the following result.

\begin{prop}
There is a raviolo vertex algebra structure on $\text{Vac}(\CL)$ such that 
\beqn
Y(a_{[-1]} |0 \rangle, z) = \sum_{n \geq 0} a_{[n]} \Omega^n + \sum_{n < 0} a_{[n]} z^{-n-1}  .
\eeqn
Moreover, if $\CV$ is any raviolo vertex algebra then there is an isomorphism
\beqn
\Hom(\CL, \CV[1]) \simeq \Hom(\text{Vac}(\CL),\CV)
\eeqn
where on the left hand side we view $\CV[1]$ as a vertex Lie algebra following Lemma~\ref{lem:RVAtoshiftedVLA}.
\end{prop}

We now define a 1-shifted Poisson vertex algebra, see also \cite{Tamarkin,OhYagi}. 
This combines the notion of a (commutative) vertex algebra and a vertex Lie algebra in a compatible way; see e.g. Section 16.2 of \cite{FBZ} for the analogous definition in the unshifted case.

\begin{dfn}
	A \emph{1-shifted Poisson vertex algebra} is the data $(\CV, |0\rangle, \pd, Y_+, Y_-)$, where $(\CV, |0\rangle, \pd, Y_+)$ is a commutative vertex algebra, $(\CV[1], \pd, Y_-)$ is a vertex Lie algebra, and for any $a \in \CV$ the modes of $Y_-(a,z)$ act by derivations of the normal-ordered product on $\CV$.
\end{dfn}

\begin{theorem}
\label{thm:RVAvsshiftedPVA}
	Let $\CV$ be a $\Z$-graded vector space. A raviolo vertex algebra structure on $\CV$ is equivalent to a 1-shifted Poisson vertex algebra structure on $\CV$. These two structure have the same translation operator $\pd$ and vacuum vector $|0\rangle$. The state-operator correspondences are related as
	\[
		Y_\pm(a,z) = Y(a,z)_\pm \qquad Y(a,z) = Y_+(a,z) + Y_-(a,z) .
	\]
	for any $a \in \CV$.
\end{theorem}

\begin{proof}
	Suppose $\CV$ has the structure of a raviolo vertex algebra with vacuum vector $|0\rangle$, translation operator $\pd$, and state-operator correspondence $Y(-,z)$; we define
	\beqn
		Y_\pm(a,z) \define Y(a,z)_\pm
	\eeqn
	Mutual locality of the $Y(a,z)$ implies that $Y_+(a,z)$ commutes with $Y_+(b,w)$ for any $a,b \in \CV$. In particular, $(\CV, |0\rangle, \pd, Y_+)$ has the structure of a commutative vertex algebra; the vertex algebra axioms follow from their raviolo analogs. Lemma \ref{lem:RVAtoshiftedVLA} implies $(\CV[1], \pd, Y_-)$ has the structure of a vertex Lie algebra. To see that $(\CV, |0\rangle, \pd, Y_+, Y_-)$ is a 1-shifted Poisson vertex algebra, it remains to check that the modes of $Y_-(a,z)$ act as derivations of the normal-ordered product on $(\CV, |0\rangle, \pd, Y_+)$, but this follows from Corollary \ref{cor:NOPcomm}.
	
	We now suppose $\CV$ has the structure of a 1-shifted Poisson vertex algebra and define
	\beqn
		Y(a,z) = Y_+(a,z) + Y_-(a,z)
	\eeqn
	Note that the commutativity of $(\CV, |0\rangle, \pd, Y_+)$ implies $Y_+(a,z)$ has no terms proportional to negative powers of $z$. (Rather, the corresponding endomorphisms act as zero.) That this defines a (raviolo) field on $\CV$ follows from the fact that $(\CV[1], \pd, Y_-)$ is a vertex Lie algebra.
	
	As the modes of $Y_-(a,z)$ act as derivations of the normal-ordered product on $(\CV, |0\rangle, \pd, Y_+)$, we have the following commutator for any $a,b \in \CV$:
	\beqn
		[a_{(m)}, b_{(-1)}] = (a_{(m)} b)_{(-1)}
	\eeqn
	where $m \geq 0$. Taking commutators with $\pd$, we find for any $l < 0$
	\beqn
		[a_{(m)}, b_{(l)}] = \sum_{n \geq \max(0,m+l+1)} \binom{m}{n} (a_{(n)} b)_{(m+l-n)}
	\eeqn
	The skew-symmetry axiom of $(\CV, \pd, Y_-)$ applied to $Y(b,z)_- a$, together with the above commutator of $[a_{(m)}$, $b_{(l)}]$ for $m, l \geq 0$ and the commutativity of the $Y_+(a,z)$, implies that the (raviolo) fields $Y(a,z)$ and $Y(b,w)$ are mutually local for any $a,b \in \CV$ using Proposition \ref{prop:locality}. 
	
	To verify the translation axiom, it suffices to check $[\pd, Y_-(a,z)] = \pd_z Y_-(a,z)$. The translation axiom of a vertex Lie algebra implies $(\pd a)_{(m)} = -m a_{(m-1)}$. Choose $b \in \CV$, we compute:
	\beqn
	\begin{aligned}
		[\pd, a_{(m)}] b & = \bigg(\pd (a_{(m)}b)_{(-2)} + a_{(m)} b_{(-2)}\bigg)|0\rangle\\
		& = -m a_{(m-1)} b_{(-1)}|0 = -m a_{(m-1)} b
	\end{aligned}
	\eeqn
	whence $[\pd, Y_-(a,z)] = \pd_z Y_-(a,z)$.
	
	We are left with checking the vacuum axiom. The vacuum axiom of $(\CV, |0\rangle, \pd, Y_+)$ implies $Y_+(|0\rangle, z)$ is the unit of the normal-ordered product; because $Y_-(a,z)$ acts a derivation of the normal-ordered product we conclude $[Y_-(a,z), Y_+(|0\rangle, w)] = 0$ for any $a \in \CV$, from which it follows $Y_-(a,z) |0\rangle = 0$ for any $a$. Thus, $Y(a,z)|0\rangle = Y_+(a,z)|0\rangle \in \CV[\![z]\!]$ by the vacuum axiom of $(\CV, |0\rangle, \pd, Y_+)$. Finally, to see that $Y(|0\rangle, z) = \id_\CV$, it suffices to check $Y_-(|0\rangle, z) = 0$, which follows from the skew-symmetry property of $(\CV, \pd, Y_-)$ together with $Y_-(a,z)|0\rangle = 0$:
	\beqn
		Y_-(|0\rangle, z) a = e^{z\pd} Y_-(a,-z)|0\rangle = 0
	\eeqn
\end{proof}

\section{Examples of raviolo vertex algebras}
\label{sec:examples}

We now describe some simple examples of raviolo vertex algebras.

Our last general result is the raviolo analog of the Reconstruction Theorem of vertex algebras, cf. Proposition 3.1 of \cite{FKRW} or Theorem 4.4.1 of \cite{FBZ}. We will find this useful for showing our examples actually furnish the structure of a raviolo vertex algebra.

\begin{prop}
	\label{prop:reconstruction}
	Let $\CV = \bigoplus \CV^r$ be a $\Z$-graded vector space, $|0\rangle$ a non-zero vector, $\pd$ a degree $0$ endomorphism of $\CV$. Further, let $\{a^i\}$ be a countable ordered set of vectors in $\CV$, with $a^i \in \CV^{r_i}$, and suppose we are given homogeneous fields
	\begin{equation*}
		A^i(z) = \sum_{m<0} z^{-m-1} A^i_{m} + \sum_{m\geq0} \Omega^m_z A^i_{m}
	\end{equation*}
	of degree $r_i$ such that the following hold:
	\begin{itemize}
		\item[1)] $A^i_{-1}|0\rangle = a^i$ and $A^i_{m}|0\rangle = 0$ for all $i$ and $m \geq 0$
		\item[2)] $\pd |0\rangle = 0$ and $[\pd, A^i(z)] = \pd_z A^i(z)$ for all $i$
		\item[3)] All fields $A^i(z)$ are mutually local
		\item[4)] $\CV$ is spanned by the vectors
		\begin{equation*}
			A^{i_1}_{j_1} \dots A^{i_l}_{j_l}|0\rangle \qquad j_k < 0
		\end{equation*}
	\end{itemize}
	Then the assignment
	\begin{equation*}
		\begin{aligned}
			Y(A^{i_1}_{j_1} \dots A^{i_l}_{j_l}|0\rangle, z) &\\
			& \hspace{-1cm} =\frac{1}{(-j_1-1)!\dots (-j_l-1)!}\norm{\pd^{-j_1-1}_z A^{i_1} \dots \pd^{-j_l-1}_z A^{i_l}}\\
		\end{aligned}
	\end{equation*}
	determines a well-defined raviolo vertex algebra structure on $\CV$. Moreover, this is the unique raviolo vertex algebra structure on $\CV$ satisfying $1) - 4)$ and such that $Y(a^i,z) = A^i(z)$.
\end{prop}

If $\CV$ has a $\frac{1}{L}\Z$-grading such that $|0\rangle$ has weight $0$, $\pd$ increases weight by $1$, the vectors $\{a^i\}$ are homogeneous, with $a^i$ of weight $s_i$, and the mode $A^i_{m}$ of $A^i(z)$ has weight $s_i-m-1$, then this grading induces a spin grading for $\CV$. Similarly, if $\CV$ has an additional $\Z/2$ grading so that $a^i$ is a boson (resp. fermion) and the modes $A^i_{m}$ are bosonic (resp. fermionic) for $m<0$ and fermionic (resp. bosonic) for $m \geq 0$ then this induces a super grading on $\CV$.

If $S$ is a graded commutative unital $\C$-algebra, then we can replace $\Z$-graded vector space by graded $S$-module, add \emph{as an $S$-module} after the word spanned in condition $4)$, require $\pd$ is an $S$-module morphism, and require $A^i$ are homogeneous fields over $S$, then the following proof yields a raviolo vertex algebra over $S$.

\begin{proof}
	The proof is nearly identical to the proof of Theorem 4.4.5 of \cite{FBZ}, but we shall repeat it for completeness.
	
	First, note that property $3)$ together with Lemma \ref{prop:derivlocality} and Lemma \ref{lem:Dong} implies that all of the fields
	\beqn
		\frac{1}{(-j_1-1)!\dots (-j_l-1)!}\norm{\pd^{-j_1-1}_z A^{i_1} \dots \pd^{-j_l-1}_z A^{i_l}}
	\eeqn
	are mutually local.
	
	Second, $\pd|0\rangle = 0$ is a consequence of $2)$ and $Y(|0\rangle, z) = \id_\CV$ is a consequence of $4)$. We can further show that applying the above field to $|0\rangle$ gives an element of $\CV[\![z]\!]$ with $A^{i_1}_{j_1} \dots A^{i_l}_{j_l}|0\rangle$ as the constant term by induction on $l$. The case $l=1$ is provided by property $1)$. The general case follows by using
	\beqn
		\norm{\pd^{-j-1}A^i(z) A(z)} = \pd^{-j-1}A^i(z)_+ A(z) + (-1)^{|A^i||A|} A(z) A^i(z)_-
	\eeqn
	to see that $\norm{\pd^{-j-1}A^i(z) A(z)}|0\rangle$ belongs to $\CV[\![z]\!]$ and has a constant term of the desired form: the second term annihilates $|0\rangle$ by property $1)$, and the inductive hypothesis implies $\pd^{-j-1}A^i(z)_+ A(z)|0\rangle$ belongs to $\CV[\![z]\!]$ and such that $\pd^{-j-1}A^i(z)_+ A(z)|0\rangle|_{z=0} = (-j-1)! A^i_{j} A(z)|0\rangle|_{z=0}$.
	
	Third, repeated applications of property $2)$ implies $[\pd, \pd^{-j-1}_z A^i(z)] = \pd^{-j}_z A^i(z)$ for any $j<0$. Induction on $l$ then shows $[\pd, A(z)] = \pd_z A(z)$ for $A(z)$ any field of the above form: the residue formula in Lemma \ref{lem:NOP} and the inductive hypothesis then implies
	\beqn
	\begin{aligned}[]
		[\pd, \norm{\pd^{-j-1}_z A^i(z) A(z)}] & = \norm{\pd^{-j}_z A^i(z) A(z)} + \norm{\pd^{-j-1}_z A^i(z) \pd_z A(z)}\\
		& = \pd_z \norm{\pd^{-j-1}_z A^i(z) A(z)}
	\end{aligned}
	\eeqn
	where we use Lemma \ref{lem:derivNOP} in the second equality.
	
	Now choose a basis amongst the $A^{i_1}_{j_1} \dots A^{i_l}_{j_l}|0\rangle$ containing $|0\rangle$ and define the state-operator correspondence $Y(-,z)$ on $\CV$ by the above formula extended to $\CV$ linearly. The above shows that $(\CV, |0\rangle, \pd, Y)$ gives $\CV$ the structure of a raviolo vertex algebra: the locality axiom is follows from the first point; the vacuum axiom follows from the second point; and the translation axiom follows from the third point.
	
	We now show that the state-operator correspondence $Y(-,z)$ does not depend on the choice of basis. Suppose $Y_1(-,z)$ and $Y_2(-,z)$ are the state-operator correspondences for two different choices. Lemma \ref{lem:Dong} and Proposition \ref{prop:derivlocality} imply $Y_1(a,z)$ and $Y_2(b,z)$ are mutually local for all $a,b \in \CV$. For fixed $a$, the vertex algebra axioms imply
	\beqn
		Y_1(a,z)|0\rangle|_{z=0} = a = Y_2(a,z)|0\rangle|_{z=0}
	\eeqn
	and the translation axioms imply
	\beqn
		\pd_z Y_1(a,z)|0\rangle = \pd Y_1(a,z)|0\rangle \qquad \pd_z Y_2(a,z)|0\rangle = \pd Y_2(a,z)|0\rangle
	\eeqn
	We see that these fields agree $Y_1(a,z) = Y_2(a,z)$ due to Proposition \ref{prop:uniqueness}.
	
	Finally, to the asserted uniqueness of the vertex algebra structure is provided by the fact that the above formula for the state-operator correspondence is uniquely determined by $Y(a^i,z) = A^i(z)$, cf. Corollary \ref{cor:stateopNOP}.
\end{proof}

This reconstruction result enables a generator-and-relations style approach to raviolo vertex algebras. We formalize these notions with the following definitions.

\begin{dfn}
	Let $\CV$ be a raviolo vertex algebra. A countable ordered set of (homogeneous) vectors $\{a^i\}$ are \defterm{(homogeneous) generators} of $\CV$ if $\CV$ is spanned by monomials of the form $a^{i_1}_{(j_1)} \dots a^{i_l}_{(j_l)}|0\rangle$. The generators $\{a^i\}$ are said to be \defterm{strong generators} if $\CV$ is spanned by monomials of the form $a^{i_1}_{(j_1)} \dots a^{i_l}_{(j_l)}|0\rangle$ with $j_i < 0$.
	
	If $\CV$ is a raviolo vertex algebra over $S$, we say that a set of (homogeneous) elements $\{b^i\}$ are \defterm{generators over $S$} if $\CV$ is spanned as an $S$-module by monomials of the form $b^{i_1}_{(j_1)} \dots b^{i_l}_{(j_l)}|0\rangle$. The generators $\{b^i\}$ are said to be \defterm{strong generators over $S$} if $\CV$ is spanned as an $S$-module by monomials of the form $b^{i_1}_{(j_1)} \dots b^{i_l}_{(j_l)}|0\rangle$ with $j_i < 0$.
\end{dfn}

With these definitions, we see that the vectors $\{a^i\}$ appearing in the formulation of Proposition \ref{prop:reconstruction} are strong generators of $\CV$. If $\CV$ is a raviolo vertex algebra over $S$ and $\{b^i\}$ are strong generators over $S$, they need not be strong generators of $\CV$ merely as a raviolo vertex algebra.

\subsection{Free field algebras}
\label{sec:freefields}
Our first example will be the simplest example of a free field algebra, roughly analogous to the complex fermion (a.k.a. $bc$ ghost) or symplectic boson (a.k.a. $\beta \gamma$ ghost) vertex algebras.

We start by considering the abelian Lie algebra $\CF := \CK \oplus \CK^{(1)}[1]$ which is concentrated in cohomological degrees $-1,0,1$. 
The residue pairing gives us a natural central extension $\widehat{\CF}$:
\be
	0 \to \C K \to \widehat{\CF} \to \CF \to 0
\ee
For $\alpha \in \CK$ and $\beta \in \CK^{(1)}[1]$ the commutator in the central extension is given by
\be
[\alpha, \beta] = K \textrm{Res} \alpha \beta = - [\beta, \alpha]
\ee
where $K$ is the (degree 0, bosonic) central generator.
Notice that the cohomological shift in $\beta$ is necessary to ensure $K$ is degree $0$.
There is an auxiliary grading, which we will refer to as flavor charge in this section, induced by declaring that $\alpha$ has flavor charge $1$ and $\beta$ has flavor charge $-1$.
This forces $K$ to have flavor charge $0$.

The Lie algebra $\widehat{\CF}$ has a basis given by the central generator $K$ together with $Y_{n} = u^n$ and $\psi_n = \Omega^n_u$ from $\CK$ as well as $\chi_n = u^n \d u$ and $X_n = \d u \Omega^{n}_u$ from $\CK^{(1)}[1]$, where $n \in \Z_{\geq 0}$.
The gradings are as follows:
\begin{itemize}
	\item The generator $K$ has cohomological degree zero and spin zero.
	\item The generator $Y_n$ has cohomological degree zero and spin $-n$.
	\item The generator $\psi_n$ has cohomological degree $1$ and spin $n+1$.
	\item The generator $X_n$ has cohomological degree zero and spin $n$. 
	\item The generator $\chi_n$ has cohomological degree $-1$ and spin $-n-1$.
\end{itemize}
The generators $X_n, \chi_n$ have flavor charge $1$ while the generators $\psi_n, Y_n$ have flavor charge $-1$.
These generators have the following commutators:
\be
	[Y_n, X_m] = \delta_{n,m}K = -[X_m, Y_n] \qquad [\psi_n, \chi_m] = \delta_{n,m}K = [\chi_n, \psi_m]
\ee
The action of $-\pd_u$ on $\CK, \CK^{(1)}$ leads to a derivation $\pd$ on $\widehat{\CF}$:
it annihilates the central generator $\pd K = 0$, its action on $Y_n, \chi_n$ is given by
\be
	\pd Y_n = -n Y_{n-1}, \qquad \pd \chi_n = -n \chi_{n-1}
\ee
and its action on $X_n, \psi_n$ is given by
\be
	\pd X_{n} = (n+1) X_{n+1} \qquad \pd \psi_n = (n+1) \psi_{n+1} .
\ee

We now define a particular `vacuum' type module for $\Hat{\CF}$.
The centrally-extended algebra $\widehat{\CF}$ has a ``positive'' subalgebra $\widehat{\CF}_{\geq0}$ generated by $Y_n, \chi_n$, and $K$ for $n \geq 0$.
For $k \in \C$ we then consider the representation $\C_k$ of $\Hat{\CF}_{\geq0}$ where $Y_n, \chi_n$ act as zero and $K$ acts as multiplication by $k$.
The vacuum module is defined by induction
\be
	FC_k = U \widehat{\CF} \otimes_{U\widehat{\CF}_{\geq0}} \C_k 
\ee
where $U\fg$ denotes the universal enveloping algebra of a Lie algebra $\fg$. 

We are mostly interested in situations with $k \ne 0$; in this case we can rescale the generators to set $k = 1$ and we will always assume this has been done. Correspondingly, $K$ acts as the identity matrix $\id_{FC_1}$ and we will omit $k, K$ from the notation entirely. The proof that there is a raviolo vertex algebra structure for $k=0$ is exactly the same as $k \ne 0$. Yet more generally, we could replace $\C_k$ by $\C[K]$ and consider the $\C[K]$-module
\be
	FC_{univ} = U \widehat{\CF} \otimes_{U\widehat{\CF}_{\geq0}} \C[K] 
\ee
The following proof shows that $FC_{univ}$ has the structure of a raviolo vertex algebra over the polynomial ring $\C[K]$ strongly generated over $\C[K]$ by $X(z)$ and $\psi(z)$. The raviolo vertex algebra $FC_k$ is identified with the quotient of $FC_{univ}$ associated to the maximal ideal $(K-k)$ of $\C[K]$.

\begin{prop}
	$FC$ is a raviolo vertex algebra with spin grading that is strongly generated by the fields
	\begin{equation*}
		X(z) = \sum_{n\geq 0} z^n X_{n} + \Omega^n_z \chi_{n} \; , \qquad \psi(z) = \sum_{n\geq 0} z^n \psi_{n} + \Omega^n_z Y_{n} .
	\end{equation*}
\end{prop}

In terms of our usual notation, we have that $X_{(n)} = X_{-n-1}$ for $n < 0$ and $X_{(n)} = \chi_n$ for $n \geq 0$ and similarly for $\psi_{(n)}$.

\begin{proof}
	The vacuum vector $|0 \rangle$ will be identified with the image of $1 \otimes 1 \in U \widehat{\CF} \otimes \C_1$ in the quotient $FC$ and the translation operator $\pd$ on $FC$ will be induced from the one defined on $\widehat{\CF}$ together with $\pd1=0$. To complete the proof, we now show that the vectors $X_0|0\rangle$, $\psi_0|0\rangle$ and the fields $X(z)$, $\psi(z)$, together with the vacuum vector $|0\rangle$ and the translation operator $\pd$, satisfy the conditions of Proposition \ref{prop:reconstruction}.
	
	The fact that the above ``spin'' induces a spin grading on $FC$ is obvious. Conditions $2)$ and $4)$ hold by construction. Condition $1)$ follows from the fact that $Y_n$ and $\chi_n$ annihilate the vacuum vector. To complete the proof we verify mutual locality of the generators by computing their commutators:
	\be
	\begin{aligned}[]
		[X(z), X(w)] & = 0 \qquad & [\psi(z), X(w)] & = \Delta(z,w) \id_{FC} \\
		[X(z), \psi(w)] & = \Delta(z,w)\id_{FC} \qquad & [\psi(z), \psi(w)] & = 0
	\end{aligned}
	\ee
\end{proof}

We see that the singular terms in the OPEs of these generating fields take the following form:
\be
\begin{aligned}
	X(z) X(w) & \sim 0 \qquad & \psi(z) X(w) & \sim \Omega_{z-w}^0\\
	X(z) \psi(w) & \sim \Omega_{z-w}^0 \qquad & \psi(z) \psi(w) & \sim 0
\end{aligned}
\ee
where we have suppressed the factors of $\id_{FC}$. In particular, $X$ and $\psi$ have regular OPEs with themselves and a non-trivial singular term in the OPE with one another. This example arises as the raviolo vertex algebra underlying the local operators in the minimal twist of a collection of a free three-dimensional $\CN=2$ supersymmetric theory called the theory of a `free chiral multiplet', cf. Section 3.2 of \cite{OhYagi}.

Let's consider the two-variable character of $FC_k$ defined by
\beqn
	\op{ch}_{FC_k}(q,y) \define \sum_{r,s \in \Z} \op{grdim}(FC_k^{(s),r}) q^s y^r
\eeqn
where $FC_k^{(s),r} \subset FC_k$ is spanned by the elements of spin $s$ and flavor charge~$r$.
Then
\beqn
	\op{ch}_{FC_k}(q,y) = \frac{(q y^{-1};q)_\infty}{(y;q)_\infty}
\eeqn
where the infinite Pochammer symbol is the formal series $(y;q)_\infty = \prod_{n \geq 0} (1-yq^n)$.
This character reproduces the superconformnal index of the free three-dimensional chiral multiplet, as expected.

There are variants of this example obtained by changing the spin that the generators carry.
For this general case, we can define a central extension of the abelian graded Lie algebra $\CF^{(s)} = \CK^{(-s)} \oplus \CK^{(s+1)}[1]$ which can be used to define the raviolo vertex algebra $FC^{(s)}_{univ}$ over $\C[K]$ and its quotients where $X(z)$ has spin $s$ and $\psi(z)$ has spin $1-s$.
	
Another variant changes the cohomological degrees of the generators at the cost of introducing a super grading. 
We give the field $X(z)$ an intrinsic cohomological degree $r$ as well as the super grading $(-1)^r$ to ensure it is always a boson; the field $\psi(z)$ is then a fermion of cohomological degree $1-r$. We will denote the resulting raviolo vertex algebra over $\C[K]$ by $FC_{r, univ}$. We instead follow the above construction with $\CF$ replaced by $\CF_r: = \Pi^{r}\big(\CK[r] \oplus \CK^{(1)}[1-r]\big)$.
	
Finally, we note that we can combine these two variants to form the raviolo vertex algebra $FC^{(s)}_{r, univ}$ strongly generated over $\C[K]$ by the bosonic field $X(z)$ of cohomological degree $r$ and spin $s$ and the fermionic field $\psi(z)$ of cohomological degree $1-r$ and spin $1-s$. We denote its quotient at the ideal $(K-1)$ by $FC^{(s)}_r$.

\subsection{Raviolo Heisenberg algebra}
\label{sec:heisenberg}
Our second example is another free field algebra, roughly analogous to the Heisenberg and symplectic fermion vertex algebras.

Consider the abelian graded Lie algebra $\CH = \CK \oplus \CK[1]$ which is concentrated in degrees $-1,0,1$.
The residue gives a $2$-cocycle on this Lie algebra
\beqn
	(\alpha, \beta) \mapsto \Res_u \diff u \alpha \del_u \beta 
\eeqn
where $\alpha\in \CK, \beta \in \CK[1]$ which, in turn, determines a central extension
\beqn
	0 \to \C K \to \Hat{\CH} \to \CH \to 0 .
\eeqn
Notice that this differs from the previous example by a derivative.
We can assign an auxiliary grading to this graded Lie algebra which we again call flavor charge\footnote{For the reader familiar with the appearance of this system in the context of three-dimensional supersymmetric gauge theory, we point out that this terminology is misleading, but we will use it only in this section.}, by declaring that $\alpha$ has charge $+1$ and $\beta$ has charge $-1$.

The Lie algebra $\Hat{\CH}$ has a basis given by the central generator $K$ together with $c_{n} = u^n$ and $\nu_n = \Omega^n_u$ from $\CK$ as well as $\varphi_n = u^n$ and $b_n = \Omega^n_u$ from $\CK [1]$, where $n \in \Z_{\geq 0}$. As before the action of $-\pd_u$ defines a derivation $\pd$ that annihilates on $K$, acts on $\CK$ as
\beqn
	 \pd \nu_n = (n+1) \nu_{n+1} \qquad \pd c_n = -n c_{n-1}
\eeqn
and on $\CK[1]$ as
\beqn
	\pd b_n = (n+1) b_{n+1} \qquad \pd \varphi_n = -n \varphi_{n-1}
\eeqn

The gradings are as follows:
\begin{itemize}
	\item The generator $K$ has cohomological degree zero and spin zero.
	\item The generator $c_n$ has cohomological degree zero and spin $-n$.
	\item The generator $\nu_n$ has cohomological degree $1$ and spin $n+1$.
	\item The generator $b_n$ has cohomological degree zero and spin $n+1$. 
	\item The generator $\varphi_n$ has cohomological degree $-1$ and spin $-n$.
\end{itemize}
The non-vanishing commutators of these generators are
\be
	[c_n, b_m] = -n \delta_{n,m+1}K = -[b_m, c_n]
\ee
and 
\be
	[\nu_n, \varphi_m] = m \delta_{n+1,m}K = [\varphi_m, \nu_n]
\ee
Note that the generators $c_0$ and $\varphi_0$ are central in $\Hat{\CH}$.

The graded Lie algebra $\Hat{\CH}$ has a positive subalgebra $\Hat{\CH}_{\geq 0} = \CK^0 \oplus \CK^0 [1] \oplus \C K$.
For $k \in \C$ define the $\Hat{\CH}_{\geq 0}$-module $\C_k$ where $\CK^0,\CK^0 [1]$ act trivially and $K$ acts by $k$.
Define the induced $\Hat{\CH}$-module
\beqn
	H_k \define U \Hat{\CH} \otimes_{U \Hat{\CH}_{\geq 0}} \C_k .
\eeqn
For $k \ne 0$ we can rescale generators to set $k = 1$, we denote $H = H_1$. As in the case of $FC$, we could even consider the $\C[K]$-module
\beqn
	H_{univ} \define U \Hat{\CH} \otimes_{U \Hat{\CH}_{\geq 0}} \C[K] .
\eeqn
and the following proof leads us to a raviolo vertex algebra over $\C[K]$ strongly generated over $\C[K]$ by two fields $b(z)$ and $\nu(z)$. The raviolo vertex algebra $H_k$ is then the quotient of $H_{univ}$ associated to the ideal $(K-k)$ of $\C[K]$.

\begin{prop}
	$H$ is a raviolo vertex algebra with spin grading that is strongly generated by the fields
	\begin{equation*}
		b(z) = \sum_{n\geq 0} z^n b_{n} + \Omega^n_z \varphi_{n} \; , \qquad \nu(z) = \sum_{n\geq 0} z^n \nu_{n} + \Omega^n_z c_{n}.
	\end{equation*}
\end{prop}
\begin{proof}
	The vacuum vector $|0 \rangle$ is the image of $1 \otimes 1 \in U \Hat{\CH} \otimes \C_1$ in the quotient $H$ and the translation operator $\pd$ is induced from the one defined on $\Hat{\CH}$ together with $\pd1=0$. We associate the vectors $b_{0}|0\rangle$ and $\nu_0|0\rangle$ to the fields $b(z)$, $\nu(z)$. To complete the proof, it suffices to show this data satisfies the conditions of Proposition \ref{prop:reconstruction}. As this is nearly identical to the previous example, we just verify mutual locality of the generators. Computing their commutators, we find
	\be
	\begin{aligned}[]
		[b(z), b(w)] & = 0 \qquad & [\nu(z), b(w)] & = -\del_w\Delta(z-w)\\
		[b(z), \nu(w)] & = \del_w \Delta(z-w) \qquad & [\nu(z), \nu(w)] & = 0
	\end{aligned}
	\ee
	so that these generators are indeed mutually local.
\end{proof}

We see that the singular terms in the OPEs of these generating fields take the following form:
\be
\begin{aligned}
	b(z) b(w) & \sim 0 \qquad & \nu(z) b(w) & \sim -\Omega_{z-w}^1\\
	b(z) \nu(w) & \sim \Omega_{z-w}^1 \qquad & \nu(z) \nu(w) & \sim 0 .
\end{aligned}
\ee
In the universal raviolo Heisenberg algebra $H_{univ}$, the OPEs replace $\Omega^1_{z-w}$ by $\Omega^1_{z-w} K$.

Let's consider the two-variable character of $FC_k$ defined by
\beqn
\op{ch}_{H}(q,y) \define \sum_{r,s \in \Z} \op{grdim}(H^{(s),r}) q^s y^r
\eeqn
where $H^{(s),r} \subset H$ is spanned by the elements of spin $s$ and flavor charge $r$.
Then
\beqn
\op{ch}_{H}(q,y) = \frac{(y^{-1}q;q)_\infty}{(yq;q)_\infty} .
\eeqn

Recall the graded Lie algebra $\Hat{\CF}$ underlying the free chiral raviolo vertex algebra $FC$ from the previous subsection.
Notice that there is a map of graded Lie algebras
\beqn
	\Hat{\CH} \to \Hat{\CF}
\eeqn
induced by sending $c_n \mapsto Y_n, \nu_n \mapsto \psi_n$, $\varphi_n \mapsto -\pd \chi_n$, $b_n \mapsto -\pd X_n$, and $K \mapsto K$.
This induces a morphism $H_{univ} \to FC_{univ}$ of raviolo vertex algebras over $\C[K]$ and hence morphisms $H_k \to FC_k$ of the quotients that send $\nu(z)$ to $\psi(z)$ and $b(z)$ to $-\pd_z X(z)$. We also note that there is a map of graded Lie algebras $\Hat{\CH} \to \Hat{\CF}^{(1)}$ and hence a morphism $H_{univ} \to FC^{(1)}_{univ}$ corresponding to the identifications $\nu(z) = \pd_z \psi(z)$ and $b(z) = X(z)$.

There are variants of this example obtained by changing the spin that the generators carry.
For this general case, we can define a central extension of the abelian graded Lie algebra $\CK^{(-s)} \oplus \CK^{(s)}[1]$ which can be used to define the raviolo vertex algebra $H^{(s)}_{univ}$ over $\C[K]$ where $b(z)$ has spin $1+s$ and $\nu(z)$ has spin $1-s$. It is also possible to give the generators an intrinsic cohomological grading at the cost of introducing a super grading: $b(z)$ would be a boson of cohomological degree $r$ and $\nu(z)$ a fermion of degree $1-r$.

\subsection{Raviolo current algebra}
\label{sec:currents}
The next example we consider is the raviolo analog of an affine vertex algebra.
We choose a complex Lie superalgebra $\fg$ and equip it with an even, symmetric, $\fg$-invariant bilinear form $h$.
With this data, we can construct a degree $+1$ one-dimensional central extension 
\beqn
0 \to \C\kappa [-1] \to \Hat{\fg}_h \to \fg\langle\!\langle u \rangle\!\rangle \to 0 
\eeqn
of the graded Lie algebra $\fg\langle\!\langle u \rangle\!\rangle = \fg \otimes \CK^u$ defined by the formula
\be
	[A \otimes \alpha, B \otimes \beta] = [A,B] \otimes \alpha \beta - \kappa h(A,B) \textrm{Res}_u\diff u \alpha \pd_u \beta
\ee
where $\alpha, \beta \in \CK$ and $\kappa$ is the degree $+1$, central term. If $\fg$ has multiple factors, one could allow for multiple degree $1$ generators. For simplicity, we will focus on the case where there is only one.
This central extension is the raviolo analog of the affine Kac--Moody central extension of $\fg(\!(z)\!)$.
Notice that the cohomological shift of the central term arises from the residues.
We refer to $\Hat{\fg}_h$ as the raviolo Kac-Moody algebra, see \cite{FHK} for related algebras.

If we choose a homogeneous basis $\{T_a\}$ of $\fg$, with $h(T_a, T_b) = h_{ab}$ and commutators $[T_a, T_b] = f^c_{ab} T_c$, then we have the following basis for $\Hat{\fg}_h$:
\begin{itemize}
	\item The central element $\kappa$ is of cohomlogical degree $+1$ and spin $0$.
	\item For each $n \geq 0$ we have the cohomological degree zero generator $J_{a,n} = T_a \otimes u^n$ which has spin $-n$.
	\item For each $n \geq 0$ we have the cohomological degree $+1$ generator $\mu_{a,n} = T_a \otimes \Omega^n_u$ which has spin $n+1$.
\end{itemize}
The generator $J_{a,n}$ (resp. $\mu_{a,n}$) is bosonic (resp. fermionic) if $T_a \in \fg_+$ and fermionic (resp. bosonic) if $T_a \in \fg_-$; the generator $\kappa$ is always fermionic. These generators have the following commutators:
\be
\begin{aligned}[]
	[J_{a,m}, J_{b,n}] = f_{ab}^c J_{c, m+n} \hspace{2cm} [\mu_{a,m}, \mu_{b,n}] = 0\\
	[J_{a,m}, \mu_{b,n}] = \begin{cases} 
		0 & n+1 < m\\
		m h_{ab} \kappa & n+1 = m\\
		f^c_{ab} \mu_{c,n-m} & n \geq m
	\end{cases} \hspace{1cm}
\end{aligned}
\ee
The action of $-\pd_u$ on $\CK$ implies the following action of the translation operator on the generators:
\beqn
	\pd \mu_{a,n} = (n+1) \mu_{a,n+1} \qquad \pd J_{a,n} = -n J_{a,n-1} \qquad \pd \kappa = 0
\eeqn


Next, we define the analog of the vacuum module associated to the raviolo Kac-Moody algebra. As we will see, this is most naturally a module over $\C[\kappa]= \C \oplus \C \kappa[-1]$.
We view $\Hat{\fg}_{\geq0} = \fg[\![u]\!] \oplus \C\kappa [-1]$ as the positive subalgebra of $\Hat{\fg}_h$ and define the vacuum module by
\be
	\CV[\fg, h]_{univ} = U \Hat{\fg}_h \otimes_{U\Hat{\fg}_{\geq0}} \C[\kappa].
\ee
where $\fg[\![u]\!]$ acts trivially on $\C[\kappa]$ and $\kappa$ by multiplication by $\kappa$.
As a $\C[\kappa]$-module, $\CV [\fg,h]_{univ}$ can be identified with the space of $\C[\kappa]$-linear polynomials in the variables $\mu_{a,n}$. 
The vacuum vector $|0\rangle$ will be the image of $1 \otimes 1$ in the quotient. The translation operator acts as above and gets extended to $\CV[\fg,h]_{univ}$ by the Leibniz rule.

The state-operator correspondence is directly analogous to the previous case, so we do not spell it out in detail. Instead, we describe the generating fields. The linear states $\mu_{a,0}|0\rangle$ will correspond to the fields
\beqn
	\mu_a(z) = \sum_{n \geq 0} z^{n} \mu_{a,n} + \Omega^n_z J_{a,n}
\eeqn
It is a simple task to verify that the above data satisfies the conditions of Proposition \ref{prop:reconstruction}. For example, the commutators of the fields $\mu_a(z)$ are given by
\be
	[\mu_a(z), \mu_b(w)] = \pd_w\Delta(z,w) h_{ab} \kappa + \Delta(z,w) f^c_{ab}\mu_c(w)
\ee
which establishes their mutual locality. We see that the singular terms in the OPEs of the generating fields are given by
\be
	\mu_a(z) \mu_b(w) \sim \Omega_{z-w}^1 h_{ab} \kappa + \Omega_{z-w}^0 f^c_{ab}\mu_c(w)
\ee

\begin{prop}
$\CV[\fg,h]_{univ}$ is a raviolo vertex algebra over $\C[\kappa]$ with spin grading and super grading that is strongly generated over $\C[\kappa]$ by the fields $\mu_a(z)$.
\end{prop}

We note that $\CV[\fg,h]_{univ}$, viewed as merely a raviolo vertex algebra, isn't strongly generated by the fields $\mu_a(z)$ because states proportional to $\kappa$ cannot be realized by acting on the vacuum with the $\mu_{a,n}$. Instead, $\CV[\fg,h]_{univ}$ is strongly generated by the $\mu_a(z)$ and the constant field $\kappa = Y(\kappa|0\rangle,z)$.

The only non-trivial ideal of $\C[\kappa]$ is the (maximal) ideal $(\kappa)$ and the resulting quotient is given by
\beqn
	\CV[\fg] = U \Hat{\fg}_h \otimes_{U\Hat{\fg}_{\geq0}}\C
\eeqn
where $\C$ is the trivial 1-dimensional representation of $\Hat{\fg}_{\geq0}$. We note that this quotient is independent of the bilinear form $h$ used to define the central extension $\widehat{\fg}_h$. This raviolo vertex algebra is strongly generated by fields $\mu_a(z)$ with OPEs
\beqn
	\mu_a(z) \mu_b(w) \sim \Omega^0_{z-w} f_{ab}^c\mu_c(w)
\eeqn

With the notion of a raviolo current algebra, we can formulate a notion of raviolo vertex algebras with symmetries. This should capture when a raviolo vertex algebra $\CV$ admits a non-trivial morphism from $\CV[\fg,h]_{univ}$. 

\begin{dfn}
	Let $\CV$ be a raviolo vertex algebra over $\C[\kappa]$. We say $\CV$ has a \defterm{Hamiltonian $\fg$ symmetry at level $h$} if it equipped with a non-trivial morphism $\CV[\fg,h]_{univ} \to \CV$ of raviolo vertex algebras over $\C[\kappa]$.
\end{dfn}

If the level $h$ vanishes or if $\CV$ is a trivial $\C[\kappa]$ module, the morphism $\CV[\fg,h]_{univ} \to \CV$ factors through the quotient $\CV[\fg] \to \CV$. By an abuse of notation, we often denote by $\mu_a(z)$ the fields on $\CV$ corresponding to the images of $\mu_{a,0}|0\rangle$ in $\CV$. For $\zeta = \zeta^a T_a \in \fg$ we denote $\mu_\zeta(z) = \zeta^a \mu_a(z)$. The modes $\mu_{\zeta,(0)}$ give $\CV$ the structure of a representation for $\fg$ via raviolo vertex algebra derivations.
As a representation we can consider its character.

As usual, we denote by $q$ the generator for the $U(1)$ action determining the spin grading.
Let $W$ be the Weyl group of $G$ and $T \subset G$ a Cartan.
For any weight $w$ of $G$ let $s^w \colon T \to \C^\times$ be the corresponding one-dimensional $T$-representation.
In particular, for a root $\alpha$ we have the representation $s^\alpha$.
If we choose generators $\{s_i\}_{i=1}^{rk(\lie{g})}$ for the Cartan $T$ then $s^\alpha = \prod_{i=1}^{rk (g)} s_i^{\alpha_i}$.
The character of a raviolo vertex algebra with a $\lie{g}$ symmetry is defined as the graded dimension with respect to both the spin grading and the grading determined by the Cartan of $\lie{g}$.
For the raviolo Kac--Moody algebra $\CV[\lie{g}]$ one has
\beqn
\text{ch}_{\CV[\lie{g}]} (q, \{s_i\}) = \prod_{\alpha \in \text{rt}(\lie{g})} (q s^\alpha; q)_\infty
\eeqn
where $\text{rt}(\lie{g})$ is the set of roots of $\lie{g}$.


%


\begin{dfn}
	Let $\CV$ be a raviolo vertex algebra with a Hamiltonian $\fg$ symmetry at level $h$ (possibly zero). A vector $a \in \CV^{(s)}$ is called a \defterm{$\fg$ primary state} if for every $\zeta \in \fg$
	\begin{equation*}
		\mu_{\zeta,(n+1)} a = 0\,, n > 0
	\end{equation*}
	The corresponding field $Y(a,z)$ is called a \defterm{$\fg$ primary field}.
\end{dfn}
The primary constraint on $a$ implies the follow OPE:
\be
	\mu_\zeta(z) Y(a,w) \sim \Omega^0_{z-w} Y(\zeta \cdot a,w)
\ee
where $\zeta \cdot a = \mu_{\zeta,(0)} a$.

From the perspective of the three-dimensional QFT, the operators $\mu_{a,(-1)}|0\rangle$ giving rise to the fields $\mu_a(z)$ generate a $\fg[\![z]\!]$ symmetry of the theory via the descent brackets $\{\{\mu_a, -\}\}^{(n)} \leftrightarrow J_{a,n}$. From this perspective, $\kappa^\CV h$ has a clear meaning. If $\kappa^\CV h$ is non-trivial, this $\fg[\![z]\!]$ symmetry suffers from an anomaly: $\mu_a$ and its derivatives generate the $\fg[\![z]\!]$ symmetry but do not transform as expected, i.e. they are not primary operators transforming in the coadjoint representation $\fg^*$.

With $\fg$ primary fields in hand, the following useful property of their normal-ordered product is an immediate consequence of Corollary \ref{cor:NOPcomm}.
\begin{prop}
\label{prop:currentprimary}
	Let $Y(a_1,z), Y(a_2,w)$ be $\fg$ primary fields for a Hamiltonian $\fg$ symmetry transforming in representations $R_1$ and $R_2$, then their normal-ordered product $\norm{Y(a_1,z) Y(a_2,z)}$ is a primary field transforming in the representation $R_1 \otimes R_2$.
\end{prop}

This behavior is quite different from the case of vertex algebras, where the normal-ordered product of primaries need not be a primary. For example, in the VOA of a free, complex fermion $\psi(z) \chi(w) \sim (z-w)^{-1}$ the current $J = \norm{\psi \chi}$ is not a primary for itself due to the anomaly.

We note that the raviolo vertex algebra of $N$ free fields has a Hamiltonian $\fgl(N)$ symmetry:

\begin{prop}
	$FC^{\otimes N}$ has a Hamiltonian $\fgl(N)$ symmetry at level $0$ generated by $\norm{\psi_j X^i}$. Moreover, $X^i$ and $\psi_i$ are primaries, transforming in the standard representation $\C^N$ and its dual $(\C^N)^*$, respectively.
\end{prop}

This follows from an explicit computation that we do not show here. There is a Hamiltonian $\fg$ symmetry (at level $0$) for any $\fg \subseteq \fgl(N)$. We also note that the raviolo Heisenberg algebra $H$ has a Hamiltonian $\fgl(1)$ symmetry:

\begin{prop}
	$H$ has a Hamiltonian $\fgl(1)$ symmetry at level $0$ generated by $\nu(z)$. The field $b(z)$ is not a primary.
\end{prop}

\subsection{Conformal raviolo algebras}
\label{sec:virasoro}
The final basic example we describe is the raviolo vertex algebra analog of the Virasoro vertex algebra. As with the raviolo current algebra studied in the previous section, the raviolo analog of the central charge $\xi$ will have degree $1$. 
We will denote the resulting raviolo vertex algebra over $\C[\xi]= \C \oplus \C\xi[-1]$ by $\textrm{Vir}_{univ}$.

Geometrically, the Virasoro algebra is a central extension of the Lie algebra of vector field s on the punctured disk.
In the raviolo context, the correct Lie algebra to start with is the Lie algebra of vector fields on $\C \times \R - \{0\}$ which are compatible with the THF structure.
In Section \ref{sec:model} we denoted a derived algebraic model for this by $\CA^{(-1)}$ and its cohomology by $\CK_{poly}^{(-1)}$.
We denote its completed version by $\CK^{(-1)}$.
The Lie bracket of vector fields endows $\CA^{(-1)}$ with the structure of a dg Lie algebra and its cohomology $\CK^{(-1)}$ with the structure of a graded Lie algebra.

The degree zero part of the graded Lie algebra $\CK^{(-1)}$ can be identified with $\C[\![u]\!] \del_u$; we will denote generators by $G_m = -u^m \del_u$, for $m \geq 0$.
In degree one there are expressions of the form $\Gamma_n = \Omega^n \del_u$, for $n \geq 0$.
The spin gradings are such that $G_m$ is of spin $1-m$ and $\Gamma_n$ is of spin $n+2$.
The bracket in $\CK^{(-1)}$ is
\be\label{eqn:vir}
\begin{aligned}[]
	[G_m, G_n] = (m-n) G_{m+n-1} \hspace{2cm} [\Gamma_m, \Gamma_n] = 0\\
	[G_m, \Gamma_n] = \begin{cases} 
		0 & n+1 < m \\
		(m+n+1) \Gamma_{n-m+1} & n+1 \geq m\\
	\end{cases} \hspace{0.5cm}
\end{aligned}
\ee

There is a shifted central extension of $\CK^{(-1)}$ defined by the following cocycle
\beqn
(f \del_u, g \del_u) \mapsto - \tfrac{1}{12} \xi\, \text{Res}_u\diff u f \pd^3_u g .
\eeqn
Explicitly, this does not change the $[G_m,G_n], [\Gamma_m,\Gamma_n]$ brackets, but modifies the last bracket in \eqref{eqn:vir} to
\be
\begin{aligned}[]
	[G_m, \Gamma_n] = \begin{cases} 
		0 & n+3 < m\\
		\frac{m(m-1)(m-2)}{12}\xi & n+3 = m\\
		0 & n+2 = m\\
		(m+n+1) \Gamma_{n-m+1} & n+1 \geq m\\
	\end{cases} \hspace{0.5cm}
\end{aligned}
\ee
where we have introduced the degree $+1$, spin zero central generator $\xi$.
We call this the \emph{raviolo Virasoro algebra} $Vir$.

We can build a raviolo vertex algebra over $\C[\xi]$ via induction over the positive subalgebra $Vir_{\geq 0} = \C[\![u]\!] \del_u \oplus \C\xi[-1]$ as before:
\beqn
	\textrm{Vir}_{univ} = U Vir \otimes_{U Vir_{\geq0}} \C[\xi]
\eeqn
To avoid repetition, we note that the vacuum vector is again the image of $1 \otimes 1 \in U Vir \otimes \C[\xi]$ in the quotient and just describe the generating field $\Gamma$. This is a field of cohomological degree 1 and spin $2$ with the following OPE with itself:
\be
	\Gamma(z) \Gamma(w) \sim \Omega^3_{z-w} (\xi/2) + \Omega^1_{z-w} 2 \Gamma(w) + \Omega^0_{z-w} \pd_w \Gamma
\ee
This is the field corresponding to the image of $\Gamma_0 \otimes 1 \in UVir \otimes \C[\xi]$ in the quotient.

\begin{prop}
	$\textrm{Vir}_{univ}$ is a raviolo vertex algebra over $\C[\xi]$ with spin grading that is strongly generated over $\C[\xi]$ by the field $\Gamma(z)$.
\end{prop}

The quotient raviolo vertex algebra Vir associated to the maximal ideal $(\xi)$ of $\C[\xi]$ has underlying vector space 
\beqn
	\textrm{Vir} = U Vir \otimes_{U Vir_{\geq0}} \C .
\eeqn
It also has a spin grading and is strongly generated by a field $\Gamma(z)$ whose OPE with itself is
\be
	\Gamma(z) \Gamma(w) \sim \Omega^1_{z-w} 2 \Gamma(w) + \Omega^0_{z-w} \pd_w \Gamma
\ee
With respect to this grading, the character of the raviolo vertex algebra $\text{Vir}$ is
\beqn
\text{ch}_{\text{Vir}}(q) = (q^2;q)_\infty .
\eeqn

\begin{dfn}
	Let $\CV$ be a raviolo vertex algebra over $\C[\xi]$ with spin grading. $\CV$ is called \defterm{conformal (of central charge $\xi^\CV$)}, if we are given a non-zero \defterm{conformal vector} $\gamma \in \CV^{1,(2)}$ such that the modes of the \defterm{stress tensor} $\Gamma(z) = Y(\gamma, z)$ satisfy the relations of the raviolo Virasoro algebra with $\xi$ acting by $\xi^\CV$, and in addition we have $G_0 = \pd$ and $G_1|_{\CV^{(s)}} = s \id_{\CV^{(s)}}$.
\end{dfn}

As described in Section 1.1 of \cite{CostelloDimofteGaiotto-boundary}, it is expected that the algebra of local operators in the holomorphic-topological twist of any three-dimensional $\CN=2$ supersymmetric theory (with a $U(1)_R$ $R$-symmetry) has a ``higher stress tensor'' $\Gamma$ (that they denote $G$) so that the descent bracket $\{\{\Gamma,-\}\}^{(0)}$ generates $z$-translations, i.e. 
\beqn
\{\{\Gamma, O\}\}^{(0)} = \pd_z O .
\eeqn
The higher order brackets $\{\{\Gamma, -\}\}^{(n)}$ naturally realize the remaining holomorphic vector fields $z^n \pd_z$ and so we are lead to the expectation that the algebra of local operators in the twist of such a theory has the structure of a (possibly dg) conformal raviolo vertex algebra.

\begin{dfn}
	Let $\CV$ be a conformal raviolo vertex algebra. A vector $a \in \CV^{(s)}$ is called a \defterm{conformal primary (of spin $s$)} if 
	\begin{equation*}
		G_{n+1} a = 0\,, n > 0
	\end{equation*}
	The corresponding field $a(z) = Y(a,z)$ is called a \defterm{conformal primary field}.
\end{dfn}

We see that the OPE of the stress tensor $\Gamma$ with a conformal primary field $a(z)$ of spin $s$ takes the following form:
\be
	\Gamma(z) a(w) \sim s \Omega^1_{z-w} a(w) + \Omega^0_{z-w} \pd_w a(w)
\ee

As we saw with Hamiltonian $\fg$ symmetries, a non-vanishing $\xi^\CV$ implies the action of holomorphic coordinate transformations is anomalous, i.e. $\Gamma$ isn't itself a conformal primary (of spin $2$). This is directly analogous to the setting in vertex algebras, where the central charge accounts for a conformal anomaly of the two-dimensional theory, cf. e.g. Section 5.4.2 of \cite{DFMS}.

As with the above current algebra, Lemma \ref{cor:NOPcomm} implies the following useful property of the normal-ordered product of conformal primaries:
\begin{prop}
\label{prop:confprimary}
	Let $a_1, a_2$ be conformal primaries of spins $s_1, s_2$, then their normal-ordered product $\norm{a_1 a_2}$ is a conformal primary of spin $s_1 + s_2$.
\end{prop}

The following two propositions are straight-forward computations:
\begin{prop}
	The raviolo vertex algebra $FC^{(s)}$ is conformal with vanishing central charge for any choice of spin $s$ for the boson $X$; the stress tensor is given by
	\be
		\Gamma = (1-s)\norm{\psi \pd_z X} -s\norm{X \pd_z \psi}\,.
	\ee
	Moreover, $X$ and $\psi$ are conformal primaries.
\end{prop}

\begin{prop}
	The raviolo Heisenberg vertex algebra is conformal with vanishing central charge; the stress tensor is given by
	\be
		\Gamma = - \norm{b\nu}\,.
	\ee
	Moreover, $b$ and $\nu$ are conformal primaries.
\end{prop}
 
Suppose $\CV$ is a raviolo vertex algebra over $\C[\kappa, \xi]$ that is conformal with central charge $\xi^\CV$ and stress tensor $\Gamma(z)$ that also has a Hamiltonian $\fgl(1)$ symmetry at level $h \in \C$ generated by a field $\mu(z) = Y(\mu_{(-1)}|0\rangle,z)$. If we further assume $\mu$ is a conformal primary, then shifting the stress tensor $\Gamma \to \wt{\Gamma} = \Gamma - \pd_z \mu$ defines a new stress tensor of central charge $\xi^\CV - \tfrac{h}{12}\kappa^\CV$. If $O$ is simultaneously a conformal primary for $\Gamma$ of spin $s$ and a primary for this abelian current of weight $q$ then it will be a conformal primary of $\wt{\Gamma}$ of spin $s + q$. Thus, if we simultaneously modify the spin grading on $\CV$ by weights for $\mu_{(0)}$, we get another conformal structure on the same raviolo vertex algebra. For example, we can realize $FC^{(s)}$ from $FC$ in this way.

Note that neither $\CV[\fg,h]_{univ}$ nor the quotient $\CV[\fg]$ is conformal by itself. This is particularly clear in the abelian case: the field $\mu$ has degree 1 and therefore $\norm{\mu^2} = 0$. The only local operator with degree 1 and spin 2 is $\pd_z \mu$, but the OPE of $\pd_z \mu$ with itself doesn't match that of a stress tensor. This result is the raviolo analog of the situation in vertex algebras where the Sugawara construction fails in the universal affine vertex algebra over $\C[k]$ and at critical level $k = -h_\fg$, where $h_\fg$ is the dual Coexeter number for $\fg$, as we cannot invert $k+h_\fg$ in either case. In the context of raviolo vertex algebras, there are no non-critical levels.

\subsection{Nilpotent Deformations}
\label{sec:nilp}

A particularly important construction when describing the raviolo vertex algebras of local operators in three-dimensional holomorphic-topological quantum field theory is the possibility of deformations coming from holomorphic-topological descent.
For the connection between deformations of ordinary vertex algebras and two-dimensional conformal field theory we refer to~\cite{LiVertex}.

For any (bosonic) local operator $O$ of cohomological degree $2$ and spin $1$, we can consider deforming the action by the term $\int_{\C \times \R} O^{(2)}$, where $O^{(2)}$ is the second descendant of $O$; at the level of local operators, this should introduce a differential given by taking (or deform the existing differential by adding) the descent bracket with $O$, cf. the discussion in Section 3.4 of \cite{CostelloDimofteGaiotto-boundary}.

The requirement that $O$ has spin $1$ ensures this integral is (twisted) Lorentz invariant or, equivalently, that $\{\{O,-\}\}^{(0)}$ has spin $0$; the requirement that $O$ has cohomological degree $2$ ensures $\int_{\C \times \R} O^{(2)}$ is degree $0$ or, equivalently, that $\{\{O, -\}\}^{(0)}$ has degree $1$; finally, $O$ must be bosonic to ensure the action is bosonic or, equivalently, that the $\{\{O,-\}\}^{(0)}$ is a fermionic derivation. Note that it is not sensible to consider arbitrary $O$ of cohomological degree $2$ and spin $1$. For this deformation to define a differential, it must also square to zero%
\footnote{More generally, if there is already a differential then this $O$ must solve a suitable Maurer-Cartan equation.} %
and hence $\{\{O, O\}\}^{(0)} = 0$.

We now suppose $\CV$ is a raviolo vertex algebra. Choose a non-zero element $W \in \CV^{2}$ of cohomological degree 2. It is immediate that $D_W = W_{(0)}$ defines a fermionic derivation of $\CV$ with cohomological degree 1, cf. Corollary \ref{cor:derivation}. 
If, moreover, $D_W{}^2 = 0$ then we see that $D_W$ gives $\CV$ the structure of a dg raviolo vertex algebra. In terms of the OPE of $W$ with itself, we see that $D_W{}^2 = 0$ if and only if the coefficient of $\Omega^0_{z-w}\Omega^0_{w}$ in the OPE $W(z)W(w)$ vanishes. We will require the slightly stronger constraint: we require that the coefficient of $\Omega^0_{z-w}$ in the OPE $W(z) W(w)$ is a total derivative, i.e. it belongs to the image of $\pd_w$.%
\footnote{This stronger constraint is physically quite natural. For the twisted theories considered in \cite{CostelloDimofteGaiotto-boundary}, it is shown in Section 3.4 thereof that the descent bracket $\{\{-,-\}\}^{(0)}$ is related to the BV bracket. Our constraint translates to requiring that BV bracket of the Lagrangian density with itself is a total derivative, whence the action solves the classical master equation.} %
\begin{dfn}
	An element $W \in \CV^{2}$ is called a \defterm{superpotential} for $\CV$ if $W_{(0)}W_{(-1)}|0\rangle$ belongs to the image of $\pd$. We use the notation $(\CV, W)$, calling the pair a \defterm{raviolo vertex algebra with superpotential}.
\end{dfn}

If $\CV$ has a spin grading, we require the superpotential $W(z)$ has spin $1$ so that the derivation $D_W$ has spin $0$. Additionally, if $\CV$ has a super grading, we require $W(z)$ is even and hence bosonic; the derivation $D_W$ is thus even and fermionic.

The above analysis can be summarized as the following lemma.

\begin{lemma}
	Let $(\CV, W)$ be a raviolo vertex algebra with superpotential. The map
	\begin{equation*}
		D_W: a \mapsto W_{(0)} a
	\end{equation*}
	gives $\CV$ the structure of a dg raviolo vertex algebra.
\end{lemma}

An important class of examples comes from conformal raviolo vertex algebras equipped with a superpotential $W$ that is a conformal primary.

\begin{dfn}
	A raviolo vertex algebra with superpotential $(\CV, W)$ is called \defterm{conformal} if $\CV$ is conformal and $W$ is a conformal primary.
\end{dfn}

\begin{corollary}
	\label{cor:conformalsuperpotential}
	Let $(\CV,W)$ be a conformal raviolo vertex algebra with superpotential, then $D_W$ gives $\CV$ the structure of a conformal dg raviolo vertex algebra. Moreover, if $a \in \CV$ is a conformal primary of spin $s$ then so too is $D_W a$.
\end{corollary}

\begin{proof}
	The proof is a simple computation. The OPE $\Gamma W$ is given by
	\be
	\Gamma W \sim \Omega^1_{z-w} W + \Omega^0_{z-w} \pd_w W \qquad \rightsquigarrow \qquad W \Gamma \sim - \Omega^1_{z-w} W
	\ee
	because $W$ is a primary of weight 1; this implies the action of $D_W\Gamma =0$.
	
	Now suppose $a(z) = Y(a,z)$ is a conformal primary operator. It follows that the OPE of $\Gamma$ and $D_W a$ is given by
	\be
	\begin{aligned}
		\Gamma D_W a & = - D_W (\Gamma a)\\
		& \sim s \Omega^1_{z-w} D_W a + \Omega^0_{z-w} \pd_w D_W a
	\end{aligned}
	\ee
	as desired.
\end{proof}

\subsubsection{Example: chiral fields with a superpotential}
\label{sec:chiralsuperpotential}
As an example of a raviolo vertex algebra with a superpotential, we consider the conformal raviolo vertex algebra of $N$ free fields 
\beqn
\CV = \bigotimes_{i=1}^N FC^{(r_i/2)}_{r_i} .
\eeqn
We denote the generating fields by $X^i(z), \psi_i(z)$ which have cohomological degrees $r_i, 1-r_i$ and spins $\frac{r_i}{2}, 1-\frac{r_i}{2}$, respectively. 
The stress tensor is
\be
\Gamma = \sum_{i=1}^N \big(1-\tfrac{r_i}{2}\big)\norm{\psi_i \pd X^i} - \tfrac{r_i}{2}\norm{X^i \pd \psi_i}
\ee

Thinking of $X^i_{(-1)}$ as coordinate functions on $\C^N$, we choose a holomorphic function $W: \C^N \to \C$ that is quasihomogeneous of weight 2 and promote it to the field $W(z) = Y(W(X^i_{(-1)})|0\rangle,z)$. The requirement that $W$ is quasihomogeneous implies $W(z)$ is a field of cohomological degree 2 and spin 1, i.e. $W(X^i_{(-1)})|0\rangle \in \CV^{2,(1)}$. Moreover, $W(z)$ has a regular OPE with itself because the $X^i$ have regular OPEs with each other, thus $W(z)$ defines a superpotential. Finally, because the $X^i$ are conformal primaries and $W$ is a normal-ordered product thereof, $W$ is a conformal primary.

\begin{prop}
	The choice of quasi-homogeneous function $W:\C^N \to \C$ equips $\CV = \bigotimes_i FC^{(r_i/2)}_{r_i}$ with the structure a conformal raviolo vertex algebra with superpotential.
\end{prop}

The OPEs of $W(z)$ with the generating fields $X^i(z)$ and $\psi_i(z)$ are particularly simple:
\be
	W(z) X^i(w) \sim 0 \qquad W(z) \psi_i(w) \sim \Omega^0_{z-w} \pd_iW(w)
\ee
where $\pd_i W(z)$ is the field corresponding to $\frac{\pd W(X^i_{(-1)})}{\pd X^i_{(-1)}}|0\rangle$. The action of the differential $D_W$ encodes the first-order pole in the OPE with $W(z)$, thus we find the following action on the generating fields
\be
	D_W X^i(z) = 0 \qquad D_W \psi_i(z) = \pd_i W(z)
\ee

This is the direct analog of the dg vertex algebra described in Section 5.1 of \cite{CostelloDimofteGaiotto-boundary} modeling the algebra of local operators in the twist of a three-dimensional $\CN=2$ theory of $N$ chiral superfields coupled by the superpotential $W:\C^N \to \C$. Quasihomogeneity of the superpotential is required to ensure that the theory admits the $U(1)_R$ $R$-symmetry used for twisting.

\subsubsection{Example: perturbative gauge theory, BRST reduction}
\label{sec:pertgaugetheory}

Another important example arises from holomorphic-topological perturbative gauge theory. 
Let $\CV$ be a raviolo vertex algebra with a Hamiltion $\fg$ symmerty at level 0 generated by fields $\mu_a$, $a = 1, ..., \dim \fg$. 

We introduce an additional $\dim \fg$ pairs of free fields $c^a, b_a$, where $c^a$ is a fermion of cohomological degree $1$ and spin $0$ and $b_a$ is a boson of cohomological degree $0$ and spin $1$.
These generate the free field raviolo vertex algebra 
\beqn
	\CV^{\lie{g}}_{bc} = (FC^{(1)})^{\otimes \dim \lie{g}} .
\eeqn

In the tensor product raviolo vertex algebra $\CV \otimes \CV^{\lie{g}}_{bc}$ we consider the field
\be
	W_\fg = \tfrac{1}{2} f^a_{bc} \norm{b_a c^b c^c} - \norm{c^a \mu_a} .
\ee
It is a simple computation to show that $W_\fg$ has a regular OPE with itself, hence defines a superpotential on $\CV \otimes \CV^{\lie{g}}_{bc}$. 
The action of the differential $D_{W_\fg}$ takes the following form:
\be
\begin{array}{c}
	D_{W_\fg} c^a(z) = \tfrac{1}{2}f^a_{bc}\norm{c^b c^c}(z) \qquad D_{W_\fg} b_a(z) = f^c_{ab}\norm{b_c c^b}(z) - \mu_a(z)\\
	D_{W_\fg} O(z) = \norm{c^a(z) \big(T_a O(z)\big)}
\end{array}
\ee
where $O \in \CV$ and, as above, $T_a O$ denotes the action of $T_a \in \fg$ on $O$. We call the resulting dg raviolo vertex algebra the \emph{BRST reduction of $\CV$ by $\fg$} and denote it $\CV/\!\!\!/\fg$.

We can modify this construction in two ways. First, $\CV$ could itself be equipped with a superpotential $W$. So long as the OPE of $\mu_a$ and $W$ is regular, i.e. the superpotential is a $\fg$-invariant primary, it is easy to check that $W + W_\fg$ has a regular OPE with itself and hence defines a superpotential on $\CV \otimes \CV^{\lie{g}}_{bc}$.

Now suppose we are given a (possibly degenerate) symmetric, $\fg$-invariant bilinear form $K_{ab}$%
\footnote{Physically, $K$ corresponds to the Chern-Simons level of the gauge fields, divided by $2\pi$. Because we are working perturbatively, we need not impose quantization conditions on these levels.} %
on $\fg$; a second deformation comes from considering the field
\be
	W_{\fg,K} = \tfrac{1}{2} f^a_{bc} \norm{b_a c^b c^c} + \tfrac{1}{2}K_{ab} \norm{c^a \pd_z c^b} - \norm{c^a \mu_a}
\ee
The OPE of $W_{\fg,K}$ with itself is \emph{not} regular, but the coefficient of $\Omega^0_{z-w}$ is a total derivative due to the $\fg$-invariance of $K_{ab}$ implying $K_{ad}f^{d}_{bc}$ is totally antisymmetric in $abc$:
\be
	W_{\fg, K} W_{\fg,K} \sim \Omega^0_{z-w} \bigg(\tfrac{1}{4}K_{ad}f^{d}_{bc}\norm{c^a c^b \pd_w c^c}\bigg) = \Omega^0_{z-w} \pd_w \bigg(\tfrac{1}{12}K_{ad}f^{d}_{bc}\norm{c^a c^b c^c}\bigg)
\ee
It follows that $W_{\fg, K}$ is once again a superpotential.

We can combine these two constructions together to get the following result.
\begin{theorem}
	\label{thm:pertgauging}
	Let $\CV$ be a raviolo vertex algebra with a Hamiltionian $\fg$ symmetry at level $0$ generated by fields $\mu_a$. Choose $K_{ab}$ a (possibly degenerate) symmetric, $\fg$-invariant bilinear form on $\fg$ and $W$ a $\fg$-invariant primary superpotential on $\CV$, then
	\begin{equation*}
		W_{tot} = \tfrac{1}{2} f^a_{bc} \norm{b_a c^b c^c} + \tfrac{1}{2}K_{ab} \norm{c^a \pd c^b} - \norm{c^a \mu_a} + W
	\end{equation*}
	is a superpotential on $\CV \otimes \CV^{\lie{g}}_{bc}$. Moreover, if $(\CV,W)$ is a conformal raviolo vertex algebra with superpotential and $\Gamma$ is its stress tensor, then $(\CV \otimes \CV^{\lie{g}}_{bc}, W_{tot})$ is a conformal raviolo vertex algebra with superpotential, where the stress tensor is given by
	\begin{equation*}
		\Gamma_{tot} = \Gamma - \norm{b_a \pd_z c^a}
	\end{equation*}
\end{theorem}

If $W: \C^N \to \C$ is a holomorphic function invariant under a (linear) $\fg$ action on $\C^N$, then the conformal raviolo vertex algebra with superpotential described in Section \ref{sec:chiralsuperpotential} satisfies the necessary conditions for this Theorem to hold. The differential $D_{tot}$ on the generating fields $c^a, b_a, X^i, \psi_i$ takes the form
\be
\begin{aligned}
	D_{tot} c^a & = \tfrac{1}{2}f^a_{bc} \norm{c^b c^c} \qquad & D_{tot} b_a & = f^c_{ab}\norm{b_c c^b} + K_{ab}\pd c^b - \mu_a\\
	D_{tot} X^i & = (\rho_a)^i{}_j \norm{c^a X^j} \qquad & D_{tot} \psi_i & = -(\rho_a)^j{}_i \norm{c^a \psi_j} +\pd_i W\\
\end{aligned}
\ee
cf. Eq. (3.34)  of \cite{CostelloDimofteGaiotto-boundary}. The resulting conformal dg raviolo vertex algebra is thus a model for the algebra of local operators in the (perturbatively) gauged theory. We hope to understand non-perturbative corrections in the future.

\section{Modules for raviolo vertex algebras}
\label{sec:modules}

We now turn to the notion of a module for a raviolo vertex algebra $\CV$. Our definition is a direct translation from the theory of vertex algebras, cf. Section 5.1.1 of \cite{FBZ}:
\begin{dfn}
	Let $\CV$ be a raviolo vertex algebra. 
	A graded vector space $M$ is called a \defterm{$\CV$-module} if it is equipped with an operation $Y_M \colon \CV \to \End(M)\otimes \CK_{dist}$ that assigns to each $a \in \CV$ a field
	\begin{equation*}
		Y_M(a,z) = \sum_{n<0} z^{-n-1} a^M_{(n)} + \sum_{n \geq0} \Omega^n_z a^M_{(n)}
	\end{equation*}
	on $M$ subject to the following axioms:
	\begin{itemize}
		\item $Y_M(|0\rangle, z) = \id_M$\\
		\item for each $a,b \in \CV$ and $m \in M$ the three expressions \begin{equation*}
			\begin{aligned}
				Y_M(a,z) Y_M(b,w) m & \in M\langle\!\langle z \rangle\!\rangle \langle\!\langle w \rangle\!\rangle\,,\\
				(-1)^{|a||b|} Y_M(b,w) Y_M(a,z) m & \in M\langle\!\langle w \rangle\!\rangle\langle\!\langle z \rangle\!\rangle\,, \text{ and}\\
				Y_M(Y(a,z-w)b,w)  m & \in M\langle\!\langle w \rangle\!\rangle\langle\!\langle z-w \rangle\!\rangle\\
			\end{aligned}
		\end{equation*}
		are expansions, in their respective domains, of the same element of $M \otimes \three$.
	\end{itemize}
	
	We say that $M' \subset M$ is a \defterm{submodule} of $M$ if $m \in M'$ then $Y_M(a,z)m \in M'\langle\!\langle z \rangle\!\rangle$ for every $a\in \CV$. A module $M$ is called \defterm{simple} if it has no non-trivial submodules.
	
	A \defterm{morphism} of $\CV$-modules $f: (M, Y_M) \to (N, Y_N)$ is the data of a linear map $f: M \to N$ intertwining the structure maps
	\begin{equation*}
		f(Y_{M}(a,z) m) = Y_{N}(a,z) f(m)
	\end{equation*}
	
	If $\CV$ has a spin grading and/or $\Z/2$ super grading, then we say that $M$ is \defterm{graded} if $M$ has a $\C$ (spin) and/or $\Z/2$ (super) grading such that whenever $a$ is homogeneous the field $Y_M(a,z)$ is homogeneous with the same gradings as $a$. A morphism $f: (M, Y_M) \to (N, Y_N)$ is \defterm{homogeneous} if $f: M \to N$ is homogeneous.
\end{dfn}

As with vertex algebras, the vector space $\CV$ underlying a raviolo vertex algebra is naturally a $\CV$-module, called the \emph{vacuum module}. Similarly, given a morphism $\CW \to \CV$ it follows that any $\CV$-module $M$ has the structure of a $\CW$-module, thereby giving us a functor from the category of $\CV$-modules to the category of $\CW$-modules.

When $\CV$ is a conformal raviolo vertex algebra, it is natural to restrict the class of modules one allows. In particular, if $\gamma \in \CV$ is the conformal vector, then it is natural to constrain the action of the modes of the field $Y_M(\gamma, z)$.

\begin{dfn}
	Let $\CV$ be a conformal raviolo vertex algebra with conformal vector $\gamma$. A $\CV$-module $M$ is called a \defterm{grading restricted generalized module} if the spin grading on $M$ induced by generalized eigenspaces $M^{(s)}$ of $\gamma_{(1)}$ satisfies
	\begin{itemize}
		\item $M^{(s)}$ is finite dimensional
		\item $M^{(s+n)} = 0$ for $n$ sufficiently negative
	\end{itemize}
	$M$ is called \defterm{conformal} if $\gamma_{(1)}$ acts in a semisimple way.
\end{dfn}

Our first result concerning modules is an analog of Proposition 5.1.2 of \cite{FBZ}.

\begin{prop}
	Let $\CV$ be a raviolo vertex algebra and $M$ a $\CV$-module. Then for every $a \in \CV$
	\begin{itemize}
		\item[1)] $Y_M(\pd a, z) = \pd_z Y_M(a,z)$
		\item[2)] All fields $Y_M(a,z)$ are mutually local.
	\end{itemize}
\end{prop}

\begin{proof}
	To prove $1)$, we consider the second axiom for $b = |0\rangle$. For any $m \in M$ and $a \in \CV$ we have
	\beqn
	Y_M(Y(a,z-w)|0\rangle, w) = \sum_{n\geq0} \frac{(z-w)^n}{n!} \pd^n_w Y(a,w)m
	\eeqn
	Assertion $1)$ follows by comparing the coefficient of $(z-w)$, together with
	\beqn
	Y(a,z-w)|0\rangle = \sum {n\geq0} = (z-w)^n a_{(-n)} |0\rangle
	\eeqn
	and $a_{(-2)}|0\rangle = \pd a$.
	
	The second assertion is also straightforward. To show that $Y_M(a,z)$ and $Y_M(b,w)$ are mutually local for all $a,b \in \CV$, we apply their commutator to a general element $m \in M$. The second axiom of a module, together with the associativity proprty of $\CV$ implies the following equality in $M \otimes \CK_{dist}^{z,w}$
	\beqn
	[Y_M(a,z), Y_M(b,w)] m = \sum_{n \geq 0} \tfrac{1}{n!}\pd^n_w \Delta(z-w) Y_M(a_{(n)}b,w) m
	\eeqn
	As this holds independent of $m$, we deduce the following equality in $\End(M) \otimes \CK^{z,w}_{dist}$
	\beqn
	[Y_M(a,z), Y_M(b,w)] = \sum_{n \geq 0} \tfrac{1}{n!}\pd^n_w \Delta(z-w) Y_M(a_{(n)}b,w)
	\eeqn
	As $a_{(n)}b = 0$ for $n \gg 0$, we conclude $Y_M(a,z)$ and $Y_M(b,w)$ are mutually local.
\end{proof}

\subsection{Raviolo vertex algebras induced from Lie algebras}

As with vertex algebras, when $\CV$ is defined as an induced module for a a graded Lie algebra $\mathfrak{G}$, such the first  examples given in Section \ref{sec:examples}, we can construct $\CV$-modules from certain $\mathfrak{G}$ modules by using the $\mathfrak{G}$ action to define an action of the generating fields. As the generators of $\CV$ must give fields on any module $M$, it follows that the action of $\mathfrak{G}$ must be \emph{smooth}. For concreteness, we will compare $\mathcal{V}[\fg,h]_{univ}$-modules and modules for the shifted central extension $\Hat{\fg}_h$ of $\fg\langle\!\langle u \rangle\!\rangle$. The following considerations apply equally well for the other universal raviolo vertex algebras $FC_{univ}$, $H_{univ}$, and $\textrm{Vir}_{univ}$ and their corresponding Lie algebras $\Hat{\CF}$, $\Hat{\CH}$, and $Vir$.

\begin{dfn}
	A $\Hat{\fg}_h$ module $M$ is \defterm{smooth} if for every $m\in M$ there exists $N \geq 0$ such that $J_{a,n} m = 0$ for $n \geq N$.
\end{dfn}

This is necessary if we want the above to define a $\CV[\fg,h]_{univ}$-module, as $Y_M(\mu_{a,0}|0\rangle,z)$ must define a field. We now show that smooth $\Hat{\fg}_h$-modules are equivalent to $\CV[\fg,h]_{univ}$-modules.

\begin{prop}
	The category of $\CV[\fg,h]_{univ}$-modules is equivalent to the category of smooth $\Hat{\fg}_h$-modules.
\end{prop}

\begin{proof}
	First let $M$ be a $\CV[\fg,h]_{univ}$ module. The fields $\mu_a^M(z)= Y_M(\mu_{a,0}|0\rangle,z)$ have the following commutators
	\begin{equation}
		[\mu_a^M(z), \mu_b^M(w)] = \pd_w \Delta(z-w) h_{ab} \kappa + \Delta(z-w) f^{c}_{ab} \mu_c^M(w)
	\end{equation}
	due to the second axiom defining a $\CV[\fg,h]_{univ}$-module. In particular, the modes of $\mu_a^M(z)$ give $M$ the structure of an $\Hat{\fg}_h$-module; $M$ is smooth because the $\mu_a^M(z)$ are fields on $M$.
	
	Now suppose $M$ is a smooth $\Hat{\fg}_h$ module. We first set $Y_M(|0\rangle,z) = \id_M$, $Y_M(\kappa|0\rangle,z) = \kappa$, and define fields $\mu_a^M(z)$, corresponding to $Y_M(\mu_{a,0}|0\rangle,z)$, in the obvious way; that these are fields on $M$ follows by smoothness. The above commutator still holds due to the fact that $M$ is a $\Hat{\fg}_h$ module. In particular, the $\mu_a^M(z)$ define mutually local fields on $M$. If we define $Y_M(\mu_{a,n}\mu_{b,0}|0\rangle, z)$ by the formula
	\beqn
		\sum_{n\geq0} (z-w)^n Y_M(\mu_{a,n}\mu_{b,0}|0\rangle, w) = \norm{\mu_a^M(z) \mu_b^M(w)}
	\eeqn
	then the above commutator translates to
	\beqn
	\begin{aligned}
		\mu_a^M(z) \mu_b^M(w) & = \pd_w \Delta_-(z-w) h_{ab} \kappa + \Delta_-(z-w) f^{c}_{ab} \mu_c^M(w)\\
		& \qquad + \sum_{n\geq0} (z-w)^n Y_M(\mu_{a,n}\mu_{b,0}|0\rangle, w)
	\end{aligned}
	\eeqn
	and
	\beqn
	\begin{aligned}
		(-1)^{(|a|+1)(|b|+1)}\mu_b^M(w) \mu_a^M(z) & = -\pd_w \Delta_+(z-w) h_{ab} \kappa - \Delta_+(z-w) f^{c}_{ab} \mu_c^M(w)\\
		& \qquad + \sum_{n\geq0} (z-w)^n Y_M(\mu_{a,n}\mu_{b,0}|0\rangle, w)
	\end{aligned}
	\eeqn
	Putting these together, we see that 
	\beqn
	\begin{aligned}
		& Y_M(\mu_{a,0}|0\rangle,z) Y_M(\mu_{b,0}|0\rangle,w) m\,,\\
		& (-1)^{(|a|+1)(|b|+1)}Y_M(\mu_{b,0}|0\rangle,w) Y_M(\mu_{a,0}|0\rangle,z) m\,, \text{ and}\\
		& Y_M(Y(\mu_{a,0}|0\rangle,z-w)\mu_{b,0}|0\rangle,w)m
	\end{aligned}
	\eeqn
	are all expansions of the same element of $M \otimes \three$, namely
	\beqn
		\Omega^1_{z-w} h_{ab} \kappa m + \Omega^1_{z-w} f^{c}_{ab} \mu_c^M(w)m + \sum_{n\geq0} (z-w)^n Y_M(\mu_{a,n}\mu_{b,0}|0\rangle, w)m
	\eeqn
	More generally, we define
	\beqn
		Y_M(\mu_{a_1, j_1} \dots \mu_{a_l, j_l}|0\rangle, z) = \frac{1}{j_1! \dots j_l!} \norm{\pd^{j_1}_z \mu_{a_1}(z) \dots \pd^{j_l}_z \mu_{a_l}(z)}
	\eeqn
	and
	\beqn
		Y_M(\kappa\mu_{a_1, j_1} \dots \mu_{a_l, j_l}|0\rangle, z) = \frac{1}{j_1! \dots j_l!}\kappa\norm{\pd^{j_1}_z \mu_{a_1}(z) \dots \pd^{j_l}_z \mu_{a_l}(z)}
	\eeqn
	The fact that axiom 2 holds for general elements of $\CV[\fg,h]_{univ}$ follows by Corollary \ref{cor:NOPcomm} and induction on $l$.
\end{proof}

Unsurprisingly, this result for the universal algebras implies the following relation for their quotients.

\begin{corollary}
	The category of $\CV[\fg]$-modules is equivalent to the category of smooth $\fg\langle\!\langle u \rangle \!\rangle$-modules.
\end{corollary}

An analogous result holds for the quotients $\text{Vir}$, $H_k$, and $FC_k$, where one must restrict to smooth $\Hat{\CH}$ or $\Hat{\CF}$ modules where the central element $K$ is required to act as $k$ on the latter two examples.

A special class of modules for $\CV[\fg,h]_{univ}$ come by induction from modules of $\fg$. Let $R$ be a representation of $\fg$, we then consider the representation $R[\kappa]$ of $\fg[\![z]\!] \oplus \C\kappa[-1]$ on which $\fg[\![z]\!]$ acts through its quotient to $\fg$ and the central element acts as multiplication by $\kappa$; the vector space 
\be
	U\Hat{\fg}_h \otimes_{U\Hat{\fg}_{\geq0}} R[\kappa]
\ee
is a smooth $\Hat{\fg}_h$-module and thus admits the structure of a $\CV[\fg,h]_{univ}$ module. Viewing a raviolo vertex algebra with a Hamiltonian $\fg$-symmetry as an $\CV[\fg,h]_{univ}$ module, we see that the $\fg$-primary operators discussed in Section \ref{sec:currents}, together with the states obtained by normal-ordered products with the currents $\mu_a(z)$ and their derivatives, transform in $\CV[\fg,h]_{univ}$-submodules of $\CV$ of this form.

\subsection{Lattice raviolo vertex algebra}
As a final example, we consider a certain module for the raviolo Heisenberg algebra $H$ that itself has the structure of a raviolo vertex algebra. Our construction is totally analogous to the construction of a lattice vertex algebra, cf. Section 5.2 \cite{FBZ}, and serves as a model for the full, non-perturbative algebra of local operators in holomorphic-topological abelian $BF$ theory, which is equivalent to the twist of a free three-dimensional $\CN=2$ vector multiplet.

We start with a family of smooth modules for the Lie algebra $\Hat{\CH}$ built by induction over the subalgebra $\Hat{\CH}_{\geq0}$. We consider the module $\C_\lambda = \C_{1,\lambda}$ generated by a vector $|\lambda\rangle$ with the action of $\Hat{\CH}_{\geq 0}$ given by
\beqn
	c_{n+1} |\lambda\rangle = \varphi_n |\lambda\rangle = 0, n \geq 0, \qquad c_0 |\lambda\rangle = \lambda |\lambda\rangle, K|\lambda \rangle = |\lambda\rangle
\eeqn
We then consider the following Fock module:
\beqn
	\textrm{Fock}_\lambda = U \Hat{\CH} \otimes_{U \Hat{\CH}_{\geq 0}} \C_{\lambda}
\eeqn
This has the structure of an $H$-module, where the states $b_0$ and $\nu_{0}$ correspond to the fields
\beqn
	b(z) = Y_M(b_0, z) = \sum_{n<0} z^{-n-1} b_n + \sum_{n\geq0} \Omega^n_z \varphi_n
\eeqn
and
\beqn
	\nu(z) = Y_M(\nu_0, z) = \sum_{n<0} z^{-n-1} \nu_n + \sum_{n\geq0} \Omega^n_z c_n
\eeqn
It is immediate that $\textrm{Fock}_0$ is isomorphic to the vacuum module. Using the fact that the stress tensor is $\Gamma = -\norm{b \nu}$, we see that $|\lambda\rangle$ is annihilated by $\gamma_{(1)}$, i.e. it is an eigenvector with eigenvalue $0$, whence $\textrm{Fock}_\lambda$ is a conformal module for $H$. We now prove the following result, cf. Lemma 5.2.2 of \cite{FBZ}.

\begin{lemma}
	Any simple graded $H$-module with spin bounded from below is isomorphic to $\textrm{Fock}_\lambda$ for some $\lambda \in \C$.
\end{lemma}

We do not need to impose the stronger constraint that the $H$-module is conformal, merely that it has a compatible spin graded that is bounded from below. Note that the spin grading on $\textrm{Fock}_\lambda$ is uniquely determined by the spin $s_\lambda \in \C$ of $|\lambda\rangle$; a choice of conformal structure on $H$ distinguishes a particular spin grading.

\begin{proof}
	We start by noting that $\textrm{Fock}_\lambda$ is itself simple. As noted above, any homogeneous vector $m \in \textrm{Fock}_\lambda$ can be written as $P|\lambda\rangle$ for $P$ some polynomial homogeneous polynomial in the $b_n, \nu_n$. If $P$ isn't constant, there exists $n > 0$ with $\pd P/\pd b_n \neq 0$ and hence $c_n m \neq 0$. The polynomial $[c_n, P]$ is again homogeneous, but has strictly lower spin. Continuing in this fashion, we find a vector $c_{n_1} \dots c_{n_k} m \neq 0$ annihilated by all $c_{n}$ for $n >0$, i.e. $P' = [c_{n_1}, [\dots, [c_{n_k},P]]]$ is a homogeneous polynomial in $\nu_n$. If $P'$ isn't constant, there is some $l > 0$ with $\pd P'/\pd \nu_l \neq 0$. Continuing in this fashion, we find a vector $\varphi_{l_1} \dots \varphi_{l_j} c_{n_1} \dots c_{n_k}m$ annihilated by $c_{n+1}$ and $\varphi_n$ for all $n \geq 0$ and therefore
	\beqn
		\varphi_{l_1} \dots \varphi_{l_j} c_{n_1} \dots c_{n_k}m = \alpha |\lambda\rangle
	\eeqn
	for some nonzero $\alpha \in \C$. Due to the fact that $\textrm{Fock}_\lambda$ is generated by acting with $b_n, \nu_n$ on $|\lambda\rangle$, we conclude that any two elements of $\textrm{Fock}_\lambda$ are related by acting with an element of $U\Hat{\CH}$, whence $\textrm{Fock}_\lambda$ is simple.
	
	Now suppose $M$ is any simple graded $H$-module with grading bounded from below. We first note that $\varphi_0$ must act as zero due to $M$ being simple -- the image of $\varphi_0$ is a submodule of $M$ because $\varphi_0$ is central in $\Hat{\CH}$. Let $M_{min}$ be the subspace of minimal spin. It follows that $M_{min}$ is annihilated by $c_{n+1}$ and $\varphi_{n}$ for $n \geq 0$. Moreover, $\dim_\C M_{min} = 1$ because $M$ is simple. We conclude that $c_0$ acts via multiplication by a scalar $\lambda$ on $M$, hence $M \simeq \textrm{Fock}_\lambda$.
\end{proof}

Let $\CV_\Z$ be the conformal $H$-module $\bigoplus_{\mathfrak{m} \in \Z} \textrm{Fock}_\mathfrak{m}$.

\begin{prop}
	$\mathcal{V}_\Z$ has the structure of a conformal raviolo vertex algebra (of central charge $0$) such that $H = \textrm{Fock}_0$ is a conformal raviolo vertex subalgebra.
\end{prop}

\begin{proof}
	Using the fact that $\textrm{Fock}_\mathfrak{m}$ is generated by $|\mathfrak{m}\rangle$ by acting with the modes of $b(z), \nu(z)$, the raviolo reconstruction theorem \ref{prop:reconstruction} implies that it suffices to define fields $V_{\mathfrak{m}}(z) = Y(|\mathfrak{m}\rangle,z)$ for all $\mathfrak{m} \neq 0$ and verify they satisfy the necessary conditions.
	
	Using the prescribed action of $H$ on $\textrm{Fock}_{\mathfrak{m}}$, we necessarily have the following OPEs of the generators $b(z)$ and $\nu(z)$ with $V_\mathfrak{m}(w)$:
	\beqn
	\label{eq:bnuVOPE}
	\begin{aligned}
		b(z) V_\mathfrak{m}(w) & = \norm{b(z) V_\mathfrak{m}(w)}\\
		\nu(z) V_\mathfrak{m}(w) & = \mathfrak{m}\Omega^0_{z-w} V_\mathfrak{m}(w) + \norm{\nu(z) V_\mathfrak{m}(w)}
	\end{aligned}
	\eeqn
	By assumption, we are taking the stress tensor for $\CV_\Z$ to be given by $\Gamma = -\norm{b\nu}$; the action of its modes on $|\mathfrak{m}\rangle$ are given by
	\beqn
		\gamma_{(n)} |\mathfrak{m}\rangle = \begin{cases} 0 & n \neq 0\\ -\mathfrak{m} b_{(-1)}|\mathfrak{m}\rangle & n = 0 \end{cases}
	\eeqn
	hence $V_{\mathfrak{m}}(w)$ must be a conformal primary of spin $0$. Its OPE with the stress tensor $\Gamma(z)$ must therefore be given by
	\beqn
		\Gamma(z) V_\mathfrak{m}(w) = -\mathfrak{m} \Omega^0_{z-w} \norm{b V_\mathfrak{m}}(w) + \norm{\Gamma(z) V_\mathfrak{m}(w)}
	\eeqn
	where $\norm{b V_\mathfrak{m}}(z) = Y(b_{(-1)}|\mathfrak{m}\rangle,z)$. In particular, part 2) of Corollary \ref{lem:translation} implies we must further have
	\beqn
	\label{eq:Vderiv}
		\pd_z V_\mathfrak{m}(z) = -\mathfrak{m} \norm{b V_\mathfrak{m}}(z)
	\eeqn
	
	As $b$ and $\nu$ are spin $1$ and the $V_\mathfrak{m}$ are all spin $0$, it follows that $\CV_\Z$ has a spin grading with support only in the non-negative integers; the only states with spin $0$ are the vectors $|\mathfrak{m}\rangle$, all of which have different weight with respect to $\nu_{(0)} = c_0$. It immediately follows that
	\beqn
	\label{eq:Vmodes}
		V_{\mathfrak{m}_1, (n)}|\mathfrak{m}_2\rangle = \begin{cases} 0 & n \geq 0\\ k_{\mathfrak{m}_1, \mathfrak{m}_2}|\mathfrak{m}_1 + \mathfrak{m}_2\rangle & n = -1\\ \end{cases}
	\eeqn
	for some $k_{\mathfrak{m}_1, \mathfrak{m}_2} \in \C$. The vacuum axiom dictates $k_{\mathfrak{m}, 0} = 1$ for every $\mathfrak{m}$. Together with Eq. \eqref{eq:bnuVOPE} and Eq. \eqref{eq:Vderiv}, the action in Eq. \eqref{eq:Vmodes} uniquely characterizes the field $V_{\mathfrak{m}}(z)$ --- Eq. \eqref{eq:Vderiv} encodes the action of the remaining modes of $V_{\mathfrak{m}_1}(z)$ and the OPEs Eq. \eqref{eq:bnuVOPE} give us the necessary commutators to define the action of these modes on a general state in $H_{\mathfrak{m}_2}$.
	
	As in the theory of vertex algebras, we can write an explicit formula for the field $V_\mathfrak{m}(z)$ using the modes of $b$. We let $S_\mathfrak{m}$ denote the shift operator defined by
	\beqn
		S_{\mathfrak{m}_1} |\mathfrak{m}_2\rangle = k_{\mathfrak{m}_1, \mathfrak{m}_2} |\mathfrak{m}_1+\mathfrak{m}_2\rangle
	\eeqn
	and that it commutes with all of the $b_{(n)}$ as well as the $\nu_{(n)}$ with $n \neq 0$. $V_\mathfrak{m}(z)$ then takes the form
	\beqn
		V_\mathfrak{m}(z) = S_\mathfrak{m} \exp\bigg(\mathfrak{m}\sum_{n < 0} z^{-n} \frac{b_{(n)}}{n} - \mathfrak{m}\sum_{n > 0} \Omega^{n-1}_z \frac{b_{(n)}}{n} \bigg)
	\eeqn
	Note that this expression is homogeneous of cohomological degree 0.
	
	It is easy to verify this expression for $V_\mathfrak{m}$ satisfies the above properties. We find the commutators of $V_\mathfrak{m}$ with the generators $b(z)$ and $\nu(z)$ are given by
	\beqn
	\begin{aligned}[]
		[b(z), V_\mathfrak{m}(w)] & = 0\\
		[\nu(z), V_\mathfrak{m}(w)] & = \mathfrak{m}\Delta(z-w) V_\mathfrak{m}(w)
	\end{aligned}
	\eeqn
	from which the OPEs in Eq. \eqref{eq:bnuVOPE} follow immediately, cf. Proposition \ref{prop:locality}. Moreover, we see that $V_\mathfrak{m}(w)$ is mutually local with both $b(z)$ and $\nu(z)$. Taking the derivative with respect to $z$, we find
	\beqn
		\pd_z V_\mathfrak{m}(z) = -\mathfrak{m}\bigg(\sum_{n < 0} z^{-n-1} b_{(n)} + \sum_{n > 0} \Omega^{n}_z b_{(n)} \bigg)V_\mathfrak{m}(z)
	\eeqn
	As the $b_{(n)}$ all commute with one another, we conclude the right-hand side is precisely $-\mathfrak{m} \norm{b V_\mathfrak{m}}(z)$ and hence Eq. \eqref{eq:Vderiv} holds as well; the fact that $b_{(0)} = \varphi_0$ is absent from the sum is no issue as it acts as $0$ on all of $\CV_\Z$. Finally, the action of $V_{\mathfrak{m}_1}(z)$ on $|\mathfrak{m}_2\rangle$ is given by
	\beqn
		V_{\mathfrak{m}_1}|\mathfrak{m}_2\rangle = k_{\mathfrak{m}_1, \mathfrak{m}_2} \exp\bigg(\mathfrak{m}_1\sum_{n < 0} z^{-n} \frac{b_{(n)}}{n}\bigg)|\mathfrak{m}_1+\mathfrak{m}_2\rangle
	\eeqn
	from which Eq. \eqref{eq:Vmodes} is immediate.
	
	We are finally in a position to verify the conditions of Proposition \ref{prop:reconstruction}. The vacuum vector is $|0\rangle$, and the translation operator $\pd$ acts as above. We consider the fields $b(z)$, $\nu(z)$, and the $V_\mathfrak{m}(z)$ for all $\mathfrak{m} \neq 0$, corresponding to the vectors $b_{(-1)}|0\rangle$, $\nu_{(-1)}|0\rangle$, and $|\mathfrak{m}\rangle$, respectively. Conditions $1)$ and $4)$ are immediate from the above, and condition $2)$ holds for $b$ and $\nu$. To see that condition $2)$ holds for $V_\mathfrak{m}(z)$, we note that the above commutators imply that the commutator of $\Gamma$ and $V_\mathfrak{m}$ is
	\beqn
		[\Gamma(z), V_\mathfrak{m}(w)] = -\mathfrak{m} \Omega^0_{z-w} \norm{b V_\mathfrak{m}}(w)
	\eeqn
	and hence $[\pd, V_\mathfrak{m}(z)] = [\gamma_{(0)}, V_\mathfrak{m}(z)] = \pd_z V_\mathfrak{m}(z)$.
	
	 We are left with checking condition $3)$, i.e. the mutual locality of these generators. We already know $b$ and $\nu$ are mutually local with one another and with $V_\mathfrak{m}$, so it suffices to verify the mutual locality of $V_{\mathfrak{m}_1}$ and $V_{\mathfrak{m}_2}$. Using the fact that the modes $b_{(n)}$ commute with one another, we see that they are mutually local if and only if
	\beqn
		S_{\mathfrak{m}_1} S_{\mathfrak{m}_2} = S_{\mathfrak{m}_2} S_{\mathfrak{m}_1} \Leftrightarrow k_{\mathfrak{m}_1, \mathfrak{m}_2 + \mathfrak{m}_3}k_{\mathfrak{m}_2, \mathfrak{m}_3} = k_{\mathfrak{m}_2, \mathfrak{m}_1 + \mathfrak{m}_3}k_{\mathfrak{m}_1, \mathfrak{m}_3}
	\eeqn
	Thus, any choice of the $k_{\mathfrak{m}_1, \mathfrak{m}_2}$ satisfying this constraint, together with $k_{\mathfrak{m},0} =1$, gives $\CV_\Z$ the structure of a raviolo vertex algebra, cf. the proof of Proposition 5.2.5 in \cite{FBZ}. Our preferred solution will be to set $k_{\mathfrak{m}_1, \mathfrak{m}_2} = 1$ for all $\mathfrak{m}_1, \mathfrak{m}_2$.
\end{proof}

It is worth noting that $b(z) = \norm{V_{1} \pd_z V_{-1}}$. In particular, $\CV_\Z$ is strongly generated by the bosonic fields $V_\pm(z) = V_{\pm 1}(z)$ and fermionic field $\nu(z)$. As noted in Section \ref{sec:examples}, $H$ has a Hamiltonian $\fgl(1)$ symmetry (at level $0$) generated by $\nu(z)$ and hence so too does $\CV_\Z$. The fields $V_{\mathfrak{m}}(z)$ are primaries of weight $\mathfrak{m}$. Physically, this $\fgl(1)$ symmetry corresponds to the \emph{topological flavor symmetry} of an abelian 3d gauge theory; the weight $\mathfrak{m}$ is identified with \emph{monopole number} and the operators $V_{\mathfrak{m}}(z)$ are called \emph{monopole operators}.

We note that the above proof can be applied to the module 
\beqn
	\CV_{\Lambda} = \bigoplus_{\lambda \in \Lambda} \textrm{Fock}_\lambda
\eeqn
for $\Lambda$ any countable additive subgroup of $\C$. Similarly, we can start from $H^{\otimes N}$ and built a raviolo vertex algebra $\CV_\Lambda$ from any countable additive subgroup $\Lambda \subset \C^N$. Unlike the case of vertex algebras, we find that there are no integrality constraints on $\Lambda$.

Although they are different $H$-modules, the raviolo vertex algebra structure on $\CV_{\Lambda}$ only depends on $\Lambda$ up to overall rescaling --- rescaling the generators $b(z), \nu(z)$ as $\sigma b(z), \sigma^{-1} \nu(z)$ for $\sigma \in \C^\times$ is an automorphism of $H$ that sends $\textrm{Fock}_\lambda$ to $\textrm{Fock}_{\sigma \lambda}$ and yields the desired isomorphism $\CV_{\Lambda} \simeq \CV_{\sigma \Lambda}$.

For any subgroup $\Lambda' \subset \Lambda$ it follows that $\CV_{\Lambda'}$ is a subalgebra of $\CV_{\Lambda}$. In the case of $N = 1$, the smallest non-trivial additive subgroups of $\C$ take the form $\Lambda = \lambda \Z$ for some nonzero $\lambda \in \C$ and hence correspond to a raviolo vertex algebra isomorphic to $\CV_\Z$. We note that the graded character of $\CV_\Z$ is given by
\beqn
	\text{ch}_{\CV_\Z}(q,x) \define \sum_{s, \lambda} \op{grdim}(\CV_\Z^{(s),\lambda}) q^s x^\lambda = \sum_{\mathfrak{m}\in\Z} x^\mathfrak{m}
\eeqn
where $\CV_\Z^{(s),\lambda}$ is the subspace of $\CV_\Z$ of spin $s$ and $c_0$ weight $\lambda$.%
\footnote{The flavor grading of the Heisenberg raviolo vertex algebra used in Section \ref{sec:heisenberg} does not extend to the lattice vertex algebra.} %
This exactly matches the superconformal index of a free $\CN=2$ vector multiplet

\bibliography{rav}
\bibliographystyle{amsalpha}
\end{document}